\documentclass[twoside,11pt]{amsart}

\usepackage[cmtip,all]{xy}
\usepackage{xspace, enumerate, times, color}
\usepackage{amssymb, hyperref}

\usepackage{graphicx}
\input xypic
\input xy
\xyoption{all}
\def\sqr#1#2{{\vcenter{\hrule height.#2pt
        \hbox{\vrule width.#2pt height#1pt \kern#1pt
                \vrule width.#2pt}
        \hrule height.#2pt}}}
\def\square{\mathchoice\sqr64\sqr64\sqr{4}3\sqr{3}3}
\def\QED{\hfill$\square$}

\numberwithin{equation}{section}

\newtheorem{Theorem}{Theorem}[section]
\newtheorem{Lemma}[Theorem]{Lemma}

\newtheorem{Corollary}[Theorem]{Corollary}
\newtheorem{Proposition}[Theorem]{Proposition}

\newtheorem{Setting}[Theorem]{Setting}
\newtheorem{Open Problem}[Theorem]{Open Problem}
\newtheorem{Notation}[Theorem]{Notation}

\newtheorem{Remark}[Theorem]{Remark}
\newtheorem{Example}[Theorem]{Example}
\newtheorem{definition}[Theorem]{Definition}
\newtheorem{Question}[Theorem]{Question}
\newtheorem{Questions}[Theorem]{Questions}

\def\m{{\mathfrak m}}

\def\a{{\mathfrak a}}

\newcommand{\ol}[1]{\overline{#1}}

\newcommand{\ul}[1]{\underline{#1}}

\def\ZZ{{\mathbb Z}}
\def\NN{{\mathbb N}}

\newcommand{\F}{\mathcal{F}}

\newcommand{\G}{\mathbf{G}}
\newcommand{\PP}{\mathbb{P}}
\newcommand{\HH}{\mathbb{H}}

\newcommand{\Po}{\mathcal{P}}
\newcommand{\Pc}{\mathcal{P}_{\leq c}(m)}
\newcommand{\Gc}{\mathbf{G}_{c,(m)}}

\def\Llra{\Longleftrightarrow}

\def\Lra{\Longrightarrow}
\def\lra{\longrightarrow}
\def\Lla{\Longleftarrow}

\newcommand{\be}{\begin{equation*}}
\newcommand{\ee}{\end{equation*}}
\newcommand{\bee}{\begin{equation}}
\newcommand{\eee}{\end{equation}}

\def\h{{\rm ht}}
\def\grade{{\rm grade}}

\def\Ass{{\rm Ass}}

\def\sdef{{\rm sdefect}}
\def\supp{{\rm supp}}
\def\NS{{\rm Nsupp}}

\def\set{{\rm set}}
\def\revlex{{\rm revlex}}
\def\alex{{\rm alex}}

\def\Sdeg{{\rm Sdeg}}

\definecolor{lighterorange}{cmyk}{0,0.42,0.66,0.0}

\begin{document}

\date{today}

\title[The structure and free resolutions of symbolic powers of star configurations]{The structure and free resolutions of the symbolic powers of star configuration of hypersurfaces} 
\author[Paolo Mantero]{Paolo Mantero}

\address{University of Arkansas, Department of Mathematical Sciences, Fayetteville,
AR 72701}
\email{pmantero@uark.edu}
\thanks{}

\maketitle

\begin{abstract}
Star configurations of hypersurfaces are schemes in $\mathbb P^n$ widely generalizing star configurations of points. Their rich structure allows them to be studied using tools from algebraic geometry, combinatorics, commutative algebra and representation theory. 
In particular, there has been much interest in understanding how ``fattening" these schemes affects the algebraic properties of these configurations  or, in other words, understanding the symbolic powers $I^{(m)}$ of their defining ideals $I$.

In the present paper (1) we prove a structure theorem for $I^{(m)}$, giving an explicit description of a minimal generating set of $I^{(m)}$ (overall, and in each degree) which also yields a minimal generating set of the module $I^{(m)}/I^m$ -- which measures how far is $I^{(m)}$ from $I^m$. These results are new even for monomial star configurations or star configurations of points; (2)  
we introduce a notion of ideals with c.i. quotients, generalizing ideals with linear quotients, and show that $I^{(m)}$ have c.i. quotients. As a corollary we obtain that symbolic powers of ideals of star configurations of points  have linear quotients; (3) we find a general formula for all graded Betti numbers of $I^{(m)}$; (4) we prove that a little bit more than the bottom half of the Betti table of $I^{(m)}$  has a regular, almost hypnotic, pattern, and provide a simple closed formula for all these graded Betti numbers and the last irregular strand in the Betti table. 

Other applications include improving and widely extending results by Galetto, Geramita, Shin and Van Tuyl, and providing explicit new general formulas for the minimal number of generators and the symbolic defects of star configurations. 

Inspired by Young tableaux, we introduce a ``canonical" way of writing any monomial in any given set of polynomials, which may be of independent interest. We prove its existence and uniqueness under fairly general assumptions. Along the way, we exploit a connection between the minimal generators $\G_{(m)}$ of $I^{(m)}$ and positive solutions to Diophantine equations, and a  connection between $\G_{(m)}$ and partitions of $m$ via the canonical form of monomials. Our methods are characteristic--free. 
\end{abstract}

\maketitle


\bibliographystyle{amsplain}




\section{Introduction}
Star configuration of points in $\mathbb P_k^n$  have the {\em generic} Hilbert function, i.e. there exists a non-empty Zariski open subset of sets of points in $\mathbb P_k^n$ having their same Hilbert function; however, among all configurations with the generic Hilbert function, star configurations have a remarkable  tendency toward extremal numerical behaviors.

For this reason, and their rich combinatorial structure, in recent years star configurations of points have attracted a strong interest. As a few examples of their applications, they are frequently employed to prove the sharpness of bounds of numerical invariants of sets of points (e.g. \cite[Section~2.4]{BH}, \cite[Ex. 4.3]{MSS}), they play an important role in the proof of Chudnovsky's conjecture for any number of very general points in $\mathbb P_k^n$ (e.g. \cite[Thm.~2.8]{FMX}), and in the decomposition of (generic) hypersurfaces as sums of products of a fixed set of hypersurfaces (e.g. \cite{CGV}). 
See also \cite{GMS}, \cite{HHT}, \cite{BD+}, \cite{BH}, \cite{AS}, \cite{HH}, \cite{GHM}, \cite{CGV}, \cite{LM}, \cite{DSST}, \cite{PS}, \cite{BS}, \cite{BCK}, \cite{GHMN}, \cite{MSS}, \cite{CCGV}, \cite{TX} for a subset of the papers published in the last 15 years proving results or raising questions regarding star configurations of points.

In this paper we are interested in a far-reaching generalization,  introduced in progressively higher generality in the papers \cite{AS}, \cite{PS} and \cite{GHM}, dubbed {\em star configuration of hypersurfaces}. It allows the schemes to have any codimension (not just $n$), and be defined by any fixed set of forms, of any degrees (not necessarily linear),  as long as some reasonable ``intersection property" is met. Essentially, fix a set of $s$ hypersurfaces in $\mathbb P_k^n$, the {\em star configuration} of these hypersurfaces is the union $Z$ of all schemes obtained by intersecting $c$ of these hypersurfaces (see also Definition \ref{2.star}). 

Our objective is to provide a complete description of the structure and Betti table of the ``fattening" $mZ$ of the scheme $Z$. More precisely, we provide a structure theorem for the symbolic powers $I_Z^{(m)}$ and a formula for their graded Betti numbers.\\
\\
{\bf General motivating questions.} When $k$ is algebraically closed, for any equidimensional scheme  $X\subseteq\mathbb P_k^n$, a celebrated result by Zariski and Nagata identifies the symbolic power $I_X^{(m)}$ of the defining ideal $I_X$ as the ideal of all hypersurfaces in $\mathbb P_k^n$ vanishing at $X$ with order at least $m$. Thus, the study of the symbolic powers $I_X^{(m)}$ arises in a very natural way. 
 
On the other hand, determining the numerical characters or even the defining equations of symbolic powers of ideals are very delicate and challenging problems, see for instance the discussion after Questions \ref{1.Q}. In fact, the following natural questions regarding the symbolic powers of an ideal in a polynomial ring are very challenging in general:

\begin{Questions}\label{1.Q}
Fix an integer $m\geq 2$. Fix an ideal $I$ in a polynomial ring over a field  for which a minimal generating set and Betti table are known. What are
\begin{enumerate}
\item the minimal number of generators of the symbolic power $I^{(m)}$?
\item the minimal number of generators of $I^{(m)}$ not lying in the ordinary power $I^m$?\\
(this is a first estimate of the size of the symbolic power, called the {\em $m$-th symbolic defect} of $I$, see \cite{GGSV})
\item a minimal generating set of $I^{(m)}$?
\item the minimal degree of an element in $I^{(m)}$?
\item the Castelnuovo-Mumford regularity of $I^{(m)}$?
\item the Hilbert function of $I^{(m)}$?
\item the Betti table of $I^{(m)}$?
\end{enumerate}
\end{Questions}

Even when $I=I_X$ is the defining ideal of a set of points in $\mathbb P_k^n$ these questions may be extremely challenging. For instance, already the special case of Question \ref{1.Q}(6) where $I=I_X $ is a set of points, even general points, is a very important open problem in Algebraic Geometry, sometimes referred to as {\em the  interpolation problem}:
\begin{Open Problem}\label{IP}[Interpolation Problem]
Let $X$ be a set of general points in $\mathbb P_k^n$. For any integers $m\geq 2$ and $d\geq 2$, what is the number of  linearly independent equations of degree $d$ passing through $X$ with multiplicity $m$?
\end{Open Problem}

A large number of papers in the literature are devoted to Open Problem \ref{IP}.  However, so far, the best general result in this direction is a celebrated interpolation theorem of Alexander and Hirschowitz which in particular solves the case $m=2$, i.e. the case of double points. The 100-page long original proof of this result was obtained in a series of 4 papers \cite{AH1}, \cite{AH2}, \cite{AH3}, \cite{AH4}; despite intensive effort over the last 20 years, which allowed considerable simplifications of Alexander--Hirschowitz Interpolation Theorem, the Interpolation Problem is still wide--open for any $m\geq 3$.

In fact, even the apparently simpler problem stated in Question \ref{1.Q}(4) is still wide-open; it only asks for the smallest possible degree $d$ of a hypersurface passing through a finite set of points in $\mathbb P^n$. Nevertheless, a solution to this problem appears currently out of reach, even for {\em general points in the plane} (i.e. when $n=2$); in fact, a celebrated conjecture by Nagata predicting {\em a lower bound}  for $d$ in $\mathbb P^2$ is still wide-open, despite strong efforts made in the last 50 years (see for instance  \cite{CM}, \cite{CHMR}, \cite{DHST} ).

Answers to special cases of Questions \ref{1.Q}  are known for star configurations, especially for the cases of codimension 2, or symbolic squares; these results are proved in several papers in the literature by a number of authors, including Bocci, Chiantini, Cooper, Fatabbi, Galetto, Geramita, Guardo, Harbourne, Lampa-Baczy\'nska, Lorenzini, Malara, Migliore, Nagel, Park, Seceleanu, Shin, Szpond and Van Tuyl, see also Section 2.

\subsection{Our results}  
In this paper we answer {\em all} of the questions of Questions \ref{1.Q} for all star configurations of hypersurfaces. 
In particular, for any symbolic power $I^{(m)}$ where $I$ is the ideal of a star configuration of hypersurfaces we 
\begin{enumerate}
\item provide a structure theorem for $I^{(m)}$, exhibiting a minimal generating set of $I^{(m)}$ and determining the minimal number of generators of $I^{(m)}$, see Theorem \ref{4.Symb} and Corollary \ref{4.sdef};
\item provide a structure theorem for $I^{(m)}/I^m$ and determine its minimal number of generators, which is a first measure of how far are $I^{(m)}$ and $I^m$, see Theorem \ref{4.Symb3} and Corollary \ref{4.sdef}; our formulas are new even whem $m\leq 4$ or $c\leq 3$;
\item if the hypersurfaces have the same degree, we determine the degrees of all minimal generators of $I^{(m)}$ and provide an explicit combinatorial formula to determine the number of generators in each degree, see Theorem \ref{4.Symb2}(2). For most degrees, we also determine a closed formula for the minimal number of generators in each degree, see Theorem \ref{7.binom};

\item (a) introduce  {\em ideals with c.i. quotients}, which generalize ideals with linear quotients, see Definition \ref{7.ci};\\
(b)  prove that $I^{(m)}$ has c.i. quotients, see  Theorem \ref{6.delta}. In particular, when all forms are linear (e.g. star configuration of points), this implies that $I^{(m)}$ has linear quotients (Corollary \ref{7.linear});
\item prove that if all the hypersurfaces have the same degree $\delta$ then the Betti table of $I^{(m)}$ has a special structure, which we dub a {\em Koszul stranded Betti table}, see Corollary \ref{7.stranded}; 
\item give an explicit formula for all graded Betti numbers of $I^{(m)}$ (Theorem \ref{7.formula}). To facilitate the computations, we also provide {\em closed formulas} solely in term of the power $m$, the codimension $c$ and number of forms $s$, for more than half of the graded Betti numbers of $I^{(m)}$, explicitly determining the entire bottom portion of the Betti table (Theorem \ref{7.strands}), and some closed formulas for the top part of the Betti table (Proposition \ref{7.top}). 
These two results combined already provide a closed formula for the Betti table of $I^{(m)}$ for $m\leq 4$ and codimension at most 11, see Corollary \ref{7.small}. Our methods also illustrate why a simple closed formula for the entire Betti table solely in term of $m,c,s$ is practically impossible to obtain.
 \end{enumerate}
 
 In particular, our results virtually answer any question about $mZ$ and $I_Z^{(m)}$, when $Z$ is a star configuration of hypersurfaces.
\\
\\
Our proofs exploit two fruitful connections between the minimal generators of $I^{(m)}$ and two classical mathematical objects. The first one, is the set of positive solutions to Diophantine equations; this connection is fundamental for the computation of the minimal number of generators of $I^{(m)}$ and $I^{(m)}/I^m$. The second one is the set of all partitions of $m$; we employ this connection to prove that $I^{(m)}$ has c.i. quotients and determine precisely these quotients. While connections with partitions are not unexpected in this setting, the way we establish the connection is unusual as it relies heavily on the canonical form of a monomial (see below).

Additionally, along the way, we introduce two technical tools which may be of independent interest. The first one is a {\em canonical form} of a monomial in a given set of forms (see Definition \ref{3.monsupp}), which is inspired by Young tableaux and has an interpretation in terms of standard monomials in a polynomial ring having an ASL structure induced by its squarefree monomials. This canonical form allows us to describe the minimal generators of star configurations and quickly compute the smallest symbolic power of a star configuration containing a given monomial (Theorem \ref{4.Sdeg}). It is also used to define the connection with partitions (Corollary \ref{6.surj}). Our definition applies to all polynomials that can be written as monomials in a prescribed set of forms (no further assumption) and in general appears to be useful in the computation of symbolic powers. We prove its uniqueness under mild assumptions (Theorem \ref{3.NF}). 

The second one is a technical ingredient which we call the {\em index of overlap} of a partition $[\ul{d}]:=[d_1,\ldots,d_t]\vdash m$ with respect to a given order (Definition \ref{6.i0}). It identifies the smallest index $i_0$ in $[\ul{d}]$ guaranteeing the non-existence of strictly larger partitions having the same first $i_0$ entries equal to the ones of $[\ul{d}]$. 
We use it to show that $I^{(m)}$ has c.i. quotients and explicitly compute the involved colon ideals, see Theorem \ref{6.delta}.

\subsection{Working with star configurations of hypersurfaces instead of monomials} By work of Geramita, Harbourne, Migliore and Nagel, one can specialize {\em monomial star configurations} (i.e. star configuration on the variables) to star configurations of hypersurfaces in such a way that several numerical statements regarding symbolic powers of monomial star configurations also hold for symbolic powers of star configurations of hypersurfaces (see \cite[Thm~3.6]{GHMN}).

However, the situation for ideal--theoretic statements (e.g. inclusion or equality of ideals, properties of colon ideals, etc.) is more complicated; in fact, 
in the more general setting where star configurations of hypersurfaces are defined, several familiar properties of monomial ideals are lost, see Remark \ref{3.prob}. For instance, a ``monomial" in a set of forms may belong to a ``monomial ideal" without being multiple of any generator. 

Therefore, several statements for star configuration of hypersurfaces usually require different arguments and additional care than the corresponding results regarding monomial star configurations. See for instance the proofs of Theorem \ref{4.Sdeg}, or Theorem \ref{4.Symb3}, or \cite[Question~4.7]{GHMN}.\\
\\
The structure of the paper is the following: in Section 2 we recall definitions and results regarding symbolic powers of star configurations; in Section 3 we introduce the notion of {\rm normal form} of a monomial in a set of forms. It always exists, and we provide general and easy-to-check sufficient conditions ensuring its uniqueness. 

In Section 4 we prove the first two main results, namely the structure Theorems \ref{4.Symb} and \ref{4.Symb3}. In Corollary \ref{4.sdef} we compute the minimal number of generators of $I_{c,\F}^{(m)}$ and $I_{c,\F}^{(m)}/I_{c,\F}^m$. As special examples we compute these numbers when $c\leq 3$ or $m\leq 4$ (see Subsection 4.2).
If additionally the forms in $\F$ have the same degree, we determine the number of generators of $I_{c,\F}^{(m)}$ in each degree, see Theorem \ref{4.Symb2}. 

In Section 5 we introduce and study the index of overlap of a partition, and define a total order on a minimal generating set of $I_{c,\F}^{(m)}$. In Section 6, we build upon the previous sections to prove our third main result, namely that $I_{c,\F}^{(m)}$ has c.i. quotients (Theorem \ref{6.delta}). In the special case of linear star configurations, this means that their symbolic powers have linear quotients (Corollary \ref{7.linear}). 
In Section 7 we prove our fourth main result, which is a formula for every graded Betti number of  $R/I_{c,\F}^{(m)}$ (Theorem \ref{7.formula}). The last main result is Theorem \ref{7.strands}, where we provide closed formulas, solely in terms of $s,c$ and $|\F|$, for most of the Betti table of $R/I_{c,\F}^{(m)}$. We also give closed formulas in terms of $s,c$ and $|\F|$ for the top strand of the Betti table under some restrictions and explain why, despite the explicit formula of Theorem \ref{7.formula}, a general closed formula, solely in terms of $s,c$ and $|\F|$, for the remaining cases and the other strands is probably out of reach.

At the time that this paper was being concluded, a preprint was posted on arXiv by J. Biermann, H. De Alba, F. Galetto, S. Murai, U. Nagel, A. O'Keefe, T. R\"omer and A. Seceleanu  \cite{BDA+}. 
Independently from us, they prove result (4)(b) of Section 1.1 in the {\em monomial case}, i.e. when all forms have degree $\delta=1$ and are variables, see \cite[Thm~3.2~and~4.3]{BDA+}, (5) (see \cite[Corollary~4.4(1)]{BDA+}) and a weaker version of (6) (see \cite[Corollary~4.4]{BDA+} and Theorems \ref{7.strands} and \ref{7.binom} and Proposition \ref{7.top}). \\
\\
{\bf Acknowledgments:} the author would like to thank L. Sega for helpful conversations regarding ideals with linearly stranded Betti tables, and Ben Blum--Smith for pointing out the ASL interpretation of the normal form of a monomial. The author would also like to thank the anonymous referee for a patient and careful revision of this manuscript and a number of helpful suggestions.

\section{Star configurations and symbolic powers}
Let $R$ be a polynomial ring over a field of any characteristic, and $I$ a homogeneous ideal of $R$. 
For every $m\in \ZZ_+$ one may define the {\em $m$-th symbolic power of $I$} as the homogeneous ideal 
$$
I^{(m)} = \bigcap_{P\in \Ass(R/I)} I^mR_P\cap R.
$$
From the definition, for every $m,t\in \ZZ_+$ one has $I^{(m)}I^{(t)}\subseteq I^{(m+t)}$ and $I^m \subseteq I^{(m)}$; in general, however, ordinary and symbolic powers are different. As a first measure of how far is a symbolic power from being equal to the corresponding ordinary power, Galetto, Geramita, Shin and Van Tuyl in \cite{GGSV} introduce the following quantity, which they dub {\em the $m$-th symbolic defect} of $I$
$$
\sdef(I,m):=\mu(I^{(m)}/I^m),
$$
where $\mu(U)$ is the minimal number of generators of a finite graded $R$-module $U$.

\begin{Example}\label{2.ex0}
Let $R=k[x_0,x_1,x_2]$ and let $I=(x_0,x_1)\cap (x_0,x_2)\cap (x_1,x_2)=(x_0x_1,x_0x_2,x_1x_2)$ be the defining ideal of the 3 coordinate points in $\mathbb P^2$. Then 
$I^{(2)}=(xyz)+I^2$ and, since $xyz\notin I^2$ (by degree reasons), one has $\sdef(I,2)=1$. 
\end{Example}

We recall now the various notions of star configurations. 

\begin{definition}\label{2.star}
Let $k$ be a field, let $R=k[x_0,\ldots,x_n]$ be a polynomial ring over $k$ and $\F=\{F_1,\ldots,F_s\}$ be a set of forms in $R$. 
Fix $c\leq \min\{s,n\}$ and assume  that any subset of $c+1$ distinct elements of $\F$ forms a complete intersection. Then the ideal
$$
I_{c,\F}=\bigcap_{\{i_1,\ldots,i_c\}\subseteq F} (F_{i_1},\ldots,F_{i_c})
$$
is called {\em the ideal of the star configuration of codimension (or height) $c$ on $\F$}. We will often refer to $I_{c,\F}$ as {\em the star configuration of height $c$ on $\F$}. Any ideal of the form $I_{c,\F}$ is called a {\em star configuration of hypersurfaces in $\mathbb P_k^n$}. 

Additionally, $I_{c,\F}$ is called 
\begin{itemize}
\item a {\em linear star configuration} if all the forms in $\F$ have degree 1;
\item a {\em star configuration of points} if it is a linear star configuration and $c=n$;
\item a {\em monomial star configuration} if $\F=\{x_0,\ldots,x_n\}$.
\end{itemize}
\end{definition}

Star configurations of points also appear in the literature under the name ``$l$-laterals" (e.g. \cite{D}). 

\begin{Remark}
Monomial star configurations are the ideals generated by all squarefree monomials of a fixed degree, see for instance Example \ref{2.ex}(a). For instance, the ideal $I$ of Example \ref{2.ex0} is the monomial star configuration of height 2 in $R=k[x_0,x_1,x_2]$.
\end{Remark}

Next, we collect in a single statement some properties of star configurations; they are proved in higher generality in \cite[Thms~3.3 and 3.6]{GHMN} and \cite[Ex.~3.4]{GHMN}. 

\begin{Proposition}\label{2.ghmn}
Let $I_{c,\F}$ be a star configuration of hypersurfaces. Then 
\begin{enumerate}
\item $I_{c,\F}$ is minimally generated by $\{F_{i_1}\cdots F_{i_{(s+1-c)}}\,\mid\,1\leq i_1<i_2<\cdots <i_{(s+1-c)}\leq s\}$;
\item For every associated prime $P\in \Ass(R/I_{c,\F})$ there exists precisely one subset $\{F_{i_1},\ldots,F_{i_c}\}\subseteq \F$ such that $(F_{i_1},\ldots,F_{i_c}) \subseteq P$;
\item For every $m\geq 1$, the ideal $I_{c,\F}^{(m)}:=(I_{c,\F})^{(m)}$ is Cohen-Macaulay of codimension $c$;
\item For every $m\geq 1$ one has $I_{c,\F}^{(m)} = \bigcap_{1\leq j_1<\cdots < j_c\leq s} (F_{j_1},\ldots,F_{j_c})^{m}$.
\end{enumerate}
\end{Proposition}

We can now give a few examples.

\begin{Example}\label{2.ex}
Let $R=k[x_0,\ldots,x_3]$. 
\begin{itemize}
\item[(a)] Let $\F=\{x_0,\ldots,x_3\}$, then the following are monomial star configurations:
$$
I_{2,\F}=\bigcap_{0\leq i <j \leq 3}(x_i,x_j) = (x_0x_1x_2,x_0x_1x_3, x_0x_2x_3,x_1x_2x_3)
$$
and 
$$
I_{3,\F}=\bigcap_{0\leq i <j <h\leq 3}(x_i,x_j,x_h) = (x_0x_1,x_0x_2,x_0x_3,x_1x_2,x_1x_3,x_2x_3).
$$
\item[(b)] Let $F_1=x_0^3$, $F_2=(x_1-x_3)^2$, $F_3=x_2^5$, $F_4=x_1^4-x_2^2x_3^2+x_0^3x_1$, $F_5=x_0^7+x_1^7+x_2^5x_3^2$ and $\F=\{F_1,\ldots,F_5\}$. Any 4 of them form a regular sequence, thus
$$
I_{3,\F}= (F_iF_jF_h\,\mid\, 1\leq i < j < h \leq 5).
$$
\end{itemize}
\end{Example}

While the structure and minimal free resolution of $I_{c,\F}$ are now well-known (e.g. see \cite{GHMN}), much less is known about the symbolic powers $I_{c,\F}^{(m)}:=(I_{c,\F})^{(m)}$ of $I_{c,\F}$.

We recall some open questions regarding the symbolic powers of $I_c^{(m)}$ mentioned in the introduction. After each question, we summarize what is known about it.
\begin{Question}\label{Qgen}
How many minimal generators does $I_{c,\F}^{(m)}$ have? What are their degrees?
\end{Question}

 Recall that for a homogeneous ideal $J$ in $S$ one defines the {\em initial degree} of $J$ as 
$$\alpha(J) = \min\{d\in \ZZ_+\,\mid\, \text{ there exists }f\in J \text{ of }\deg(f)=d\}.$$

Partial answers to Question \ref{Qgen} are essentially only known when $I_{c,\F}$ is the ideal of a {\em monomial} star configuration; in these cases
\begin{itemize}
\item Bocci and Harbourne in \cite[Lemma~2.4.1]{BH} prove that if $m=rc$ is a multiple of $c$, then the initial degree of $I_{c,\F}^{(m)}$ is $\alpha(I_{c,\F}^{(rc)})=|\F| \cdot c$.\\
Thus $I_{c,\F}^{(rc)}$ contains no generators of degree $<|\F|\cdot c$ and at least one generator of degree $|\F|\cdot c$.

\item Lampa-Baczy\'nska and Malara in \cite[Prop~3.2 and 4.2]{LM} determine the minimal generators of $I_{c,\F}^{(m)}$ if $c\leq 2$ and $|\F|=3$ or if $c\leq 3$ and $|\F|=4$.

\item Herzog, Hibi and Trung in \cite[Prop~4.6]{HHT} prove that for $1\leq m \leq c$ one has 
$$
I_{c,\F}^{(m)} = I_{c-m+1,\F} + \sum\limits_{j=1}^{m-1} I_{c,\F}^{(j)}I_{c,\F}^{(m-j)}
$$
from these, one can determine generating sets for $I_{c,\F}^{(m)}$ which are very far from being minimal.
\end{itemize}
Combining the first partial result above with the techniques of \cite{GHMN}, one can prove that the initial degree of $I_{c,\F}^{(m)}$ is $\alpha(I_{c,\F}^{(rc)})=r\delta |\F|$ for any star configuration (not necessarily monomial), provided all forms of $\F$ have degree $\delta\ge 1$.
\\
\\
The following question is partly answered by Galetto, Geramita, Shin and Van Tuyl in \cite{GGSV} when $m\leq 3$, or when $c=2$ and $n=2$:
\begin{Question}\label{Qdef}
What is the symbolic defect $\sdef(I_{c,\F},m)=\mu(I_{c,\F}^{(m)}/I_{c,\F}^m)$?
\end{Question}
More in detail, they prove the following:
\begin{itemize}
\item if $I_{c,\F}$ is the ideal a monomial star configuration, then ${\rm sdef}(I_{c,\F}, m)=1$ if and only if $c=m=2$, see \cite[Thm~3.11]{GGSV};

\item $\sdef(I_{c,\F},2) \leq \binom{s}{c-2}$ and equality holds if $I_{c,\F}$ defines a linear star configuration, see \cite[Cor.~3.15]{GGSV};

\item $\sdef(I_{c,\F}, 3)\leq \binom{s}{c-3} + \binom{s}{c-2}\binom{s}{c-1}$, see \cite[Cor~3.17]{GGSV};

\item an upper bound for the symbolic defect when $c=2$, $m=2q$ is even, $n=2$, and $I_{c,\F}$ defines a linear star configuration in $\mathbb P_k^2$: ${\rm sdef}(I_{2,s}, 2q) \leq 1+s(q-1)$, see see \cite[Thm~3.20]{GGSV}.
\end{itemize}

\begin{Question}\label{QBetti}
What is the Betti table of $I_{c,\F}^{(m)}$?
\end{Question}

For Question \ref{QBetti} the following is known:

\begin{itemize}
\item ($c=2$) By  \cite{CF+}, the Betti table of $R/I_{2}^{(m)}$ is known for any $m\geq 1$, where $I_{2}:=I_{2,\F}$ is a monomial star configuration of codimension 2;

\item ($m=2$) In  \cite[Thm~3.2]{GHM}, the Betti table of  $R/I_{c,\F}^{(2)}$ is determined for any linear star configuration $I_{c,\F}$;

\item ($m=2$ and $c=2$) In \cite[Thm~5.3]{GGSV} the Betti table of  $R/I_{2,\F}^{(2)}$ is determined.

\end{itemize}

In the present paper we answer Questions \ref{Qgen} and \ref{Qdef} for any star configuration of hypersurfaces. Our methods provide a full answer also to Question \ref{QBetti}, however because the results would be extremely complicated to state in full generality, we choose to state them only when all forms have the same degree $\delta \geq1$. 

\section{The normal form of a monomial}

In this technical section we introduce a way of writing a monomial in a polynomial ring  inspired by Young tableaux, which we call the {\em normal form} of the monomial. The normal form is an important technical tool for the results in all subsequent sections; we prove its uniqueness and existence in a fairly general setting (see Theorem \ref{3.NF} and Propositions \ref{gcd} and \ref{mon}). 

\begin{definition}
Let $\F=\{F_1,\ldots,F_s\}$ be forms in a polynomial ring $R$ over a field $k$, we define a {\em monomial in $\F$} as an expression $M=F_{i_1}\cdots F_{i_r} \in k[F_1,\ldots,F_s]$ which is a product (possibly with repetition) of elements in $\F$. 
The monomial $M$ is {\em squarefree} if there is one expression  $M=F_{i_1}\cdots F_{i_r} $ for $M$ with $i_h\neq i_j$ for every $h\neq j$.

An ideal $I$ in $R$ is a {\em monomial ideal in $\F$} if $I$ has a generating set consisting of monomials in $\F$.
\end{definition}

For instance, let $\F=\{x_0,x_1,x_0+x_1\}$. Then $M=(x_0+x_1)^2$ is a monomial in $\F$, and $M'=x_0^2+x_0x_1=x_0(x_0+x_1)$ is a squarefree monomial in $\F$. The ideal $((x_0+x_1)^2,x_0^2+x_0x_1)$ is then a monomial ideal in $\F$.

\begin{Remark}\label{3.prob}
As anticipated in the introduction, while the use of the word ``monomial" appears natural in this setting, one should be aware that most properties of ordinary monomial ideals are not inherited in general by monomials in a set of forms $\F$. For instance, consider the following familiar properties of monomials:
\begin{enumerate}[(a)]
\item if $I$ is a monomial ideal, then all the associated primes in $R$ of $I$ are monomials;
\item if $M$ is a monomial and $I$ a monomial ideal, then $M\in I$ if and only if $M$ is a multiple of a minimal generator of $I$;
\item if $M,N_1,\ldots,N_r$ are monomials, then $(N_1,\ldots,N_r):M =  \sum_{i=1}^r (N_i:M)$.
\end{enumerate}
Then all of these statements fail, in general, if one replaces the word ``monomial" by ``monomial in $\F$".

This partly illustrates the difficulties in extending some results that hold for monomials to the case of ``monomials in $\F$". 
\end{Remark}

These problems are illustrated by the fact that \cite[Prop.3.8(1)]{GHMN} is not stated as an ``if and only if" statement, and \cite[Question~4.7]{GHMN} is stated as a question and not a theorem.
\\

{\it Proof of Remark \ref{3.prob}}
(a) Take $\F=\{F:=(x_0+x_1)^2\}\subseteq R=k[x_0,x_1]$ and $I=(F)$, then $\Ass_R(R/I)=\{(x_0+x_1)\}$, and $P=(x_0+x_1)$ is not a monomial ideal.

(b) and (c)  Take $\F=\{x_0,x_1,x_0^2+x_1^2\}$, let $I=(x_0,x_1)$, and $M=x_0^2+x_1^2$, then $M\in I$, but $M$ is not multiple of $x_0$ or $x_1$. 
Additionally, $(x_0:M) + (x_1:M) = (x_0,x_1) \neq (x_0,x_1):_RM=R$. 

\QED
\bigskip


We can now introduce the notion of normal form.

\begin{definition}\label{3.monsupp}
Let $\F$ be a set of forms in a polynomial ring $R$ over a field $k$. For any monomial $M$ in $\F$, a {\em normal form of $M$ with respect to $\F$} is an expression 
$$
M=M^{(1)}\cdots M^{(t)}, 
$$
where $M^{(i)} = \prod_{F\in S_i}F$ and $\emptyset \subseteq S_1 \subseteq S_2 \subseteq \cdots \subseteq S_{t}\subseteq \F$

The number $t$ of monomial factors in a normal form is called the {\em length} of the normal form. The minimum of all lengths of all normal forms of $M$ is denoted $\lambda_{\F}(M)$.
\end{definition}

We are grateful to Ben Blum--Smith for pointing out to us the following interpretation of the normal form of a monomial in a polynomial ring.
\begin{Remark}
Let $R=k[x_0,\ldots,x_n]$ and assume $\F=\{x_0,\ldots,x_n\}$. 
Let $H$ be the poset of all squarefree monomials in $R$ with the partial order $M\leq N$ $\Llra$ $\supp(M)\supseteq \supp(N)$. 

$H$ induces a structure of Algebra with Straightening Law (ASL) on $R=k[H]$, whose standard monomials  are precisely the monomials in $R$ written in the normal form, e.g. see \cite{Ei} or \cite[Ex.1.4(a)]{Br}.
\end{Remark}

Of course, the monomials $M^{(i)}$ appearing in the normal form need not be all distinct.
\begin{Example}\label{3.xyzw}
Let $R=k[x,y,z,w]$, $\F=\{x,y,z,w\}$, then a normal form of $M=x^2y^3zw$ is $
M=(xyzw)(xy)(y)$, i.e.
$$
M=M^{(1)}M^{(2)}M^{(3)} \; \text{ where }M^{(1)}=xyzw,\; M^{(2)} = xy, \; \text{ and }\; M^{(3)}=y.
$$
The normal form of $N=x^7y^2z^3w^6$ is
$$
N=(xyzw)(xyzw)(xzw)(xw)(xw)(xw)(x) $$
where $N^{(1)}=N^{(2)}=xyzw$, $N^{(3)} = xzw$, $N^{(4)}=N^{(5)}=N^{(6)}=xw$ and $N^{(7)}=x$.
Theorem \ref{3.NF} implies that these are the only normal forms of $M$ and $N$, thus $\lambda(M)=3$ and $\lambda(N)=7$.
\end{Example}

It is not hard to see that normal forms with respect to a $\F$ always exist, with no assumptions on $\F$ (the proof is very similar to the existence part of the proof of Theorem \ref{3.NF}). However, in general, there could be multiple distinct normal forms; therefore we define a slightly more restricted setting where we can prove at once existence and uniqueness of the normal form.

\begin{definition}\label{3.unique}
Let $R$ be a polynomial ring over a field $k$, and $\F=\{F_1,\ldots,F_s\}$ be forms in $R$. We say that $\F$ {\em allows a unique monomial support} if for any monomial $M$ in $\F$, there is only one way (up to relabelling and reordering) to write $M$ as a monomial in $\F$. 
If this is the case, 
\begin{itemize}
\item for any monomial $M=F_{i_1}^{a_{i_1}} \cdots F_{i_d}^{a_{i_d}}$, the set $\supp(M)=\{F_{i_1},\ldots,F_{i_d}\}$ is well-defined and called the {\em support of $M$ (with respect to $\F$)}.

\item Then {\em normal form of $M$ with respect to $\F$} can be written as 
\begin{equation}\label{3.def}
M=M^{(1)}\cdots M^{(t)}, \; \text{ where the } M^{(j)} \text{ are squarefree monomials in }\F, 
\end{equation}
$$\text{and }\;\;\emptyset\neq \supp(M^{(i+1)}) \subseteq \supp(M^{(i)}) \text{ for all } i=1,\ldots,t-1.$$

When the set $\F$ has been specified and there is no ambiguity, we simply refer to (\ref{3.def}) as a normal form of $M$ (omitting further reference to $\F$) and its length as $\lambda(M)$.
\end{itemize}
\end{definition}

It is clear that to have a unique monomial support one needs assumptions on $\F$. E.g. if $F_1=xy$, $F_2=zw$, $F_3=xw$ and $F_4=yz$, then the element $M=xyzw$ can be written as $F=F_1F_2=F_3F_4$. Thus $\supp(M)$ is not well-defined in this case. First, we give a couple of easy sufficient conditions.
\begin{Example}
Let $R=k[x_0,\ldots,x_n]$.
\begin{enumerate}
\item If $\F\subseteq \{x_0,\ldots,x_n\}$ is a subset of the variables of $R$, then $\F$ allows a unique monomial support (which is the ``classical" monomial support).
\item If $F_1,\ldots,F_s$ is a regular sequence in $R$, then the subalgebra $k[F_1,\ldots,F_s]\subseteq R$ is isomorphic to a polynomial ring, thus $\F=\{F_1,\ldots,F_s\}$ allows a unique monomial support. 
\end{enumerate}
\end{Example}

Since we want to construct more classes of examples, in Proposition \ref{gcd} and Proposition \ref{mon} we identify general and easily checked sufficient conditions on $\F$ ensuring that $\F$ allows a unique monomial support.
\begin{Proposition}\label{gcd}
Let $R=k[x_0,\ldots,x_n]$. If $\F=\{F_1,\ldots,F_s\}$ are forms in $R$ such that $\gcd(F_i,F_j)=1$ for every $i\neq j$, then $\F$ allows a unique monomial support. 
 \end{Proposition} 

\begin{proof}
We prove by induction on $d\geq 1$ that for any monomial $M=F_{i_1}^{a_{i_1}}\cdots F_{i_d}^{a_{i_d}}$ in $\F$ with $\lambda(M)=d$, the set $\supp(M):=\{F_{i_1},\ldots,F_{i_d}\}$ is well-defined. 

If not, then there are two distinct expressions $M=F_{i_1}^{a_{i_1}}\cdots F_{i_d}^{a_{i_d}}=F_{j_1}^{b_{j_1}}\cdots F_{j_u}^{b_{j_u}}$, where the $F_{i_h}, F_{j_k} \in \F$. Since $R$ is a domain, by cancellation and induction we may further assume the sets $\{F_{i_1},\ldots,F_{i_d}\}$ and $\{F_{j_1},\ldots,F_{j_u}\}$ are disjoint. Since $R$ is a UFD, then $F_{j_1}$ divides $F_{i_1}\cdots F_{i_d}$, which contradicts the assumption. 
 \end{proof}
 
 One immediately obtains that in the setting of star configurations of hypersurfaces the notion of support is well-defined. This is crucial for our description of their symbolic powers.
 \begin{Corollary}
 Let $R=k[x_0,\ldots,x_n]$, let $\F$ be $s$ forms in $R$ and fix $1\leq c < s$. If any subset of $(c+1)$ elements of $\F$ forms a regular regular sequence, then $\F$ allows a unique monomial support. 
 \end{Corollary}

One may wonder whether the sufficient condition in Proposition \ref{gcd} is also necessary; this is not the case. For instance, in $R=k[x,y]$, the support of a monomial in $\F=\{x^2,xy\}$ is well-defined, despite the fact that $\gcd(x^2,xy)=x$. Indeed, more generally, one can prove the following

 \begin{Proposition}\label{mon}
  Let $R=k[x_0,\ldots,x_n]$. Let $F_1,\ldots,F_s$ be monomials in $R$ such that for every $1\leq i \leq s$, except at most one, one has $\supp(F_i)\not \subseteq \bigcup_{j\neq i}\supp(F_j)$. Then $k[F_1,\ldots,F_s]$ is isomorphic to a polynomial ring in $s$ variables, thus in particular $\F=\{F_1,\ldots,F_s\}$ allows a unique monomial support. 
    \end{Proposition}
 
 \begin{proof}
 Let $\pi:k[y_1,\ldots,y_s] \lra k[F_1,\ldots,F_s]$ be the natural epimorphism of $k$-algebras defined by $\pi(y_j)=F_j$. We prove $\pi$ is injective by induction on $s\geq 1$. If $s=1$ there is nothing to prove, so we may assume $s\geq 2$.  To prove injectivity of $\pi$ let $a_1,\ldots,a_s\in k[F_1,\ldots,F_s]$ be polynomials such that $\sum_{j=1}^s a_jF_j=0$, we need to show that all the $a_j=0$. 
 
(1) First, assume there exists an index $i$ with $a_i\neq 0$ and a variable $x\in \supp(F_i)\setminus \bigcup_{j\neq i}\supp(F_j)$. Since $x$ divides $F_i$ and does not divide the other $F_j$, then there exists $h \neq i$ such that $x$ divides one of the monomials in $a_h$. If we write $a_h = \sum c_{\ul{\alpha}}F^{\ul{\alpha}}$ where $0\neq c_{\ul{\alpha}}\in k$ and $F^{\ul{\alpha}}=F_1^{\alpha_1}\cdots F_s^{\alpha_s}$ is a monomial in $k[F_1,\ldots,F_s]\subseteq R$, we then have that $x$ divides $F^{\ol{\ul{\alpha}}}$ for some index $\ol{\ul{\alpha}}=(\ol{\alpha}_1,\ldots,\ol{\alpha}_s)$ with $c_{\ol{\ul{\alpha}}}\neq 0$. Since $x$ does not divide $F_j$ for $j\neq i$, this implies $\ol{\alpha}_i>0$. Let $a_h' = a_h - c_{\ol{\ul{\alpha}}} F^{\ol{\ul{\alpha}}} $ and $a_i' = a_i + c_{\ol{\ul{\alpha}}} F^{\ul{\widehat{\alpha}}}$ where $\ul{\widehat{\alpha}}$ is obtained from $\ol{\ul{\alpha}}$ by subtracting $1$ to $\ol{\alpha}_i$ and adding 1 to $\ol{\alpha}_h$. 

We can iterate this procedure until we may assume that there exists an expression $\sum_{j=1}^s b_jF_j=0$ with $x$  not appearing in the monomial support of any of the $b_j$, thus $b_j \in A=k[F_r\,\mid\,r\neq i]$. Since also $F_j \in A$ while $F_i$ is a monomial divisible by $x$, it follows that $b_i=0$. Thus, we have an expression $\sum_{j\neq i }b_j F_j=0$ with $b_j\in A=k[F_r\,\mid\,r\neq i]$ for all $j$. By induction hypothesis, this implies that also $b_j=0$ for all $j\neq i$, proving injectivity of $\pi$. 

(2) We may then assume we are not in (1), which immediately implies (by assumption) that there is only one index $j$ with $a_j\neq 0$, i.e. $a_jF_j=0$ (and, additionally, $\supp(F_j)\subseteq \bigcup_{h\neq j}\supp(F_h)$, but this is irrelevant at this point). Since $k[F_1,\ldots,F_s]\subseteq R$ is a domain, this implies that $a_j=0$. This finishes the proof. 
 
 \end{proof}

\begin{Remark}\label{3.rem}
It follows from the definition that if $\F$ allows a unique monomial support and $N_1,\ldots,N_t$ are monomials in $\F$, then $\supp(N_1\cdots N_t)=\supp(N_1)\cup \cdots \cup \supp(N_t)$.
\end{Remark}
 
 We now prove existence and uniqueness of the normal form.
 
 \begin{Theorem}\label{3.NF}[Existence and uniqueness of the normal form]
 Let $\F$ be a set of forms in $R=k[x_0,\ldots,x_n]$. If $\F$ allows a unique monomial support, then for every monomial $M$ in $\F$ there exists a unique normal form of $M$ with respect to $\F$. 
 
 In particular, in this case $\lambda_{\F}(M)$ is the length of the normal form of $M$.
 \end{Theorem}
 
\begin{proof}
First, observe that by Remark \ref{3.rem}, if any such form (\ref{3.def}) exists, then one has $\supp(M)=\supp(M^{(1)})$. 
Write $M=F_{i_1}^{a_{i_1}}\cdots F_{i_d}^{a_{i_d}}$ with $a_{i_h}\geq 1$, we prove both parts of the statement by induction on $a:=a_{i_1}+\ldots +a_{i_d}\geq 1$. If $a=1$, then $M=F_{i_1}$ is the only normal form of $M$.

Assume then $a>1$. Let $M^{(1)} = \prod_{F_j\in \supp(M)}F_j$ denote the product of all the $F_j$ in the support of $M$, let $M'=M/M^{(1)}$. By  Remark \ref{3.rem} and the definition of $M^{(1)}$ we have $\supp(M')\subseteq \supp(M)=\supp(M^{(1)})$. 

Since $M'$ is the product of $a'<a$ of the $F_j$, then by induction hypothesis, there exists a unique way to write $M'$ in the form (\ref{3.def}), say 
$M'=M^{(2)}M^{(3)} \cdots M^{(u)}$. We claim that $M=M^{(1)}M^{(2)}M^{(3)} \cdots M^{(u)}$ satisfies condition (\ref{3.def}).

Indeed, $M^{(1)}$ is squarefree by construction, and each $M^{(i)}$ for $i\geq 2$ is squarefree by induction. Additionally, for every $2\leq i \leq u-1$ we have $\supp(M^{(i+1)})\subseteq \supp(M^{(i)})$ by inductive hypothesis, and $\supp(M^{(2)})\subseteq \supp(M^{(1)})$ by the above. 

This proves existence in the inductive step. Uniqueness follows from the initial observation that any form (\ref{3.def}) for $M$ must have $\supp(M^{(1)})=\supp(M)$ and the assumption on $\F$, thus the only option for $M^{(1)}$ is precisely $M^{(1)} = \prod_{F_j\in \supp(M)}F_j$.
So the choice of $M^{(1)}$ is unique by the above, and the choice of the other squarefree monomials $M^{(i)}$ is unique by induction. This proves the uniqueness part of the statement. 

\end{proof}

\begin{Corollary}\label{3.monomial}
Any monomial in a polynomial ring has a unique normal form.
\end{Corollary}

In the following section, we employ the normal form to unveil the structure of the symbolic powers of star configurations.

\section{Symbolic powers of star configurations: minimal generating sets and symbolic defects}

In this section we  provide the structure theorems for the symbolic powers of all star configurations of hypersurfaces, see Theorems \ref{4.Symb} and \ref{4.Symb3}. In particular, we can compute their minimal number of generators (which was not known, even for monomial star configurations, see Question \ref{Qgen}), and the subtler invariant  $\sdef(I,m)=\mu(I^{(m)}/I^m)$ introduced in \cite{GGSV}, called the $m$th symbolic defect of $I$, which provides a first measure of how different are $I^{(m)}$ and $I^m$. \\
\\
Since we want to discuss star configurations of hypersurfaces, these are our running assumptions:
\begin{Setting}\label{4.set}
We let $k$ be a field of any characteristic, and let
\begin{itemize}
\item $R=k[x_0,\ldots,x_n]$ be a polynomial ring over a field;
\item $c,s$ be integers with $1\leq c < s$;
\item $\F=\{F_1,\ldots,F_s\}$ be forms in $R$ such that any $(c+1)$ of them form a regular sequence.
\end{itemize}
Additionally, 
\begin{itemize}
\item the ideal $I_c=I_{c,\F}$ denotes the star configuration of height $c$ on $\F$, i.e. 
$$I_c=I_{c,\F}=\bigcap_{1\leq j_1<\cdots < j_c\leq s} (F_{j_1},\ldots,F_{j_c});$$
\item by ``a monomial" we mean ``a monomial in $\F$"; by its ``normal form", we mean ``its normal form with respect to $\F$".

\end{itemize}
\end{Setting}

We use the normal form of $M$ to introduce  the following useful invariant. 

\begin{definition}\label{4.def}
Let $R,c,s, \F$ be as in Setting \ref{4.set}. For any monomial $M$ in $\F$, let $M=M^{(1)}M^{(2)}\cdots M^{(t)}$ be its normal form.  The {\em symbolic degree of $M$} with respect to $c$ and $\F$ is the non-negative integer
$$
\Sdeg_{c,\F}(M)=\sum\limits_{i=1}^t \max\{0, c - s + |\supp(M^{(i)})|\}.
$$
When $\F$ is understood, we simply write $\Sdeg_c(M):=\Sdeg_{c,\F}(M)$. If $c$ is also understood, we write $\Sdeg(M):=\Sdeg_{c,\F}(M)$.
\end{definition}

The statement of Theorem \ref{4.Sdeg} explains the choice of the name.

 \begin{Example}\label{4.ex1}
 Let $R=k[x, y, z, w]$ and $\F=\{x,y,z,w\}$. Let $M,N$ be as in Example \ref{3.xyzw}, then
$$\Sdeg_{2}(M) = \max\{0, 2-4+4\} + \max\{0, 2-4+2\} + \max\{0, 2-4+1\} = 2,$$
while
$$\Sdeg_{3}(M) = \max\{0, 3-4+4\} + \max\{0, 3-4+2\} + \max\{0, 3-4+1\} = 4,$$
Also, $\Sdeg_2(N) = 5$, while $\Sdeg_3(N)=11$.
\end{Example}

\begin{Example}\label{4.ex2}
Let $R=k[x_0,\ldots,x_5]$ and let $\F=\{F_1,\ldots,F_7\}$ be 7 forms such that any $5$ of them form a regular sequence. Let $M=F_1^4F_2^3F_3^2F_4F_5^2F_6^3F_7^4$, then its normal form is
$$
M=(F_1F_2F_3F_4F_5F_6F_7)(F_1F_2F_3F_5F_6F_7)(F_1F_2F_6F_7)(F_1F_7).
$$
Then, for instance, $\Sdeg_4(M)=4+3+1+0=8$, while $\Sdeg_2(M)=2+1+0+0=3$.
\end{Example}

By definition, $\Sdeg(M)\geq 0$. We first characterize when $\Sdeg(M)=0$.
\begin{Lemma}\label{4.0}
Let $R,c,s,\F$ be as in Setting \ref{4.set}. 

For any monomial $M$ one has $\Sdeg_{c,\F}(M)=0$ if and only if $|\supp(M)|\leq s-c$.
\end{Lemma}

\begin{proof}
From the proof of Theorem \ref{3.NF} we have that $\supp(M^{(1)})=\supp(M)$. 
Now, since $\supp(M^{(i+1)})\subseteq \supp(M^{(i)})$, then 
$\max\{0,c-s+|\supp(M^{(i+1)})|\}\leq \max\{0,c-s+|\supp(M^{(i)})|\}$ for every $i$. 

``$\Lra$" assume $\Sdeg_{c,\F}(M)=0$. Since each number $\max\{0,c-s+|\supp(M^{(i)})|\}$ is non-negative, and their sum is zero, then $c-s+|\supp(M^{(i)})|\leq 0$ for every $i$.  Since $\supp(M)=\supp(M^{(1)})$ (by the proof of Theorem \ref{3.NF}), the statement follows.

``$\Lla$" Assume $|\supp(M)|\leq s-c$. 
For every $i$ we have 
$$ 0\leq  \max\{0,c-s+|\supp(M^{(i)})|\} \leq  
 \max\{0,c-s+|\supp(M)|\}=0 
$$
thus $\max\{0,c-s+|\supp(M^{(i)})|\}$ are all zero, and therefore their sum, which is $\Sdeg(M)$, is zero.
\end{proof}

The definition of $\Sdeg(M)$ and the proof of Theorem \ref{3.NF} yield the following:
\begin{Remark}\label{4.sq}
Let $R,c,s,\F$ be as in Setting \ref{4.set}. 
\begin{enumerate}
\item If $N$ is a squarefree monomial in $\F$, then $\lambda(N)=1$ and $N=N^{(1)}$ is its normal form. Thus for any $c$ one has $\Sdeg_c(N)=\max\{0, c-s+|\supp(N)|\}$. 

\item If $M=M^{(1)}\cdots M^{(t)}$ is the normal form of a monomial $M$, then 
$$
\Sdeg_c(M)=\sum\limits_{i=1}^t \Sdeg_c(M^{(i)}).
$$
\end{enumerate}
\end{Remark}

For our intended application, we need to understand better the behavior of this invariant when we multiply two monomials. Since  it is not easy to describe the normal form of $MN$ from the normal forms of $M$ and $N$, then it is not clear what is a sharp relation between $\Sdeg_{c}(MN)$, $\Sdeg_{c}(M)$ and $\Sdeg_{c}(N)$.

It is easily proved that $\Sdeg_{c}(MN) \geq \max\{\Sdeg_{c}(M), \Sdeg_{c}(N)\}$, so one may wonder whether this is, in general, the sharpest possible inequality. The answer is negative: in Proposition \ref{4.mult}(3) we prove the following sharper inequality
$$\Sdeg_{c}(MN)\geq \Sdeg_{c}(M) + \Sdeg_{c}(N).$$

Interestingly, however, our proof of this sharper inequality (proved in Proposition \ref{4.mult}(3)) heavily relies on a symbolic power interpretation of the symbolic degree, which is determined in the next result. We identify the symbolic degree of a monomial with the smallest symbolic power of $I_{c,\F}$ containing the given monomial. We use the convention that $I^{(0)}=R$.

\begin{Theorem}\label{4.Sdeg}
Let $R,c,s,\F$ be as in Setting \ref{4.set}. For any monomial $M$ one has
$$
\Sdeg_{c}(M) = \max\{u \,\mid\, M\in I_{c}^{(u)} \} = \text{ the integer } m \text{ with }M\in I_c^{(m)}\setminus I_{c}^{(m+1)}.
$$

\end{Theorem}

For instance, let $R,\F$ and $M$ be as in Example \ref{4.ex2}. Since $\Sdeg_2(M)=3$, then $M\in I_2^{(3)}\setminus I_2^{(4)}$; since $\Sdeg_4(M)=8$, then $M\in I_4^{(8)}\setminus I_4^{(9)}$.

Theorem \ref{4.Sdeg} is the key to prove the structure theorem for any symbolic power of any star configuration of hypersurfaces, Theorem \ref{4.Symb}, which in turn allows the computation of their minimal number of generators and  symbolic defects, see Theorems \ref{4.Symb2} and \ref{4.Symb3}, and Corollaries \ref{4.sdef}, \ref{4.sdef2}, \ref{4.sdef3}, and \ref{m<5}.

\begin{proof}
For any monomial $N$, we write $\Sdeg(N)$ for $\Sdeg_{c,\F}(N)$; we also write $I$ for $I_{c,\F}$. 
Let $m :=\Sdeg(M)$. We first show that $M\in I^{(m)}$ when $M$ is squarefree. In this case, if $m=0$, there is nothing to prove. If $m>0$ then  by Remark \ref{4.sq} we have $|\supp(M)|\geq s-c+m$. Then for any subset $J\subseteq \F$ with $|J|=c$ one has $|\supp(M) \cap J | \geq m$, i.e. $M\in (F\in \F\,\mid\, F \in J)^m$, thus, by Proposition \ref{2.ghmn}(4) $M\in I^{(m)}$.

For the general case, let $M=M^{(1)}\cdots M^{(t)}$ be the normal form of $M$. Since each $M^{(i)}$ is squarefree, then by the above, $M^{(i)} \in I^{(\Sdeg(M^{(i)}))}$, thus 
$$M\in \prod_{i=1}^t I^{(\Sdeg(M^{(i)}))} \subseteq I^{(\sum\limits_{i=1}^t \Sdeg(M^{(i)}))} = I^{(m)},$$
where the rightmost equality follows by Remark \ref{4.sq}.

To conclude the proof it suffices to prove that $M\notin I^{(m+1)}$. If $m=0$, then by Lemma \ref{4.0} we have $|\supp(M)|\leq s-c$. If we were working with ordinary monomials, Proposition \ref{2.ghmn}(1) would now allow us to conclude the statement. However, as explained in Section 1.2 and Remark \ref{3.prob}, to prove $M\notin I$ it is not sufficient to show that $M$ is not a multiple of any minimal generator of $I$. Instead, we use the primary decomposition of $I$. Let $H\subseteq \supp(M)^C :=\F\setminus \supp(M)$ be a subset of $c$ elements of $\supp(M)^C$; by assumption on $\F$, the ideal $\a_H=(F\,\mid\, F\in H)$ is a complete intersection. Let $P\in \Ass(R/\a_H)$ be an associated prime, then $\h(P)=c$; since $I\subseteq \a_H$ has height $c$, it follows that $P\in \Ass(R/I)$. 
Since $P$ contains $H$, then by Proposition \ref{2.ghmn}(2) the prime $P$ does not contain any element in $H^C=\F\setminus H$, in particular it does not contain any element in $\supp(M)$, therefore $M$ is a unit in $R_P$. This shows that $M\notin IR_P$, and thus $M\notin I$.

Next, assume $m>0$, and thus $|\supp(M^{(1)})|=|\supp(M)|>s-c$ by Lemma \ref{4.0}. We may then define the number 
$$j=\max\{i \,\mid\, |\supp(M^{(i)})| > s-c\}.$$
Thus, for any $1\leq i \leq j$ we may write $|\supp(M^{(i)})|= s-c +d_i$ for some $d_i \in \ZZ_+$, and then 
$$|\Sdeg(M^{(i)})| = d_i \qquad \text{ for all }1\leq i \leq j.$$
In particular, by Remark \ref{4.sq}(2), one has $\Sdeg(M)=d_1+\ldots + d_j$.\\
\\
{\bf Claim.} There exists a subset $U\subseteq \{F_1,\ldots,F_s\}$ of $c$ elements with $|U \cap \supp(M^{(i)})| = d_i$ for every $1\leq i \leq j$, and $U\cap \supp(M^{(h)})=\emptyset$ for every $h>j$.  \\
\\

Let $V := \supp(M^{(j)})^C$ be the complement of $\supp(M^{(j)})$ in $\F$; by definition $|V|=c-d_j<c$. We add to $V$ a set of $d_j$ elements of $\F$ as follows: if $j=t$, we just take $W$ to be any subset of $d_j$ elements in $\supp(M^{(j)})$; if $j<t$, we let $W$ be a subset of $d_j$ elements in $\supp(M^{(j+1)})^C$. Set 
$$U:=V\cup W.$$ 
By construction $|U|=c$. 
First observe that, independently of whether $j=t$ or $j<t$, the definitions of $U$ and of normal form give
$$
\supp(M^{(j)})^C \subseteq U, \;\text{ and }\supp(M^{(i)})^C \subseteq \supp(M^{(i+1)})^C \text{ for every }1\leq i\leq t-1.
$$

We now prove that $U$ has the claimed properties.  
We first show that $U \cap \supp(M^{(i)})=\emptyset$ for all $i>j$. If $j=t$ there is nothing to prove, so we may assume $j<t$. By construction of $U$, we have 
$$U\subseteq \supp(M^{(j)})^C \cup \supp(M^{(j+1)})^C = \supp(M^{(j+1)})^C \subseteq \supp(M^{(i)})^C$$ 
for every $i>j$. Therefore,  $U \cap \supp(M^{(i)})=\emptyset$ for $i>j$.  

Additionally, since $\supp(M^{(i)})^C \subseteq U$, then $\supp(M^{(i)}) \cap U = U \setminus \supp(M^{(i)})^C$, thus  
$$
|\supp(M^{(i)}) \cap U| = |U| - |\supp(M^{(i)})^C| = c-(s-(s-c+d_i))=d_i,
$$
for every $1\leq i \leq j$. This proves the claim.

We finish the proof of the theorem. Let $D_i = U \cap \supp(M^{(i)})$ for every $1\leq i \leq j$. Let $\a:=\a_U=(F\,\mid\, F \in U)$ be the complete intersection ideal generated by the $c$ elements of $\F$ in $U$. By Proposition \ref{2.ghmn}(4), to show that $M\notin I^{(m+1)}$ it suffices to show $M\notin \a^{m+1}$, and thus  it suffices to show that $M\notin (\a R_P)^{m+1}$ for some $P\in \Ass(R/\a^{m+1})$. 

Let $P$ be any associated prime of $R/\a^{m+1}$. Since $\a$ is a complete intersection ideal, then $\a^{m+1}$ is Cohen-Macaulay and $\Ass(R/\a)=\Ass(R/\a^{m+1})$ (e.g. \cite{CN}), thus in particular $\h(P)=c$. By assumption on $\F$, for any $F_h\in U^C=\F\setminus U$ we have $F_h\notin P$, thus $F_h$ is a unit in $R_P$. Then, by the Claim, we have 
\begin{itemize}
\item $M^{(h)}R_P=R_P$ is the unit ideal for $h>j$, 
\item and $M^{(i)}R_P = \left(\prod_{h\in D_i}F_h\right)R_P$.
\end{itemize}
It follows that $MR_P =\left(\prod_{i=1}^j \left(\prod_{h\in D_i}F_h\right)\right) R_P$. Since $\a_P=\a R_P$ is a complete intersection of height $c$, then the associated graded ring ${\rm gr}_{\a_P}(R_P)$ is isomorphic to a polynomial ring in $c$ variables; thus, in particular, each $F_h$ in $D_i$ lies in $\a_P \setminus \a_P^2$, and the order of $M$ in ${\rm gr}_{\a_P}(R_P)$ is the sum of the orders, i.e.
$\sum_{i=1}^j \left( \sum_{h\in D_i}1\right)=\sum_{i=1}^j d_i = m.$ (Recall that the order of an element $f$ in a Noetherian local ring $(A,\m)$ is $o(f):=\max\{t\in \NN_0\,\mid\, f\in \m^t\}$.)
Therefore, $M\notin (\a_P)^{m+1}$, which concludes the proof.

\end{proof}

We can now determine how the symbolic degree of a product compare with the symbolic degree of each factor.
\begin{Proposition}\label{4.mult}
Let $R,c,s,\F$ be as in Setting \ref{4.set}. Let  $M,N$ be monomials and  $F\in \F$. 

\begin{enumerate}
\item \begin{enumerate}
\item If $|\supp(M)|<s-c$, then $\Sdeg_c(FM)=\Sdeg_c(M)=0$;
\item  If $F\notin \supp(M)$ and $|\supp(M)|\geq s-c$, then $\Sdeg_c(FM)=\Sdeg_c(M)+1$;
\end{enumerate}
\item If $N$ is squarefree and $\supp(M)\subseteq \supp(N)$, then $\Sdeg_c(MN)=\Sdeg_c(M)+\Sdeg_c(N)$.
\item One has $\Sdeg_c(MN)\geq \Sdeg_c(M)+\Sdeg_c(N)$.
\end{enumerate}

\end{Proposition}

\begin{proof}
For simplicity, we write $\Sdeg(M)$ for $\Sdeg_c(M)$. 

(1) Let $M=M^{(1)}\cdots M^{(t)}$ be the normal form of $M$. (a) If $|\supp(M)|<s-c$, then $|\supp(FM)|\leq s-c$. By Lemma \ref{4.0}, both $\Sdeg(M)=\Sdeg(FM)=0$. 

(b) Since $|\supp(M^{(1)})|=|\supp(M)|\geq s-c$, then $\Sdeg(M^{(1)}) = c-s+|\supp(M^{(1)})|$. Since $F\notin \supp(M)$, then 
$$\Sdeg(FM^{(1)})= c-s+|\supp(FM^{(1)})|  =1+(c-s+|\supp(M^{(1)})|)=1+\Sdeg(M^{(1)}).$$
Now, the normal form of $FM$ is $FM = N^{(1)}M^{(2)}\cdots M^{(t)}$, where $N^{(1)}=FM^{(1)}$. It follows by Remark \ref{4.sq} that
$$
\begin{array}{ll}
\Sdeg(FM) &= \Sdeg(N^{(1)}) + \sum\limits_{i=2}^t \Sdeg(M^{(i)})=1+\Sdeg(M^{(1)})+\sum\limits_{i=2}^t \Sdeg(M^{(i)})\\ 
&= 1 + \Sdeg(M).
\end{array}$$

(2) Let $Q:=MN$, and let $M=M^{(1)}\cdots M^{(t)}$ be the normal form of $M$. Since $\supp(M)\subseteq \supp(N)$, then $\supp(Q)=\supp(N)$. Since $N$ is squarefree, the proof of Theorem \ref{3.NF} gives that the first term $Q^{(1)}$ in the normal form of $Q$ is $Q^{(1)}=N$, and since $Q/Q^{(1)}=M$, then 
$$
Q=NM^{(1)}\cdots M^{(t)} \text{ is the normal form of }Q=MN.
$$
Therefore, by Remark \ref{4.sq}, $\Sdeg(Q)=\Sdeg(N)+\Sdeg(M)$.

(3) Let $m=\Sdeg(M)$ and $m'=\Sdeg(N)$, then by Theorem \ref{4.Sdeg} one has $M\in I_{c}^{(m)}$ and $N\in I_{c}^{(m')}$. It follows that $MN \in I_{c}^{(m)}I_{c}^{(m')} \subseteq I_{c}^{(m+m')}$. 
Since $MN \in I_{c}^{(m+m')}$ then by Theorem \ref{4.Sdeg}, one has $\Sdeg(MN)\geq m+m'$. 

\end{proof}

We can now apply Theorem \ref{4.Sdeg} to provide an explicit description of the minimal generating set of symbolic powers of star configurations.
\begin{Theorem}\label{4.Symb}
Let $R,c,s,\F$ be as in Setting \ref{4.set}. 
For any $m\geq 1$, a minimal generating set of $I_{c,\F}^{(m)}$ is given by
\be\label{I_cm}
\G_{c,(m)}= \left\{ \begin{array}{l |l}
 &  \Sdeg(M)=m, \; \text{ and } |\supp(M^{(j)})|\geq s+1-c \\
M \text{ monomial in }\F										& \\ 										 &  \mbox{ for every } M^{(j)} \text{ in the normal form of }M  \end{array}\right\}
\ee
\end{Theorem}

\begin{proof}
First, we observe that, by Theorem \ref{4.Sdeg}, for any $M\in I_{c,\F}^{(m)}$ one has $\Sdeg(M)\geq m$. 
\ul{$\G_{c,(m)}$ is a generating set of $I_{c,\F}^{(m)}$.} By \cite[Theorem~3.6(1)]{GHMN}, there exists a minimal generating set $\mathcal A$ of $I_{c,\F}^{(m)}$ consisting of monomials in $\F$. We claim that for each $N\in \mathcal A$ we have $\Sdeg(N)=m$. Indeed, if not, then by the above we must have $\Sdeg(N)>m$. Now, by \cite[Theorem~3.6(1)]{GHMN} and the definition of normal form, there exists a monomial star configuration $I_c$ and a minimal monomial generator $N'$ of $I_c^{(m)}$ with $\Sdeg_c(N)\geq m+1$. For monomial star configurations one has $I_c^{(m+1)}\subseteq \m \cdot I_c^{(m)}$ where $\m$ is the homogeneous maximal ideal of $R$ (e.g. take $P=\m$ in \cite[Prop.~7]{EM}), thus $N'$ could not be a minimal generator of $I_c^{(m)}$, yielding a contradiction.

Now, for any $N\in \mathcal A$, let $N=N^{(1)}\cdots N^{(t)}$ be the normal form of $N$, let $N':=N'=N^{(1)}\cdots N^{(u)}$ where $u$ is the maximum index $j$ such that $|\supp(N^{(j)})|\geq s+1-c$. Clearly the above is the normal form of $N'$, the monomial $N'$ divides $N$, and, by construction and Lemma \ref{4.0} and Remark \ref{4.sq}, we have $\Sdeg(N')=\Sdeg(N)=m$. Then $N'\in \Gc$.  

Now, if $\Sdeg(N)=m$, then $N'\in \Gc$, which proves this first part of the statement.\\
\\
\ul{Minimality of $\G_{c,(m)}$.} 
By \cite[Thm~3.6(2)]{GHMN}, the ideal $I_{c,\F}^{(m)}$ has the same minimal number of generators as $I_{c,\HH}^{(m)}$, where $I_{c,\HH}$ is the monomial star configuration of degree $c$ in $S=k[y_1,\ldots,y_s]$. 
Thus if we prove the minimality of $\G_{c,(m)}$ when $\F$ consists of variables, then the minimality of $\G_{c,(m)}$ follows in general. Assume then $\F$ consists of variable. To prove minimality of $\Gc$ we show that no element  $M \in \Gc$ is divisible by any other element of $\Gc$. This follows at once by the following stronger claim:  if $M\in \G_{c,(m)}$ and a variable $y \in \F$ divides $M$, then $M':=M/y\notin I_{c,\F}^{(m)}$.

Let $M=M^{(1)}\cdots M^{(t)}$ be the normal form of $M$. Since $y\in \supp(M)=\supp(M^{(1)})$, then the number $h=\max\{j\in \{1,\ldots,t\} \,\mid\, y\in \supp(M^{(j)})\}$ is well-defined. 
We observe that 
$$
M/y = M^{(1)}\cdots M^{(h-1)}\left( \frac{M^{(h)}}{y}\right) M^{(h+1)}\cdots M^{(t)}
$$
is the normal form of $M/y$ (indeed, we only need to check the condition on the supports, which is satisfied by our definition of $h$). In particular, by Remark \ref{4.sq} we have 
{\small
$$
\Sdeg(M/y) = \sum\limits_{j\neq h}\Sdeg(M^{(j)}) + \Sdeg\left(\frac{M^{(h)}}{y}\right) =\Sdeg(M) -\Sdeg(M^{(h)}) + \Sdeg\left(\frac{M^{(h)}}{y}\right).$$}

Since $|\supp(M^{(h)})|\geq s+1-c$, then by Proposition \ref{4.mult}(1)(b) we have $\Sdeg\left(\frac{M^{(h)}}{y}\right)$ $ = \Sdeg(M^{(h)}) - 1$. This proves that $\Sdeg(M/y)=\Sdeg(M)-1=m-1$, therefore, by Theorem \ref{4.Sdeg}, $M/y\notin I_c^{(m)}$, concluding the proof.

\end{proof}

\subsection{Symbolic defects}

Let $\mu(W)$ denote the minimal number of generators of a finite graded $R$-module $W$. For any homogeneous $R$-ideal $I$, the $R$-module $I^{(m)}/I^m$ measures how far is $I^m$ from the $m$-th symbolic power of $I$. The goal of this subsection is to prove the structure result for for $I_{c,\F}^{(m)}/I_{c,\F}^{m}$ (see Theorem \ref{4.Symb3}).

Once we have proved it, in Corollary \ref{4.sdef} we obtain numerical statements about the minimal number of generators of this module which, we recall, is dubbed the {\em $m$-th symbolic defect} of $I_{c,\F}$, written $\sdef(I_{c,\F},m):=\mu(I_{c,\F}^{(m)}/I_{c,\F}^m)$. 

The $m$-th symbolic defect has been introduced in \cite{GGSV} as a first measure of the difference between $I^{(m)}$ and $I^m$, and the authors study it for ideals of points and star configurations in particular. Indeed, in \cite{GGSV} partial results are proved regarding the symbolic defects of $I_{c,\F}^{(m)}$ when $c\leq 3$ or $m\leq 3$, see the discussion after Question \ref{Qdef}. 

As an application of Corollary \ref{4.sdef}, we improve and complete these results, see Subsection 4.2.
\bigskip

\begin{Proposition}\label{4.Icm}
Let $R,c,s,\F$ be as in Setting \ref{4.set}. Then for any $m\geq 1$
\begin{enumerate}
\item $(N\in I_c^m\,\mid\,|\supp(N)|>s-c+1) \subseteq (F_1,\ldots,F_s)I_c^{(m)}$;
\item $\Gc \cap I_c^m = \{N^m\,\mid\, N\in G_{c,(1)}\}$.
\end{enumerate}
\end{Proposition}

\begin{proof} (1) Let $N$ be a monomial in $I_c^m$ with $|\supp(N)|>s-c+1$. We may write $N=QM_1\cdots M_m$ where each $M_j$ is a minimal monomial generator of $I_c$, i.e. $M_j\in \G_{c,(1)}$, and $Q$ is an element of $R$. Since $|\supp(N)|>s-c+1$, by Remark \ref{3.rem} there exists $F\in \supp(Q)\cup\supp(M_1)\cup \cdots \cup \supp(M_{m-1})$ with $F\notin \supp(M_m)$. If $F\in \supp(Q)$, then $N \in (F)I_c^m$ and the statement follows. If $F\in \supp(M_i)$ for some $i$, write $M_i = FM'$. Since $s>c$, then $|\supp(M_i)|=s-c+1\geq 2$, thus $M'$ is a non-unit monomial $M'$  and $M_1M_i \in M'(FM_1)$. Since $FM_1$ is squarefree, then $\Sdeg(FM_1)=2$, thus $FM_1\in I_c^{(2)}$, so $M_1M_i \in M'I_c^{(2)}$. Therefore
$N \in \left(Q\prod_{j\neq 1,i}M_j \right)M' I_c^{(2)}\subseteq  I_c^{m-2} M' I_c^{(2)}\subseteq M' I_c^{(m-2)}I_c^{(2)} \subseteq M' I_c^{(m)}$, proving the statement.

(2) The inclusion ``$\supseteq$" is clear, so we prove the other inclusion. Let $M\in \Gc$ and assume $M\in I_c^m$. If $|\supp(M)|>s-c+1$, then by (1) we have $M\in (F_1,\ldots,F_s)I_c^{(m)}$, contradicting that $M$ is a part of a minimal generating set of $I_c^{(m)}$ (Theorem \ref{4.Symb}). Then $|\supp(M)|=s-c+1$. Since $M\in \Gc$, it follows that every $M^{(j)}$ in the normal form of $M$ has $|\supp(M^{(j)})|=s-c+1$, and since $\Sdeg(M)=m$ it follows that $M=(M^{(1)})^m$. Since $M^{(1)}$ is squarefree with $|\supp(M^{(1)})|=s-c+1$ it follows that $M^{(1)}\in \G_{c,(1)}$, thus  $M\in  \{N^m\,\mid\, N\in G_{c,(1)}\}$.

\end{proof}

We are ready for the second structure theorem.
\begin{Theorem}\label{4.Symb3}
Let $R,c,s,\F$ be as in Setting \ref{4.set}. For any $m\geq 1$, a minimal generating set of $I_{c,\F}^{(m)}/I_{c,\F}^m$ as an $R$-module is given by
\be\label{I_cmQ}
\G_{c,(m)}'= \left\{ \begin{array}{l |l}
 &  \Sdeg(M)=m, \; \text{ and } |\supp(M^{(j)})|\geq s+2-c \\
M \text{ monomial in }\F										& \\ 										 &  \mbox{ for every } M^{(j)} \text{ in the normal form of }M  \end{array}\right\}
\ee
equivalently, $\G_{c,(m)}^{\prime}=\left\{M \in \G_m \,\mid |\supp(M)|\geq s+2-c \right\}$.
\end{Theorem}

\begin{proof}
$\G_{c,(m)}'$ is a generating set of $I_c^{(m)}/I_c^m$ by Proposition \ref{4.Icm}(2). Observe that to prove minimality it is not sufficient to show that no element of $\Gc'$ is divisible by another element of $\Gc'$ (for the reasons explained in Remark \ref{3.prob} regarding monomials in $\F$). Instead, assume by contradiction that $\Gc'$ is not minimal, then there  exist monomials $M_0, M_1,\ldots,M_r \in \Gc$ and $N_1,\ldots,N_u$ in $I_c^m\setminus \Gc$ such that $M_0 \in (M_1,\ldots,M_r, N_1,\ldots,N_u)$. 

We prove that $N_i \in  (F_1,\ldots,F_s)I_c^{(m)}$ for every $i=1,\ldots,u$. If $|\supp(N_i)|>s-c+1$, the statement is proved in Proposition \ref{4.Icm}(1). If $|\supp(N_i)|=s-c+1$ then $N_i = QN^m$ for some $N\in \G_{c,(1)}$ and some monomial $Q$ in $\F$. Since $N_i \notin \Gc$, by Proposition \ref{4.Icm}(2) we obtain that $Q$ is a non-unit monomial, thus $N_i  \in  (F_1,\ldots,F_s)I_c^{(m)}$.

Therefore, $M_0 \in (M_1,\ldots,M_r) + (F_1,\ldots,F_s)I_c^{(m)}$. By Nakayama's Lemma, it follows that $I_c^{(m)}$ is generated by $\Gc \setminus \{M_0\}$, contradicting the minimality of $\Gc$ (Theorem \ref{4.Symb}).

\end{proof}

We conclude this section by establishing a combinatorial formula allowing us to count the number of minimal generators of $I_{c,\F}^{(m)}$ in each degree. For simplicity, we only state it when all the forms of $\F$ have the same degree. First, however, we need to establish an order convention.\\
\\
{\bf Convention:} For any subset $B\subseteq [c]:=\{1,2,\ldots,c\}$ we write the elements of $B$ in increasing order, i.e. $B=\{b_1,\ldots,b_h\}$ with $1\leq b_1<b_2<\cdots<b_h\leq c$.\\
\\
With this convention, we then have:

\begin{Theorem}\label{4.Symb2}
Let $R,c,s,\F$ be as in Setting \ref{4.set}. 
If all forms in $\F$ have degree $\delta\geq 1$, then 
\begin{enumerate}
\item For any fixed $m\geq 1$, the ideal $I_c^{(m)}$ has only generators in degrees
$\delta(t(s-c)+m)$, where $\displaystyle \lceil \frac{m}{c} \rceil \leq t \leq m$.

\item For every fixed $m\geq 1$ and $t$ with $\displaystyle \lceil \frac{m}{c} \rceil \leq t \leq m$, the number of generators of $I_c^{(m)}$ of degree $\delta(t(s-c)+m)$ is 
$$\sum\limits_{B=\{b_1,\ldots,b_h\} \subseteq [c]} \left(\left|\mathbb S_{t,B}\right|  \binom{s}{c-b_h}\binom{s-c+b_h}{b_h-b_{h-1}}\binom{s-c+b_{h-1}}{b_{h-1} - b_{h-2}}\cdots \binom{s-c+b_2}{b_2-b_1} \right) $$ where

$\mathbb S_{t,B}$ is the set of all distinct positive solutions to the system of Diophantine equations
$$
\left\{\begin{array}{cccccccl}
b_1x_1 &+& b_2x_2 &+& \ldots& +& b_hx_h & = m\\
x_1 	     & + & x_2 & + & \ldots  & + & x_h & = t.
\end{array}   \right.
$$

\end{enumerate}
\end{Theorem}

\begin{proof}
(1) By Theorem \ref{4.Symb} it suffices to prove that the elements in $\G_m$ have degrees $\delta(t(s-c)+m)$ for some $\left\lceil\frac{m}{c}\right\rceil \leq t \leq m$. So let $M\in \G_m$ and let 
$$
M= M^{(1)}\cdots M^{(t)}
$$
be the normal form of $M$ with respect to $\F$.  Let $m_i:=\Sdeg(M^{(i)})$, i.e. $|\supp(M^{(i)})|$ $=s-c+m_i$ for every $i=1,\ldots,t$. By Theorem \ref{4.Symb} we have $m=\Sdeg(M)=\sum_{i=1}^t m_i$, thus 
{\small$$
\deg(M)=\delta \sum\limits_{i=1}^t |\supp(M^{(i)})| = \delta \sum_{i=1}^t (s-c+m_i) =\delta( t(s-c)+\sum_{i=1}^t m_i) =\delta( t(s-c)+m),
$$ }
where $t=\lambda(M)$. It is easily seen that $t$ is in the desired range: clearly $t\leq m$, or else $\Sdeg(M)=\sum_{i=1}^t m_i \geq t>m$ yielding a contradiction.
On the other hand, if $t<\left\lceil\frac{m}{c}\right\rceil$ then $\Sdeg(M)=\sum_{i=1}^t m_i \leq \sum_{i=1}^t c =tc<m$, which is again a contradiction.

(2) We begin by setting two pieces of notation.
A {\em positive solution} to a Diophantine equation $\sum_{i=1}^tb_ix_i=m$ is an ordered $t$-uple $(n_1,\ldots,n_t)$ such that $\sum_i b_in_i=m$ and $n_i>0$ for every $i$. 
 For the rest of this proof we use a variation of the normal form: let $M=M^{(1)}\cdots M^{(t)}$ be the normal form of $M$, if we collect together all terms $M^{(j)}$ having the same support, we can then write
\be
M=M_c^{d_c}M_{c-1}^{d_{c-1}}\cdots M_2^{d_2}M_1^{d_1}
\ee
where $M_1,\ldots,M_c$ are squarefree monomials in $\F$ with $\supp(M_1)\subsetneq \supp(M_2) \subsetneq \cdots  \subsetneq \supp(M_c)$, $|\supp(M_i)|=s-c+i$, and $d_i\geq 0$. 
Since $|\supp(M_i)| = s+i-c$, then $\Sdeg(M_i)=i$ for every $i$, and by Proposition \ref{4.mult}(2) we have $\Sdeg(M)=\sum_{i=1}^c id_i$. 

Now, for any such $M$ we call the {\em (positive) normal support} of $M$ the set $\NS(M)$ $=\{i\in [c]\,\mid\, d_i>0\}=:B$. This set detects only the $M_i$ appearing with positive exponent in the normal form of $M$. %
Following the convention stated right before the theorem, we write this normal support $B$ as $B=\{b_1,\ldots,b_h\}$ with $b_1<b_2<\cdots <b_h$.

By the above, $\Sdeg(M)=m=\sum_{i=1}^c id_i = \sum_{b_i \in \NS(M)=B} b_id_{b_i}$.

Let $\mathbb S_t:=\bigsqcup_{B \subseteq[c]}\mathbb S_{t,B}$ be the disjoint union of the sets $\mathbb S_{t,B}$ described in the statement. For any $t \geq 1$, we consider the function 
$$f_t:U_t:=\{M\in \G_{c,(m)}\,\mid\,\deg(M)=\delta(t(s-c)+m)\} \lra \mathbb S_{t}$$
 defined by $f_t(M)=(d_{b_1},\ldots,d_{b_h})$ where $\{b_1,\ldots,b_h\}=\NS(M)$. We use $f_t$ to compute $|U_t|$, which is the quantity we need to determine.

 We prove that if $\mathbb S_t\neq \emptyset$, then $f_t$ is surjective. 
 Let $\ul{d}=(d_{b_1},\ldots,d_{b_h})\in \mathbb S_{t,B}$.  For any $i=1,\ldots,h$ let $M_{b_i}:=x_1x_2\cdots x_{s-c+b_i}$ and define $M:=M_{b_h}^{d_{b_h}}\cdots M_{b_2}^{d_{b_2}}M_{b_1}^{d_{b_1}}$. Since $\ul{d} \in \mathbb S_{t,B}$, then $\sum_{i=1}^hb_id_{b_i}=m$, thus, by Proposition \ref{4.mult}(2) we have $\Sdeg(M)=\sum_{i=1}^hb_id_{b_i}=m$. Additionally, we have
 {\small $$\deg(M)=\delta\left(\sum_{i=1}^h d_{b_i}(s-c+b_i) \right) =\delta\left( \left(\sum_{i=1}^hd_{b_i}(s-c)\right) + \sum_{i=1}^h d_{b_i}b_i \right)= \delta(t(s-c)+m).$$}
 This proves that $f_t(M)=(\ul{d})$, thus $f_t$ is surjective.  
 
Next, for any $B\subseteq [c]$ we compute  $|f^{-1}(\mathbb S_{t,B})|$. 
One has $M\in f_t^{-1}(\ul{d})$ if and only if $\NS(M)=B$, $M=M_{b_h}^{d_{b_h}}M_{b_{h-1}}^{d_{b_{h-1}}} \cdots M_{b_1}^{d_{b_1}}$ and $\deg(M)=\delta(t(s-c)+m)$. 

We begin by examining what are the possible squarefree monomials appearing as $M_{b_h}$: there are $\binom{s}{s-c+b_h}=\binom{s}{c-b_h}$ such possible monomials (one for each subset of $s-c+b_h$ elements of $[s]$). After we fix one such monomial $M_{b_h}$, there are $\binom{s-c+b_h}{s-c+b_{h-1}}=\binom{s-c+b_h}{b_h-b_{h-1}}$ possible squarefree monomials appearing as $M_{b_{h-1}}$ (one for each subset of $s-c+b_{h-1}$ elements of $\supp(M_{b_h})$). Analogously, after fixing $M_{b_h}$ and $M_{b_{h-1}}$, there are precisely $\binom{s-c+b_{h-1}}{s-c+b_{h-2}}=\binom{s-c+b_{h-1}}{b_{h-1}-b_{h-2}}$ possibilities for the monomial $M_{b_{h-2}}$, and so on.

This shows that for any $\ul{d}\in \mathbb S_{t,B}$ there are precisely 
$$\binom{s}{c-b_h}\binom{s-c+b_h}{b_h-b_{h-1}}\binom{s-c+b_{h-1}}{b_{h-1} - b_{h-2}}\cdots \binom{s-c+b_2}{b_2-b_1}$$ distinct monomials in $f_t^{-1}(\ul{d})$. Notice that this number is independent of $\ul{d}$, thus 
$$
\begin{array}{ll}
\left|f_t^{-1}(\mathbb S_{t,B})\right| &= \sum_{\ul{d}\in \mathbb S_{t,B}} \binom{s}{c-b_h}\binom{s-c+b_h}{b_h-b_{h-1}}\binom{s-c+b_{h-1}}{b_{h-1} - b_{h-2}}\cdots \binom{s-c+b_2}{b_2-b_1} \\
&\\
&=|\mathbb S_{t,B}| \binom{s}{c-b_h}\binom{s-c+b_h}{b_h-b_{h-1}}\binom{s-c+b_{h-1}}{b_{h-1} - b_{h-2}}\cdots \binom{s-c+b_2}{b_2-b_1}.\\
\end{array}
$$
Recall that $U_t=\{M\in \G_{c,(m)}\,\mid\,\deg(M)=\delta(t(s-c)+m)\}$. From the above we obtain $\left| U_t \right| = \sum_{B\subseteq [c]} \left|f_t^{-1}(\mathbb S_{t,B})\right| =  \sum_{B\subseteq [c]}|\mathbb S_{t,B}| \binom{s}{c-b_h}\binom{s-c+b_h}{b_h-b_{h-1}}\cdots \binom{s-c+b_2}{b_2-b_1}$, which finishes the proof.

\end{proof}

We can now determine a minimal generating set of $I_c^{(m)}/I_c^m$ and then deduce the symbolic defects of star configurations.

\begin{Corollary}\label{4.sdef}
Let $R,c,s,\F$ be as in Setting \ref{4.set}.  Then, for any $m\geq 1$, 
\begin{enumerate}
\item the minimal number of generators $\mu(I_c^{(m)})$ of $I_c^{(m)}$ is 
$$
\sum\limits_{B=\{b_1,\ldots,b_h\}\subseteq [c]} |\mathbb S_B| \binom{s}{c-b_h}\binom{s-c+b_h}{b_h-b_{h-1}}\binom{s-c+b_{h-1}}{b_{h-1} - b_{h-2}}\cdots \binom{s-c+b_2}{b_2-b_1},
$$
where $\mathbb S_B:=\bigcup_{t\geq 1}\mathbb S_{t,B}$ is the set of all distinct positive solutions to the Diophantine equation $b_1x_1+\cdots + b_hx_h=m$.

\item the symbolic defect $\sdef(I_c,m) $ of $I_c$ is
{\small \bee\label{G_m'}
\left(\sum\limits_{B=\{b_1,\ldots,b_h\}\subseteq [c]} |\mathbb S_B| \binom{s}{c-b_h}\binom{s-c+b_h}{b_h-b_{h-1}}\binom{s-c+b_{h-1}}{b_{h-1} - b_{h-2}}\cdots \binom{s-c+b_2}{b_2-b_1} \right)- \binom{s}{c-1}.
\eee}

\end{enumerate}
\end{Corollary}

\begin{proof}
Part (1) follows from Theorem \ref{4.Symb2}. Part (2) follows from (1) and Theorem \ref{4.Symb3}.

\end{proof}

\begin{Remark}
In Theorem \ref{7.binom} we provide a closed formula for the number of minimal generators of degree $\delta(t(s-c)+m)$ when $t\geq \left\lfloor \frac{m}{2}\right\rfloor$.
\end{Remark}

\subsection{Symbolic defects for small height or small symbolic powers}

We now illustrate how the previous results provide explicit minimal generating sets and closed formulas for number of generators and symbolic defects for $I_{c,\F}^{(m)}$ when the height $c$ is small (Corollaries \ref{4.sdef2} and \ref{4.sdef3}) or the power $m$ is small (Corollary \ref{m<5}). Already these special cases are new in the literature.\\
\\
Let us recall that Galetto, Geramita, Shin and Van Tuyl employed results of Harbourne and Huneke about the symbolic Rees algebra of points in $\mathbb P^2$ to prove in \cite[Thm. 3.20]{GGSV} the following upper bound for the {\em even} symbolic defects of $I_{2,\F}$ -- provided $\F$ consists of linear forms and the ambient ring $R=k[x_0,x_1,x_2]$ is a polynomial ring in $3$ variables:
\be
\sdef(I_{2,\F} ,2q) \leq 1+s(q-1).
\ee
No upper bound was known for the odd symbolic defects $\sdef(I_{2,\F} ,2q+1)$. 

For illustration, we now extract the case $c=2$ from Corollary \ref{4.sdef} -- it is Corollary \ref{4.sdef2}. Already this special case provides
\begin{itemize}
\item the precise formula for the even symbolic defect $\sdef(I_{2,\F} ,2q)$ of star configuration of any degrees (not necessarily ``linear")  of height 2 and {\em any} projective space $\mathbb P^n$ (not necessarily $n=2$).
\item the precise formula for the \underline{odd} symbolic defects $\sdef(I_{2,\F} ,2q+1)$ -- which is fairly different from the formual for even symbolic defects, 
\item an explanation of why the two formulas are  different: essentially because the size of the solution set of  $2x_1+x_2=m$ depends heavily on the parity of $m$;
\item an explicit description of the minimal generating sets of $I_2^{(m)}$ and $I_2^{(m)}/I_2^m$.
\end{itemize}

In the special case of even symbolic defects of {\em linear} star configurations in $\PP^2$, Corollary \ref{4.sdef2} below shows that the above upper bound proved by Galetto, Geramita, Shin and Van Tuyl is actually an equality. \\
\\
Also, in the proofs we associate to the set $B=\{b_1,\ldots,b_h\}$ the Diophantine equation $b_1x_{b_1} + \ldots + b_hx_{b_h}=m$ instead of the equivalent equation $b_1x_1+\ldots + b_hx_h=m$, this provides consistency across all Diophantine equations obtained from the sets $B$. For instance, to $B=\{1\}$ and $B=\{2\}$ we associated $x_1=m$ and $2x_2=m$ instead of $x_1=m$ and $2x_1=m$.

\begin{Corollary}\label{4.sdef2}(``The case $c=2$")
Fix $c=2$, let $R,s,\F$ be as in Setting \ref{4.set} (in particular any three forms in $\F$ form a regular sequence).

Let $I_2$ denote a star configuration of height 2 on $\F$, let $G:=F_1F_2\cdots F_s$ and for any $1\leq i \leq s$ let $G_i:=F/F_i=\prod_{j\neq i}F_j$. Then for any $m\geq 1$,
\begin{enumerate}
\item 
\begin{enumerate}[(a)]
\item a minimal generating set of $I_2^{(m)}$ is given by
\be
\G_{2,(m)}= \{G^jG_i^{m-2j}\,\mid\,0\leq j \leq \lfloor \frac{m}{2} \rfloor,\,\,i=1,\ldots,s\}.
\ee

\item $\mu(I_2^{(m)})=\left\{ \begin{array}{ll}
s+ \lfloor\frac{ms}{2}\rfloor, & \mbox{ if }m \in 2\ZZ+1 \\
&\\
1+ \frac{ms}{2}, & \mbox{ if }m \in 2\ZZ\\
\end{array}\right.
$\\

\end{enumerate}

\item \begin{enumerate}
\item A minimal generating set of $I_2^{(m)}/I_2^m$ as $R$-module is given by
\be
\G_{2,(m)}'= \{G^jG_i^{m-2j}\,\mid\,1\leq j \leq \lfloor \frac{m}{2} \rfloor,\,\,i=1,\ldots,s\}.
\ee

\item One has 
\be  \sdef(I_2,m)=\left\{ \begin{array}{ll}
s\lfloor \frac{m}{2} \rfloor, & \mbox{ if }m \in 2\ZZ+1 \\
&\\
1+ s\left( \frac{m}{2} -1\right), & \mbox{ if }m \in 2\ZZ\\
\end{array}\right.
\ee
\end{enumerate}
\end{enumerate}

\end{Corollary}

\begin{proof}
We first prove (1)(b) and (2)(b). Since $c=2$, then either $B=\{1\}$ or $B=\{2\}$ or $B=\{1,2\}$. Let $N_B:=|\mathbb S_B|$. Corollary \ref{4.sdef} yields
\be
\mu(I_2^{(m)})= sN_{\{1\}} + N_{\{2\}} + sN_{\{1,2\}}=s+N_{\{2\}} + sN_{\{1,2\}},
\ee
because $N_{\{1\}}=1$ (corresponding to the only solution to the equation $x_1=m$).
Observe that $N_{\{2\}}$ is the number of positive solutions to the equation $2x_2=m$. If $m\in 2\ZZ+1$ is odd, there are no solutions, if $m\in 2\ZZ$ is even, then there is precisely one solution. So, 
\be
\mu(I_2^{(m)})= \left\{ \begin{array}{ll}
s+ sN_{\{1,2\}}, & \mbox{ if } m\in 2\ZZ+1\\
s+1+ sN_{\{1,2\}}, & \mbox{ if } m\in 2\ZZ
\end{array}\right.
\ee

Next, we need to compute $N_{\{1,2\}}$, i.e. the number of distinct positive solutions to the diophantine equation $2a_1+a_2=m$. First, assume $m=2t+1$ is odd. Since $m$ is odd, the equation has integer solutions if and only if $a_2$ is odd. Since $a_1>0$ and $1\leq a_2\leq m=2t+1$, then there are $t$ odd possible values of $a_2$ in the range $1\leq a_2 \leq m-1=2t$. For each of them, there is precisely one solution for $a_1$. So when $m$ is odd we have $N_{\{1,2\}} = t= \lfloor \frac{m}{2} \rfloor$.

Now, assume $m=2t$ is even. This time $2a_1+a_2=m$ has solutions if and only if $a_2$ is even. Since $a_2\geq 1$, then there are precisely $t-1 = \frac{m}{2}-1$ solutions. This proves that 
$
N_{\{1,2\}} = \frac{m}{2}  - 1$ if $m$ is even. Therefore
\be
\mu(I_2^{(m)})= \left\{ \begin{array}{ll}
s+ s\left(\lfloor \frac{m}{2} \rfloor\right), & \mbox{ if } m\in 2\ZZ+1\\
s+1+ s\left(\frac{m}{2}-1\right), & \mbox{ if } m\in 2\ZZ
\end{array}\right.
\ee
and then by Corollary \ref{4.sdef} 
\be
\sdef(I_2,m)= \left\{ \begin{array}{ll}
s\lfloor \frac{m}{2} \rfloor, & \mbox{ if }m \in 2\ZZ+1 \\
&\\
1+ s\left( \frac{m}{2} -1\right), & \mbox{ if }m \in 2\ZZ\\
\end{array}\right.
\ee

(1)(a) Let $M\in \G_{2,(m)}$. Since $c=2$, every $M^{(j)}$ in the normal form of $M$ has $|\supp(M^{(j)})|\geq s+1-c=s-1$. Thus, by collecting together all $M^{(j)}$ with the same support we see that for every $M$ there exists an index $i$ and integers $a_1,a_2$ such that $M=G^{a_1}G_i^{a_2}$, where $G$ and $G_i$ are as in the statement. Since $\Sdeg(G)=2$ and $\Sdeg(G_i)=1$ for every $i$, we have $2a_1+a_2=m$. Then, by Theorem \ref{4.Symb}, we obtain
$$
\G_{2,(m)} =\left\{M=G^{a_1}G_i^{a_2}\,\mid\, 1\leq i \leq s,\, 2a_1+a_2=m \right \}.
$$
Since $2a_1+a_2=m$, then $0\leq a_1 \leq  \lfloor \frac{m}{2} \rfloor$. This proves (1)(a).

Part (2)(a) follows from (1)(a) and Corollary \ref{4.sdef}.

\end{proof}

\begin{Remark}
For $c\geq 3$ no formula was known for $\mu(I_c^{(m)})$ or $\sdef(I_c,m)$. In some sense, Corollary \ref{4.sdef} explains why: when $c\geq 3$ grows, the formula gets progressively more involved. To illustrate it, we provide the closed formula for star configurations when $c=3$; observe that it depends on the remainder of the power $m$ when divided by $6$.

Closed formulas can be given for any $c\geq 4$, the issue is that if one sets $L(c):={\rm lcm}(1,2,\ldots,c)$, then these formulas have $L(c)$ distinct cases, one for each possible remainder of $m$ when divided by $L(c)$.
\end{Remark}

\begin{Corollary}\label{4.sdef3}(``The case $c=3$")
Fix $c=3$, let $R,s,\F$ be as in Setting \ref{4.set}. For any positive integer $m$, write $m=6q+r$ where $q,r$ are integers and $0\leq r \leq 5$. 
\begin{enumerate}
\item If $G:=F_1F_2\cdots F_s$, $G_i := G/F_i$ for every $i=1,\ldots,s$ and $G_{ij}:=G/F_iF_j$ for every $1\leq i < j \leq s$, then a minimal generating set of $I_3^{(m)}$ is  
\be
\G_{3,(m)} = \{G^{a_1}G_i^{a_2}G_{ij}^{a_3}\,\mid\, 1\leq i < j \leq s,\; a_h\in \mathbb N_0 \mbox{ and }3a_1+2a_2+a_3=m\},
\ee
and one has \be
\mu(I_3^{(m)})=\left\{
\begin{array}{ll}
\binom{s}{2}( \frac{m^2}{6} + \frac{m}{3} )  + s\cdot \frac{m}{6} + 1, & \mbox{ if } r=0 \\
\binom{s}{2} ( \frac{(m-1)^2}{6} + 2\frac{m-1}{3} +1  )  + s\frac{m-1}{6}, & \mbox{ if } r=1 \\
\binom{s}{2} ( \frac{(m-2)^2}{6} + m -1) + s  ( 1 + \frac{m-2}{6}) , & \mbox{ if } r=2 \\
\binom{s}{2} (\frac{(m-3)^2}{6} + 4\frac{m}{3}  -1  )  + s\frac{m-3}{6} + 1, & \mbox{ if } r=3 \\
\binom{s}{2} (\frac{(m-4)^2}{6} +  5\frac{m-1}{3} -1  )   + s(1  + \frac{m-4}{6}), & \mbox{ if } r=4 \\
\binom{s}{2} ( \frac{(m-5)^2}{6}  + 2m - 4 ) + s\frac{m+1}{6} , & \mbox{ if } r=5 \\
\end{array}
\right.
\ee

\item A minimal generating set of $I_3^{(m)}/I_3^m$ is  
\be
\G_{3,(m)}' = \{G^{a_1}G_i^{a_2}G_{ij}^{a_3}  \in \G_{3,(m)} \,\mid\, a_1+a_2>0\},
\ee
and  the $m$-th symbolic defect of $I_3$ is
\be
\sdef(I_3,m)= \left\{
\begin{array}{ll}
\binom{s}{2}( \frac{m^2}{6} + \frac{m}{3} -1)  + s\cdot \frac{m}{6} + 1, & \mbox{ if } r=0 \\
\binom{s}{2} ( \frac{(m-1)^2}{6} + 2\frac{m-1}{3}   )  + s\frac{m-1}{6}, & \mbox{ if } r=1 \\
\binom{s}{2} ( \frac{(m-2)^2}{6} + m -2) + s  ( 1 + \frac{m-2}{6}) , & \mbox{ if } r=2 \\
\binom{s}{2} (\frac{(m-3)^2}{6} + 4\frac{m}{3}  -2  )  + s\frac{m-3}{6} + 1, & \mbox{ if } r=3 \\
\binom{s}{2} (\frac{(m-4)^2}{6} +  5\frac{m-1}{3} -2  )   + s(1  + \frac{m-4}{6}), & \mbox{ if } r=4 \\
\binom{s}{2} ( \frac{(m-5)^2}{6}  + 2m - 5 ) + s\frac{m+1}{6} , & \mbox{ if } r=5 \\
\end{array}
\right.
\ee

\end{enumerate}
\end{Corollary}

\begin{Example}
Let $R=k[x_0,\ldots,x_n]$ and $\F=\{F_1,\ldots,F_s\}$ be $s\ge 4$ forms of any degrees such that any 4 of them form a regular sequence. Let $I_3:=I_{3,\F}$ be their star configuration of height 3. Then $I_3^{(28)}$ has
$$\mu(I_3^{(28)})=140\binom{s}{2} + 5s = 70s^2-65s$$ 
minimal generators; its $28$-th symbolic defect is 
$$
\sdef(I_3,28) = 139\binom{s}{2} + 5s.
$$
\end{Example}

Next, we consider the case of small symbolic powers. In \cite[Cor. 3.15]{GGSV} the authors prove that $\sdef(I_{c,\F}, 2) \leq \binom{s}{c-2}$ and equality holds for {\em linear} star configurations. In Corollary \ref{m<5}(1) we remove any restriction on the degrees.

 Recall that if $a$ is a negative integer, then $\binom{t}{a}=0$ for every $t\in \ZZ$.

\begin{Corollary}\label{m<5}(``The cases $m\leq 4$")
Let $R,c,s,\F$ be as in Setting \ref{4.set}, then
\begin{enumerate}
\item $\mu(I_c^{(2)}) = \binom{s+1}{c-1}$ and $\sdef(I_c, 2) = \binom{s}{c-2}$.
\item $\mu(I_c^{(3)}) =\binom{s}{c-1} + (s-c+2)\binom{s}{c-2}+\binom{s}{c-3}$, and 
$$\sdef(I_c, 3) = (s-c+2)\binom{s}{c-2}+\binom{s}{c-3}.$$
\item $\mu(I_c^{(4)}) =\binom{s}{c-1} + \binom{s}{c-2}(s-c+3)+\binom{s}{c-3}\binom{s-c+3}{2} + \binom{s}{c-4}$ and 
$$\sdef(I_c,4)= \binom{s}{c-2}(s-c+3)+\binom{s}{c-3}\binom{s-c+3}{2} + \binom{s}{c-4}.$$
\end{enumerate}
\end{Corollary}

\begin{proof}
The starting point is given by the formulas in Corollary \ref{4.sdef}.
(1) Since $m=2$, the only Diophantine equations having positive solutions are $x_1=2$ and $2x_2=2$. Therefore, the only subsets $B\subseteq [c]$ one can take are $B=\{1\}$ and $B=\{2\}$ and for each of them $|\mathbb S_B|=1$. Then 
\be
\mu(I_c^{(2)}) = \binom{s}{c-1} + \binom{s}{c-2} = \binom{s+1}{c-1},
\ee
and by Theorem \ref{4.sdef} we have $\sdef(I_c,2)=\mu(I_c^{(2)}) - \binom{s}{c-1} = \binom{s}{c-2}.$

(2) Since $m=3$, the only Diophantine equations having positive solutions are $x_1=3$, $x_1+2x_2=3$ and $3x_3=3$. Therefore, the only subsets $B\subseteq [c]$ one can take are $B=\{1\}$ and $B=\{1,2\}$ and $B={\{3\}}$. Since for each of them $|\mathbb S_B|=1$, then
\be
\mu(I_c^{(3)}) = \binom{s}{c-1} + (s-c+2)\binom{s}{c-2}+\binom{s}{c-3}.
\ee
The statements about $\sdef(I_c,3)$ now follow by Theorem \ref{4.sdef}.

(3) is proved similarly. The Diophantine equations having positive solutions are 
\begin{itemize}
\item $x_1=4$ (in this case $B=\{1\}$),
\item $2x_2=4$, whose only solution is $x_2=2$, (in this case $B=\{2\}$)
\item $x_1+2x_2=4$, whose only solution is $x_1=2$ and $x_2=1$ (and $B=\{1,2\}$),
\item $x_1+3x_3=4$, whose only solution is $x_1=1$ and $x_3=1$, (and $B=\{1,3\}$),
\item $4x_4=4$, whose only solution is $x_4=1$ (in this case $B=\{4\}$).
\end{itemize}
Since each of them has a unique positive solution, then $|\mathbb S_B|=1$ for all $B$. The formula now follows from Corollary \ref{4.sdef}.

\end{proof}

\section{Partitions and a total order on the minimal generators of $I_{c,\F}^{(m)}$}

The goal of this technical section is to determine a total order on the generating set $\Gc$ of $I_c^{(m)}$; it will be used in the next section to prove that symbolic powers of  star configurations have {\em c.i. quotients} (see Definition \ref{7.ci}). In particular, symbolic powers of linear star configurations have linear quotients (Corollary \ref{7.linear}) and then their Betti tables have {\em linearly stranded shape}, i.e. they are obtained as the (numerical) union of linear strands which start at each degree of a minimal generator of $I_c^{(m)}$ (Corollary \ref{7.stranded}).\\
\\
To define the total order, we establish a connection between partitions of $m$ and elements in $\Gc$.

Since for every element $M$ in $\Gc$, the monomials $M^{(j)}$ appearing in its normal form have support of size at least $s-c+1$, then $\Sdeg(M^{(j)})=d_j$ for some $d_j\geq 1$, and $M^{(j)}$ is a minimal generator of the star configuration $I_{c-d_j+1}$. 

Therefore, motivated by the upcoming connection with partitions (Corollary \ref{6.surj}), we adopt the following more efficient notation for {\em the normal forms of the elements in $\Gc$}. 

\begin{Notation}\label{5.1}
Let $M\in \Gc$, we write its normal form as 
$$
M =  M_{(d_1)}M_{(d_2)} \cdots M_{(d_t)},
$$
where $\supp(M_{(d_{(j+1)})})\subseteq \supp(M_{(d_{j})})$ for every $j$ and $|\supp(M_{(d_j)})| = s-c+d_j$. 
\end{Notation}

Observe that, by Theorem \ref{4.Symb}, each $M_{(d_{j})}$ is a minimal generator of the star configuration $I_{c-d_j+1}$.

\begin{Example}\label{4.xyzw}
Let $R=k[x,y,z,w]$, $c=3$ and $\F=\{x,y,z,w\}$ (i.e. monomial star configuration). 
Then the normal form of $M=x^2y^2zw$ is 
$
M=\underbrace{(xyzw)}_{M_{(3)}}\underbrace{(xy)}_{M_{(1)}} = M_{(3)}M_{(1)}$. 

In particular, $t=\lambda(M)=2$ and $d_1=3$ and $d_2=1$. Also, $M_{(3)}$ is a minimal generator of $I_{c-3+1}=I_1=(xyzw)$, and $M_{(1)}$ is a minimal generator of $I_{c-1+1}=I_c=I_3$.

On the other hand, the normal form of $N=x^6y^2z^3w^6$ is 
$$
N=\underbrace{(xyzw)}_{N_{(3)}}\underbrace{(xyzw)}_{N_{(3)}}\underbrace{(xzw)}_{N_{(2)}}\underbrace{(xw)}_{N_{(1)}}\underbrace{(xw)}_{N_{(1)}}\underbrace{(xw)}_{N_{(1)}}=N_{(3)}N_{(3)}N_{(2)}N_{(1)}N_{(1)}N_{(1)},
$$
thus for $N$ we have $t=\lambda(N)=6$ and $d_1=d_2=3$, $d_3=2$ and $d_4=d_5=d_6=1$.
\end{Example}

We next observe that monomial orders can be extended to the much general situation of monomials in $\F$, provided that $\F$ allows a unique monomial support:
\begin{definition}
Let $R$ be a polynomial ring over a field and let $\F$ be a set of forms of $R$ that allows a unique monomial support. A {\em monomial order on (the monomials in) $\F$} is a total order on the set of all monomials in $\F$ with $M>1$ for any monomial $M$ in $\F$, and 
 for all monomials $N,N',M$ in $\F$ with $N>N'$, one has $MN>MN'$. 
\end{definition}

We remark that ordinary monomial orders induce monomial orders on $\F$:
\begin{Lemma}\label{5.order}
Let $S=k[y_1,\ldots,y_s]$ and $R=k[x_0,\ldots,x_n]$ be polynomial rings over the same field $k$. Let $\F=\{F_1,\ldots,F_s\}$ be a set of forms in $R$ that allows a unique monomial support. 

Then any monomial order on $k[y_1,\ldots,y_s]$ with $y_{h_1}>\cdots >y_{h_s}$ induces a monomial order on $\F$ with  $F_{h_1}>F_{h_2}>\cdots > F_{h_s}$.
\end{Lemma}

\begin{proof}
Let $\varphi:S\lra R$ be the $k$-algebra map defined by $\varphi(y_i)=F_i$ for $i=1,\ldots,s$. Since $\F$ allows a unique monomial support, for any monomial $M$ in $\F$ there is a unique monomial $M_*$ in $S$ with $\varphi(M_*) = M$.

Then, for any two monomials $M,N$ in $\F$, we set $M>N$ if and only if $M_*>N_*$. It is immediately seen that this defines a monomial order on $\F$. 

\end{proof}

Depending on the settings and objectives, the definition of partitions of integers may or may not allow zero entries. 
For our purpose, we are only interested in {\em partitions with positive entries}. Thus, we denote by $\Po_{\leq c}(m)$ the set of all possible partitions $[\ul{d}]=[d_1,\ldots,d_t] \vdash m$ of $m$, where each entry is positive and at most $c$, and the entries are listed in non-increasing order, i.e. 
$$
\Pc=\left\{[\ul{d}]=[d_1,\ldots,d_t] \vdash m \,\mid\, c\geq d_1\geq d_2\geq \cdots \geq d_t\geq 1 \right\}.
$$
 We call $t:=\lambda([\ul{d}])$ the {\em length} of the partition $[\ul{d}]$. For any $t\in \ZZ_+$, we let 
 $$
 [\Pc]_t:=\{[\ul{d}] \in \Pc\,\mid\,\lambda(\ul{d})=t\}.
 $$
Summarizing the above, for us a given vector $[\ul{d}]=[d_1,\ldots,d_t]$ is a partition of $m$ if and only if 
\begin{itemize}
\item[(i)] $\sum_j d_j =m$;
\item[(ii)] $d_j>0$ for all $j$;
\item[(iii)] $d_j\geq d_{j+1}$ for all $j$.
\end{itemize}

The following important connection between minimal generators of $I_{c,\F}^{(m)}$ and partitions of $m$ now follows immediately from Theorem \ref{4.Symb} and the normal forms of monomials (in Notation \ref{5.1}).

\begin{Corollary}\label{6.surj}
Let $R,c,s,\F$ be as in Setting \ref{4.set}. Then for any $m\geq 1$ 
\begin{enumerate}
\item there exists a surjective function \\
\begin{center}
\begin{tabular}{rcl}
$P:\G_{c,(m)}$ & $\lra$ & $\Pc$ \\
&&\\
$M=M_{(d_1)}\cdots M_{(d_t)} $& $\longmapsto$ & $[d_1,\ldots,d_t]$.
\end{tabular}
\end{center}
where $M=M_{(d_1)}\cdots M_{(d_t)} $ is the normal form of $M$ in Notation \ref{5.1}.\\

\item Further assume all forms in $\F$ have degree $\delta$. Then, for any $\left\lceil \frac{m}{c} \right \rceil\leq t \leq m$, the restriction of $P$ to $\{M\in \Gc \,\mid\, \deg(M)=\delta(t(s-c)+m)\}$ gives a surjection
$$U_t:=\{M\in \Gc \,\mid\, \deg(M)=\delta(t(s-c)+m)\}  \lra [\Pc]_t$$

\end{enumerate}
\end{Corollary}

When $M,M'$ are monomials in $\Gc$ with $P(M)=P(M')$ we say that $M,M'$ {\em have the same associated partition}. 

For instance, if $M\in \Gc$ has normal form $M=(M_{(1)})^m=M_{(1)}M_{(1)}\cdots M_{(1)}$ for some $M_{(1)}\in \G_{c,(1)}$, then its associated partition is $P(M)=[1,1,\ldots,1]$, with $m$ entries. If $M_{(1)}'\neq M_{(1)} \in \G_{c,(1)}$, then $M$ and $M'=(M_{(1)}')^m$ have the same associated partition. Since $|\G_{c,(1)}|=\binom{s}{s-c+1}=\binom{s}{c-1}$, then there are $\binom{s}{c-1}$ distinct monomials in $\Gc$ whose associated partition is $[1,\ldots,1]$. 
\bigskip

Next, we recall the {\em anti-graded lex} total order {\em alex} on the partitions of $m$:
$$
[d_1,\ldots,d_t] >_{\rm alex} [b_1,\ldots,b_u] \Llra t<u, \quad \text{ or } \quad  t=u \text{ and } [\ul{d}]>_{\text{lex}}[\ul{b}].
$$

Thus, if $\lambda(M')<\lambda(M)$, then $P(M')>_{\rm alex}P(M)$. We now extend $\alex$  to a total order $\tau $ on the minimal generators of $I_c^{(m)}$; when two monomials are associated to the same partition, we employ the normal form to break the tie:  
if $M$ and $M'$ are two monomials with $P(M)=P(M')$ and $M=\prod_{j=1}^t M_{(d_j)}$ and $M'=\prod_{j=1}^t M_{(d_j)}'$ are the respective normal forms, we define \\
\begin{center}
$M'\gg_{\revlex} M$ $\iff$ $\exists\,j\geq 1$ such that $M_{(d_i)}' = M_{(d_i)}$ for all $i=1,\ldots,j-1$ and $M_{(d_j)}' >_{\revlex} M_{(d_j)}$.
\end{center}

\begin{Remark}\label{5.rev}
Let $M,M'$ be monomials of the same degree. If $M,M'$ are squarefree, then $M'\gg_{\revlex}M$ if and only if $M' >_{\revlex}M$, i.e. if and only if $M'$ is larger than $M$ in the ``usual" {\em revlex} order.

However, for non-squarefree monomials this is not true anymore. 
For instance, if $M=x_1x_3^2$ and $M'=x_2^2x_3$, then $M<_{\revlex}M'$ however $M\gg_{\revlex} M'$ because the normal forms are $M=(x_1x_3)x_3$ and $M'=(x_2x_3)x_2$ and $x_1x_3>_{\revlex}x_2x_3$. 
\end{Remark}

We can now define the total order on $\Gc$. 
\begin{definition}\label{tau}
Let $R,c,s,\F$ be as in Setting \ref{4.set}. Fix the order $F_1>F_2>\cdots >F_s$ (see Remark \ref{5.order}). Fix $m\geq 1$. We define the following total order $\tau$ on $\Gc$: for any $M,N$ one sets
$$
M >_{\tau} N \Llra \left\{ \begin{array}{l}
P(M)>_{\rm alex}P(N)\\
 \text{or }  \\
 P(M)=P(N) \text{ and } M\gg_{\revlex} N\\
\end{array}
\right.
$$
\end{definition}

\begin{Remark}\label{5.dec}
Let $R,c,s,\F$ be as in Setting \ref{4.set}, and let $>$ be the total order on $\Gc$ of Definition \ref{tau}. For any $M,N\in \Gc$, if $N>M$, then $\deg(N)\leq \deg(M)$.
\end{Remark}

\begin{Example}
Let $R$ be a polynomial ring over a field, $c=6$, $\F=\{F_1,\ldots,F_{10}\}$ a set of forms such that any 7 of them form a regular sequence. 

Let $M_1=(F_1F_2\cdots F_7)^3$, $M_2=(F_1F_2\cdots F_6F_{10})^3$, $M_3=(F_1\cdots  F_5)^4F_6^3F_7F_8$ and $M_4=F_1F_2F_3F_4F_5(F_6\cdots F_{10})^4$.  Then
\begin{itemize}
\item $P(M_1)=P(M_2)=[3,3,3]$, $P(M_3)=[4,2,2,1]$ and $P(M_4)=[6,1,1,1]$;
\item $M_1,M_2,M_3,M_4$ all lie in $\G_{6,(9)}$, because $|P(M_i)|=9$, for $i=1,2,3,4$;
\item $M_2> M_4$, because $\lambda(P(M_2))=3<\lambda(P(M_4))=4$, so$ P(M_2)>_{\alex}P(M_4)$;
\item $M_4> M_3$,  because $\lambda(P(M_3))=\lambda(P(M_4))=4$ and $P(M_4)>_{\alex}P(M_3)$;
\item $M_1> M_2$,  because $P(M_1)=P(M_2)$ and $M_1\gg_{\rm revlex} M_2$.
\end{itemize}

\end{Example}

We conclude this section with a few facts relative to partitions. For any partition $[d_1,\ldots,d_t]\vdash m$ we define the symbol $d_0:=\infty$. It does not change the original partition (because we do not add $d_0$ to the partition) but it ensures that $d_1<d_0$.\\

\begin{Lemma}\label{6.part1}
For any $[d_1,\ldots,d_t]\in \Pc$, let $i\geq 1$ be any index with $d_i<d_{i-1}$. If $i<t$, then there exists a partition  $[b_1,\ldots,b_u]\in \Pc$ with either $u=t$ or $u=t-1$ such that \\
\begin{center}
$[\ul{b}]>_{\rm alex} [\ul{d}]$ and $b_j=d_j$ for all $j\neq i,t$.
\end{center}
\end{Lemma}

\begin{proof}
We define $[b_1,\ldots,b_u]$ as follows: if $d_t=1$, then $u=t-1$; if $d_t>1$, then $u=t$. The entries are $b_j= d_j$ for all $j\neq i,t$, and $b_{i}=d_{i}+1$. If $d_t>1$, we also set $b_{t} = d_{t}-1$. 
It is immediately checked that $[\ul{b}]$ satisfies the conditions (i)--(iii) stated before Corollary \ref{6.surj}, thus $[\ul{b}]\in \Pc$.

\end{proof}

Next, we define a number, the {\em index of overlap}, which is a keystone for our proof and applications of Theorem \ref{6.delta}. It detects the longest initial strand that a partition has in common with a strictly larger partition. 
\begin{definition}\label{6.i0}
Fix a total order $>$ on the elements of $\Pc$. 
For any partition $[\ul{d}]\in \Pc$, the {\em index of overlap} of $[\ul{d}]$ is
$$
\begin{array}{ll}
\max\{j\,\mid\,\exists\,[b_1,\ldots,b_u]\in \Pc & \text{ with }[\ul{b}]> [\ul{d}] \\
& \qquad\text{ and }d_h=b_h \text{ for all }h=1,\ldots,j-1\}.
\end{array}$$

When the total order is $>_{\alex}$ we use the notation
$$
i_0([\ul{d}]) := \max\{j\,\mid\,\exists\,[\ul{b}]\in \Pc \text{ with }[\ul{b}]>_{\alex} [\ul{d}] \text{ and }d_h=b_h \text{ for } 1\leq h \leq j-1\}.
$$
\end{definition}

\begin{Example}
Let  $c=5$, $m=11$ and $\Po=\Po_{\leq 5}(11)$. Then (a) $i_0([5,3,3])=2$ while (b) $i_0([4,4,1,1,1])=3$.
\end{Example}
\begin{proof}
(a) The only partitions in $\Po$ larger than $[5,3,3]$ in $\rm alex$ order are $[5,5,1]$, $[5,4,2]$. So, there is a strictly larger partition with the same first entry as $[\ul{d}]$, e.g. $[5,5,1]$, but there is no strictly larger partition whose first two entries are $[5,3]$, then $i_0([5,3,3])=2$.

(b) In ${\rm alex}$ order, $[4,4,3]$ is strictly larger than $[4,4,1,1,1]$ and shares with it the first two entries, thus $i_0([4,4,1,1,1])\geq 3$. We prove equality. Assume by contradiction there exists a partition $[\ul{d}]\in \Po$ of the form $[\ul{d}]=[4,4,1,d_4,\ldots,d_t]$ with $[\ul{d}]>_{\alex}[4,4,1,1,1]$. Being a partition, it must have $1\leq d_t \leq d_{t-1}\leq \cdots \leq d_4\leq 1$, i.e. $d_j=1$ for all $j\geq 4$. Since the sum of the entries is $m=11$, it follows that $[\ul{d}]=[4,4,1,1,1]$,  contradicting that $[\ul{d}]>_{\alex}[4,4,1,1,1]$.

\end{proof}

We say that a partition $[d_1,\ldots,d_t]$ has a {\em jump} in position $j$ (or at $j$) if $d_j<d_{j-1}$. Since we follow the convention that $d_0=\infty$, then any partition has a jump in the first position.
 
We will make use of the somewhat surprising following fact: once we fix the total order {\em alex}, there is {\em always} a jump in $[\ul{d}]$ at the index of overlap, and, more precisely, {\em $i_0$ is the index of the last jump in the truncated vector $[d_1,\ldots,d_{t-1}]$}. 

\begin{Lemma}\label{6.part2}
For any $[\ul{d}]\in \Pc$, set $i_0=i_0([\ul{d}])$ and $t=\lambda([\ul{d}])$. Then 
\begin{enumerate}
\item $i_0<t$ and $d_{i_0}<d_{i_0-1}$;
\item If $d_t< d_{i_0}$ then $d_{j}=d_{i_0}$ for all $i_0\leq j \leq t-1$. 
\end{enumerate}
Equivalently, $i_0([\ul{d}]) = \max\,\{j=1\ldots,t-1 \,|\,d_j <d_{j-1}\}$.
\end{Lemma}

\begin{proof} (1) We show that $i_0<t$ and $d_{i_0}<d_{i_0-1}$. Assume $i_0=t$, then there exists a partition $[b_1,\ldots,b_u] \in \Pc$ with $[\ul{b}]>_{\rm alex}[\ul{d}]$ and $b_j=d_j$ for all $j=1,\ldots,t-1$. In particular one has $u\geq t-1$. Since $[\ul{b}]>_{\rm alex}[\ul{d}]$, then one also has $u\leq t$. Now, if $u=t-1$, then $m=\sum_{j=1}^{t-1}b_j = \sum_{j=1}^{t-1}d_j$. Since $[\ul{d}]\vdash m$, then this implies $d_t=0$, which is a contradiction. Therefore $u=t$, but since $m=\sum_{j=1}^t b_j = \sum_{j=1}^t d_j$ and $b_j=d_j$ for all $j=1,\ldots,t-1$, then also $b_t=d_t$, thus $[\ul{b}]=[\ul{d}]$, which is a contradiction. 

Next, assume by contradiction that $d_{i_0}=d_{(i_0-1)}$. Let $[\ul{b}] \vdash m$ be a partition in $\Pc$ with $[\ul{b}]>_{\rm alex}[\ul{d}]$   and $d_h=b_h$ for $1\leq h \leq i_0-1$. Since $[\ul{b}]>_{\rm alex}[\ul{d}]$, then $b_{i_0}>d_{i_0}=d_{i_0-1}=b_{i_0 -1 }$, which contradicts the fact that the entries of $[\ul{b}]$ are non-increasing (see part (iii) in the discussion before Corollary \ref{6.surj}).

(2) Let us define $ i_1:=\min\{j\,\mid\,d_j<d_{i_0}\}>i_0$. Assume by contradiction that $i_1 <t$, we can then apply Lemma \ref{6.part1} with $i=i_1$ to obtain a partition $[\ul{b}]$ as in Lemma \ref{6.part1} -- in particular it has $b_j=d_j$ for all $j\neq i_1,t$. Then $[\ul{b}] \in \Pc$, and $[\ul{b}]>_{\rm alex}[\ul{d}]$ and $[\ul{b}]$, and $[\ul{d}]$ agree up to the index $i_1-1 \geq i_0$. This contradicts the definition of $i_0$.

\end{proof}

Lemma \ref{6.part2} makes it very fast to determine $i_0([\ul{d}])$ for any given partition $[d_1,\ldots,d_t]$, see for instance the proof of Theorem \ref{7.strands}(2), or Examples \ref{7.extop} and \ref{7.last}.

\section{A generalization of ideals with linear quotients}
Monomial ideals with linear quotients have been introduced by Herzog and Takayama \cite{HT}. Later this notion was generalized by Conca and Herzog to modules with linear quotients \cite{CH}. In general, ideals (or modules) with linear quotients have been studied  for their nice minimal free resolutions (constructed by iterated mapping cones), see for instance \cite{SZ} or \cite{SV}. As an example of their application, often times one proves that an ideal $I$ has a linear resolution by showing it is generated in a single degree and $I$ has linear quotients (i.e. $I$ is a module with linear quotients).


Here, we propose a generalization of linear quotients:
\begin{definition}\label{7.ci}
Let $R$ be a polynomial ring over a field, $I$ a homogeneous ideal, we say that $I$ has {\em c.i. quotients} if there exists a total order $h_1>h_2>\cdots >h_r$ on a minimal generating set $\{h_1,\ldots,h_r\}$ of $I$ such that for any $1\leq i \leq r-1$ the ideal
$
(h_1,\ldots,h_i):h_{i+1}
$ 
is a complete intersection ideal.  
When this is the case, following Herzog and Takayama, we write $\set(h_{i+1})$ for a minimal generating set of $(h_1,\ldots,h_i):h_{i+1}$. 

If additionally for every $i$ there exists $\delta\in \ZZ_+$ such that all forms of $\set(h_i)$ have the same degree $\delta$, we say that $I$ has {\em $\delta$-c.i. quotients}. 
\end{definition}

When $\delta=1$, the definition recovers the notion of linear quotients: 
\begin{Example}
(1) An ideal $I$ has linear quotients if and only if $I$ has 1-c.i. quotients.\\
(2) Every complete intersection ideal has c.i. quotients. 
\end{Example}

Recall that for any integer $a$ and any graded $R$-module $M$, one denotes by $M(a)$ the module $M$ shifted by $a$, i.e. whose $i$-th graded component is $[M(a)]_i = [M]_{i+a}$. 
A careful study of the proof of \cite[Lemma~1.5]{HT} shows that this proof does not depend on the ideal being monomial neither on the value of $\delta \geq1$ (the key point in their proof is that a certain mapping cone gives the {\em minimal} free resolution of the ideal, and the minimality of such resolution follows by a short arithmetic computation whose validity does not depend on $I$ being monomial or taking $\delta=1$).

Therefore, following the same proof of \cite[Lemma~1.5]{HT} -- i.e. by taking iterated mapping cones -- one obtains the following Eliahou--Kervaire--type formula for the graded Betti numbers of ideals with $\delta$-c.i. quotients.
\begin{Proposition}\label{6.rks}
Let $H=\{h_1,\ldots,h_r\}$ be a minimal generating set of $I\subseteq R$ such that $\deg(h_i)\leq \deg(h_{i+1})$ for every $i$, and assume $I$ has $\delta$-c.i. quotients with respect to $H$. Then 
$$
\beta_{i,j}(R/I) = \sum_{1\leq \ell \leq r,\, \deg(f_{\ell})=j+i-1} \binom{|\set(f_{\ell}|}{i-1}.
$$
\end{Proposition}

For instance, if $I$ has $\delta$-c.i. quotients with respect to $H$, and $h\in H$ has $|\set(h)|=3$, then $h$ contributes to the minimal free resolution of $R/I$ with a copy of $R(-(\delta\deg(h) ))$ in homological degree 1, with $R(-(\delta(\deg(h) + 1)))^3$ in homological degree 2, with $R(-(\delta(\deg(h) + 2)))^3$ in homological degree 3, with a copy of $R(-(\delta(\deg(h) + 3)))$ in homological degree 4, and no contribution in higher homological degrees. 
\bigskip

From Proposition \ref{6.rks} one immediately obtains the following observation:
\begin{Remark}\label{6.set}
Assume $I$ is an ideal with $\delta$-c.i. quotients with respect to a minimal generating set $H$ satisfying $\deg(h_i)\leq \deg(h_{i+1})$ for every $i$. 

Then, knowing $|\set(h)|$ for any $h\in H$ is equivalent to knowing the entire Betti table of $R/I$.
\end{Remark}

We now identify a special type of Betti tables.
\begin{definition}\label{7.Koszul}
Let $R$ be a polynomial ring over a field, $1\leq n_1<n_2<\cdots<n_r$ be integers and $I$ a homogeneous ideal generated in degrees $n_1,\ldots,n_r$. 

We say that $I$ has a {\em Koszul stranded Betti table} if there exists a $\delta \in \ZZ_+$ such that
$$
\beta_{i,j}(R/I)\neq 0 \text{ only if }j \in \{n_h + \delta(i-1)\,\mid\,h=1,\ldots,r\}.
$$
If $\delta=1$, we say that $I$ has a {\em linearly stranded Betti table}.
\end{definition}

So, for instance, if $I$ has a linearly stranded Betti table, then its Betti table appears as the union of linear strands starting at the degrees of the minimal generators of $I$.

Observe also that if $I$ has a Koszul stranded Betti table, then for any $i\geq 1$ there are at most $r$ values of $j$ such that $\beta_{i,j}(R/I)\neq 0$, where $r$ is as in Definition \ref{7.Koszul}. 

From Proposition \ref{6.rks} we immediately obtain the following result. 
\begin{Corollary}\label{6.Koszul}
If $I$ has $\delta$-c.i. quotients then $I$ has a Koszul stranded Betti table.
\end{Corollary}

\begin{Example}\label{6.exa}
Let $I$ be a Cohen--Macaulay ideal of height 3 with $2$-c.i. quotients generated in degrees $16$, $20$ and $24$. Let $J$ be Cohen-Macaulay of height 4 with $3$-c.i. quotients, generated in degrees $12$ and $18$.
(Notice that ideals with these properties exist, see Examples \ref{6.exa2} and \ref{7.exa2}).

Then $R/I$ and $R/J$ have the following Koszul stranded Betti tables
$$
\beta_I:=\begin{array}{r | c c c c}
\beta(R/I)    & 0 & 1 & 2 & 3 \\\hline
 0 & 1 &  &  &  \\
 \vdots &  & &  \\
 15 && \star & & \\
 16 && 	& \star & \\
 17 &&  	& 	& \star \\
  18 &&  	& 	& \\
  19 && \star & & \\
 20 && 	& \star & \\
 21 &&  	& 	& \star \\
 22 &&  	& 	& \\
 23 && \star & & \\
 24 && 	& \star & \\
 25 &&  	& 	& \star \\
\end{array}
\qquad \text{ and }\qquad \beta_J:=\begin{array}{r | c c c c c}
\beta(R/J)    & 0 & 1 & 2 & 3 & 4\\\hline
 0 & 1 &  &  & &   \\
 \vdots &  & & & &   \\
11 && \star & &  & \\
12 &&&&&\\
13 && &  \star & &  \\
14 &&&&&\\
15 && &  & \star &  \\
16 &&&&&\\
17 &&\star && &  \star\\
18 &&&&&\\
19 && &  \star & &  \\
20 &&&&&\\
21 && &  & \star &  \\
22 &&&&&\\
23 &&&& &  \star\\
\end{array}
$$
where the stars mark the only possible non-zero entries in each Betti table.
\end{Example}

Before we prove the main result of this section, we fix some notation and prove a couple of auxiliary results. Let $R=k[x_0,\ldots,x_n]$ and $\F$ be a set of forms in $R$; for any subset $B\subseteq \F$ we let $B^C := \F\setminus B$ and define the ideal $\a_B$ as
$$
\a_B := (F\in \F\,\mid\, F \in B) \subseteq R.
$$
\begin{Lemma}\label{6.colon}
Let $R$ be a polynomial ring over a field, let $\F$ be a set of forms such that any two of them form a regular sequence. 
For any two monomials $M,N$ in $\F$, if $\supp(N)\not \subseteq \supp(M)$, then $N:_RM \subseteq \bigcap_{F\in \supp(N)\setminus \supp(M)} (F)$. 

In particular, $N:_RM \subseteq  \a_{\supp(N)\setminus \supp(M)} \subseteq \a_{\supp(M)^C} \cap \a_{\supp(N)}$.
\end{Lemma}

\begin{proof}
Let $F\in \supp(N)\setminus \supp(M)$. Since $\F$ allows a unique monomial support (by Proposition \ref{gcd}), then in particular given any two monomials in $\F$ there is a unique monomial in $\F$ of largest degree dividing both monomials -- we denote it the $\gcd$ of the two monomials. 
Let $O:=\gcd(M,N)$, write $M=OM'$ and $N=ON'$, then by cancellation $N:M = N':M'$. By assumption, we have $F\in \supp(N')$ and $F\notin \supp(M')$. By assumption on $\F$, it follows that $\gcd(F,M')=1$, thus $\grade(F,M')=2$. Therefore
$$
N:M = N':M' \subseteq (F):M' = (F).
$$

\end{proof}


\begin{Lemma}\label{6.supp0}
Let $R,c,s,\F$ be as in Setting \ref{4.set}. 
Let $M,M'$ be monomials in $\Gc$ with the same associated partition $P(M')=P(M)=[d_1,\ldots,d_t]$. 
Let $M=\prod_{j=1}^t M_{(d_{j})}$ and $M' = \prod_{j=1}^t M_{(d_j)}'$ be their respective normal forms in Notation \ref{5.1}. 

\begin{enumerate}
\item If $h_0$ is the largest integer $h$ such that $M_{(d_j)}'=M_{(d_j)}$ for all $j<h$, then 
$$
M' : M \subseteq \a_{\supp(M_{d_{h_0}})^C}.
$$
\item If $\supp(M_{(d_i)})=\supp(M_{(d_i)}')$, then for every $j$ with $d_j=d_i$ we have $\supp(M_{(d_j)})=\supp(M_{(d_j)}')=\supp(M_{(d_i)})$.
\end{enumerate}
\end{Lemma}

\begin{proof} 
(1) By cancellation we have 
$M':M = N':N$, where  $N=\prod_{j\geq h_0} M_{(d_j)}$ and $N' = \prod_{j\geq h_0} M_{(d_j)}'$.
Since $\supp(N)=\supp(M_{(d_{h_0})})$ and $\supp(N')=\supp(M_{(d_{h_0})}')$ are distinct squarefree monomials (by definition of $h_0$) and $|\supp(N)|=|\supp(N')|$ (because $P(N)=P(N')$), then
$\supp(N')\not\subseteq \supp(N)$.
The statement now follows by Lemma \ref{6.colon}.

(2) The assumption that $d_j=d_i$ implies that $\supp(M_{(d_j)})=\supp(M_{(d_i)})$ and $\supp(M_{(d_j)}') = \supp(M_{(d_i)}')$. 

\end{proof}

We prove a first useful inclusion used in Theorem \ref{6.delta}.
\begin{Proposition}\label{6.supp}
Let $R,c,s,\F$ be as in Setting \ref{4.set}. Let $[\ul{b}]$ and $[\ul{d}]$ be two partitions in $\Pc$ with $[b_1,\ldots,b_u]>_{\rm alex}[d_1,\ldots,d_t]$. Let $i$ be the integer such that $b_j=d_j$ for all $j=1,\ldots,i-1$ and $b_i>d_i$. 

For any $M\in \Gc$ with $P(M)=[\ul{d}]$ one has
$$
\left(N \in \Gc\,\mid\, P(N)=[\ul{b}]\right) : M \subseteq \a_{\supp(M_{d_{i}})^C}.
$$

\end{Proposition}

\begin{proof}
First, since $P(M)=[d_1,\ldots,d_t]$, then the normal form of $M$ (in Notation \ref{5.1}) is $M=M_{(d_1)}\cdots M_{(d_t)}$. 
To prove the statement it suffices to prove the inclusion locally at every associated prime of $\a:=\a_{\supp(M_{(d_i)})^C}$. So, let $P\in \Ass(R/\a)$ and let $\F' =\supp(M_{(d_i)})^C$.  Since $|\F'|=|\supp(M_{(d_i)})^C| = c-d_i<c$, then by the assumption on $\F$ the ideal $\a=\a_{\F'}$ is a complete intersection of height $c-d_i$, thus in particular $\h(P)=c-d_i$. 
For any $j\geq i$ and any $F\in \supp(M_{(d_j)})$, by assumption on $\F$ we have $\h(\a+(F))=c-d_i+1$, then in particular $F$ is regular on $R/P$. This has the following important consequences:
\begin{enumerate}
\item the monomial $M_{(d_j)}$ is a unit in $R_P$ for any $j\geq i$ and then $MR_P = M'R_P$ where $M'=\prod_{j=1}^{i-1}M_{(d_j)}$. Since $\F'$ is a regular sequence and $\a=\a_{\F'}$, then $M' \in \a_P^{o(M')} \setminus \a_P^{o(M')+1}$ where $o(M'):=\sum_{j=1}^{i-1}(d_j-d_i)$.

\item For any positive integer $b\leq d_i$ one has $(I_{c-b+1})_P = R_P$;

\item For any positive integer $b>d_i$ one has $(I_{c-b+1})_P = \left(I_{c-b+1,\F'}\right)_P$, where $I_{c-b+1,\F'}$ is the star configuration of height $c-b+1$ in the forms $\F'$.
Since for any minimal generator $N$ of $I_{c-b+1,\F'}$ we have $N$ is squarefree with 
$|\supp(N)| = |\F'| - (c-b+1) + 1 =  b-d_i$, we conclude that for any $b>d_i$ one has $$(I_{c-b+1})_P \subseteq (\a^{b-d_i - 1})_P.$$

\end{enumerate}

Since for any monomial $N\in\left(N \in \Gc\,\mid\, P(N)=[\ul{b}]\right)$ and any $j$ the squarefree monomial $N_{(b_j)}$ lies in $I_{c-b_j+1}$, then $\left(N \in \Gc\,\mid\, P(N)=[\ul{b}]\right) \subseteq \left( \prod_{j=1}^u I_{c-b_j+1}\right)$. Thus,
$$
\begin{array}{lll}
\left(N \in \Gc\,\mid\, P(N)=[\ul{b}]\right)_P :_{R_P} M & \subseteq \left( \prod_{j=1}^u \left(I_{c-b_j+1}\right)_P\right) :_{R_P}M'  & \\
& \subseteq \left(\prod_{j=1}^i \left(I_{c-b_j+1}\right)_P\right) :_{R_P}M' & \\
& \subseteq \left(\prod_{j=1}^i \left(\a^{b_j-d_i+1}\right)_P\right) :_{R_P}M'  & \\
& =  \left(\a^{\sum_{j=1}^i (b_j-d_i+1)}\right)_P :_{R_P}M' ,
\end{array}
$$
where the last inclusion holds by (3) above, because $b_j>d_i$ for $1\leq j \leq i$.

Since $\a_P$ is a complete intersection ideal in $R_P$, then the associated graded ring $G={\rm gr}_{\a_P}(R_P)$ is isomorphic to a polynomial ring in $c-d_i$ variables, and therefore
$$
\left(\a^{\sum_{j=1}^i (b_j-d_i+1)}\right)_P :_{R_P}M' = \a_P^{\left(\sum_{j=1}^i (b_j-d_i+1) \right)-o_{G}(M')},
$$
where $o_G(M')$ is the order of $M'$ in $G$. By point (1) above, we have $o_G(M')= \sum_{j=1}^{i-1}(d_j-d_i)$. Then
$$
\left(\sum_{j=1}^i (b_j-d_i) \right) - o_G(M') = \left(\sum_{j=1}^i (b_j-d_i) \right) -  \sum_{j=1}^{i-1}(d_j-d_i) = b_i-d_i
$$
where the rightmost equality follows because $b_j=d_j$ for $j=1,\ldots,i-1$ by assumption. Since $b_i>d_i$ by assumption, we conclude that 
$$
\left(N \in \Gc\,\mid\, P(N)=[\ul{b}]\right)_P :_{R_P} M  \subseteq \a_P^{b_i-d_i}\subseteq \a_P,
$$
finishing the proof.

\end{proof}

\begin{Proposition}\label{6.mon}
Let $R,c,s,\F$ be as in Setting \ref{4.set}. Let $M,N_1,\ldots,N_r$ be monomials in $\F$ and set $H=(N_1,\ldots,N_r)$. Let $B\subseteq \F$ be any subset of at most $c$ elements of $\F$, and $\a:=\a_B=(F\,\mid\,F\in B)$.

Then $H:M \subseteq \a$ if and only if $(N_i:M)\subseteq \a$ for every $i=1,\ldots,r$.
\end{Proposition}

\begin{proof}
Assume $N_i:M \subseteq \a$ for every $i$. We prove $H:M \subseteq \a$ by showing the inclusion holds locally at every $P\in \Ass(R/\a)$. 

Fix $P\in \Ass(R/\a)$ and let $S=R/\a$. By assumption on $\F$, the elements $\{F\,\mid\,F\in B\}$ form a regular sequence of height at most $c$ and $P$ is an associated prime of $I_{|B|,\F}$. 
By Proposition \ref{2.ghmn}(2) applied to $I_{|B|,\F}$, for any $F\in B^C$ we have $F\notin P$. Now, for any monomial $N$ let $N'$ be the monomial obtained from $N$ by replacing the elements $F\in B^C \cap \supp(N)$ by $1$. Then $HS_P: MS_P = (N_1',\ldots,N_r')S_P : M'S_P=\left((N_1',\ldots,N_r'):M'\right)S_P$. 

Now, $N_1',\ldots,N_r', M'$ are monomials in the regular sequence $\{F\,\mid\,F\in B\}$, thus  $k[F\,\mid\,F\in B]\cong k[y_F\,\mid\,F\in B]$ is isomorphic to a polynomial ring in $|B|$ variables and then $(N_1',\ldots,N_r'):M' = \sum_{j=1}^r (N_j':M')$. Therefore
$$
HS_P: MS_P =\left((N_1',\ldots,N_r'):M'\right)S_P = \left(\sum_{j=1}^r (N_j':M')\right)S_P \subseteq \a S_P.
$$

\end{proof}

Let $R,c,s,\F$ be as in Setting \ref{4.set}. For any subset $B=\{b_1,\ldots,b_h\}\subseteq \F$ with $b_1<b_2<\cdots <b_h$ and any monomial $M$ in $B$, we set
$$
m_B(M) :=\max\{j\in \{1,\ldots,h\},\mid\,F_{b_j}\mbox{ divides }M\}.
$$
If $B=\F$, we simply write $m(M)$. 

\begin{Proposition}\label{6.m=1}
Let $R,c,s,\F$ be as in Setting \ref{4.set}. Then $I_{c,\F}$ has c.i. quotients with respect to the order $\tau$ of Definition \ref{tau}. 

Additionally, for any $M\in \G_{c,(1)}$ one has
$$
|\set(M)| = m(M) - s+c -1.
$$
If moreover all forms of $\F$ have the same degree $\delta\geq 1$, then $I_{c,\F}$ has $\delta$-c.i. quotients.
\end{Proposition}

\begin{proof}
Let $>$ be the order $\tau$ on $\G_{c,(1)}$ defined in Definition \ref{tau}. Both statements follow if we show that for any $M\in \G_{c,(1)}$ one has 
$$
(N\in \G_{c,(1)}\,\mid\,N>M) : M = (F_j\in \supp(M)^C\,\mid\, j<m(M)).
$$
``$\supseteq$": For any $F_j$ in $\supp(M)^C$, let $N_j$ be the monomial $N_j:=\frac{F_jM}{F_{m(M)}}$, then $F_jM\in (N_j)$. By Theorem \ref{4.Symb} one has $N_j\in \G_{c,(1)}$ and $N_j$ and $M$ are squarefree monomials. If $j<m(M)$, then $m(N_j)<m(M)$, thus $N_j\gg_{\revlex}M$ by Remark \ref{5.rev} and then $N_j>M$. This proves the inclusion.

``$\subseteq$": Assume $N\in \G_{c,(1)}$ with $N>M$. Since $N\neq M$ are squarefree with $|\supp(N)|=|\supp(M)|=s-c+1$, then $\supp(N)\not\subseteq \supp(M)$; additionally, since $N>M$, then by Remark \ref{5.rev} $N>_{\revlex} M$, thus $\supp(N)\setminus \supp(M)\subseteq \{F_j\in \supp(M)^C\,\mid\, j<m(M)\}$. The statement now follows by Lemma \ref{6.colon}.

\end{proof}

For any $M\in \Gc$, we let $M=M_{(d_1)}\cdots M_{(d_t)}$ be its normal form. We write $i_0(M):=i_0(P(M))$ for the index of overlap of $P(M)$ (see Definition \ref{6.i0}), and define the ideal
$$
\a_0(M):= \a_{\supp(M_{(d_{i_0})})^C}=(F\in \F\,\mid\, F\notin \supp(M_{(d_{i_0})})),\; \text{ where }i_0:=i_0(M).
$$
Since $\a_0(M)$ is generated by $c-d_{i_0}$ elements of $\F$, then (by assumption on $\F$) the ideal $\a_0(M)$ is a complete intersection.

We are now ready to prove the main result of this section: symbolic powers of star configurations of hypersurfaces have c.i. quotients. 

\begin{Theorem}\label{6.delta}
Let $R,c,s,\F$ be as in Setting \ref{4.set}; then for any $m\geq 1$ the ideal $I_{c,\F}^{(m)}$ has c.i. quotients with respect to the order $\tau$ of Definition \ref{tau}.

 Additionally, for any $M\in \Gc$, let $P(M)=[d_1,\ldots,d_t]$ and set $i_0:=i_0(M)$. Then 
\begin{enumerate}
\item $|\set(M)| = m(M_{(r)}) - s+c-r $, if $P(M)=[c,\ldots,c,r]$ is the maximal element in $\Pc$;\\ 
\item $|\set(M)| = c-d_{i_0}$, if $d_t=d_{i_0}$;\\
\item $|\set(M)| = c-d_{i_0} + m_{\supp(M_{(d_{i_0})})} (M_{(d_t)}) - (s - c + d_t)$, if $d_t<d_{i_0}$ and $P(M)\neq[c,\ldots,c,r]$ is not maximal in $\Pc$.
\end{enumerate}

 Moreover, if all forms in $\F$ have the same degree $\delta$, then $I_{c,\F}^{(m)}$ has $\delta$-c.i. quotients.
\end{Theorem}

%

\begin{proof}
Let $>$ denote the total order $\tau$ on $\Gc$ defined in Definition \ref{tau}; we use it to order the minimal generators of $\Gc$, and write them as
$$\Gc=\{N_1,\ldots,N_r\}, \text{ where } N_1> N_2> \cdots > N_r.
$$
For any fixed $M\in \Gc$ we find an explicit minimal generating set $\set(M)$ for the colon ideal $(N\in \Gc\,|\,N>M):M$, see Claim 3 below. We then remark that $\set(M)$ is a regular sequence for any $M$, which is not obvious {\em a priori}, because the elements of $\F$ do not necessarily form a regular sequence.

Also, before starting the proof, we notice that in part (2) of the statement we don't need to add the assumption ``if $P(M)$ is not maximal in $\Pc$", because if $P(M)$ is maximal in $\Pc$ and $d_t=d_{i_0}$, then $P(M)=[c,\ldots,c]$, in which case $|\set(M)|=0$ and formulas (1) and (2) agree.
\\ 
\\
{\bf Claim 1.} Let $M\in \Gc$ and $\a_{\supp(M_{(d_{i_0}) })^C}=:\a_0(M)$. If $P(M)$ is not maximal in $\Pc$ in the alex order, then 
$$
\left( N\in \Gc\,\mid\,P(N)>_{\rm alex}P(M) \right) : M = \a_0(M).
$$

``$\subseteq$": let $N\in \Gc$ with $P(N):=[b_1,\ldots,b_u]>_{\rm alex}P(M)$.  By Proposition \ref{6.supp}, we have $N:M \subseteq  \a_{\supp(M_{(d_i)})^C}$ where $i$ is the integer such that $b_j = d_j$ for $j=1,\ldots,i-1$ and $b_{i}>d_{i}$. By definition of $i_0(M)$ we have $i\leq i_0(M)$, therefore $\a_{\supp(M_{(d_i)})^C} \subseteq \a_{\supp(M_{(d_{i_0}) })^C}=\a_0(M)$.

``$\supseteq$": Let $F\in \supp(M_{(d_{i_0}) })^C$ and let $j_0 := \min\{j\,\mid\, F\in \supp(M_{(d_j)})^C \}\geq i_0$. Observe that, if $j_0>1$, then $\supp(M_{(d_{j_0}) }) \neq \supp(M_{(d_{(j_0-1)}) })$, thus $d_{j_0}<d_{(j_0-1)}$. Since $j_0\leq i_0<t$ (by Lemma \ref{6.part2}), then by Lemma \ref{6.part1} there exists a partition $[\ul{b}] \in \Pc$ with $b_j=d_j$ for all $j\neq j_0,t$, $[\ul{b}]>_{\rm alex}[\ul{d}]$, and $\lambda([\ul{b}])\in \{t-1,t\}$. 

By Corollary \ref{6.surj} there exists $N' \in \Gc$ with $P(N')=[\ul{b}]$. Since any permutation on $\F$ applied to $N'$ yields element $N\in \Gc$ with $P(N)=[\ul{b}]$, we can take $N$ such that in its normal form we have $\supp(N_{(b_j)})=\supp(M_{(d_j)})$ for all $j\neq j_0,t$, and $\supp(N_{(b_{j_0}) }) = \{F\} \cup \supp(M_{(d_{j_0}) })$, and $N_{(b_t)}=\frac{M_{(d_t)}}{G}$ for some $G\in \supp(M_{(d_t)})$ (the choice of $G$ is irrelevant). 

By construction, we have $FM=GN$, thus $F\in (N:M)$, concluding the proof of Claim 1.\\
\\
 {\bf Claim 2.} Let $M,M'\in \Gc$ with $P(M')=P(M)=[d_1,\ldots,d_t]$ and $M'\gg_{\revlex}M$. Let $M=\prod_{j=1}^tM_{(d_j)} $ and $M'=\prod_{j=1}^t M_{(d_j)}'$ be their normal forms. We define
  $$
 h_0=h_0(M,M') := \max\{h\,\mid\, M_{(d_j)}' = M_{(d_j)} \text{ for all }j<h \}.
 $$
 \begin{itemize}
 \item[(a)] If $h_0\leq i_0(M)$, then 
$M':  M \subseteq \a_0(M)$;
\item[(b)] If $h_0> i_0(M)$, then 
$M':  M =M_{(d_t)}' : M_{(d_t)}.$\\
\end{itemize}

\noindent Let $i_0:=i_0(M)$. (a) follows by Lemma \ref{6.supp0}(1) and $\supp(M_{(d_{h_0}) })^C \subseteq \supp(M_{(d_{i_0}) })^C$. 
(b) Since $h_0>i_0$ and $M'\neq M$, then Lemma \ref{6.supp0}(2) implies $d_t<d_{i_0}$. By Lemma \ref{6.part2}(2) one has
 $\supp(M_{(d_j)}) = \supp(M_{(d_{i_0}) })$ for every $i_0\leq j \leq t-1$, and Lemma \ref{6.supp0}(2) implies $\supp(M_{(d_j)}') = \supp(M_{(d_j)})$ for every $i_0\leq j \leq t-1$. Since $i_0<h_0$, then we also know that $\supp(M_{(d_j)}') = \supp(M_{(d_j)})$ for every $j\leq i_0$. Therefore, if we set $L=\prod_{j=1}^{t-1}M_{(d_j)}$, we have
$$ M = L M_{(d_t)} \qquad \text{ and }M' = L M_{(d_t)}'.$$
The statement now follows by cancellation of $L$ in the colon. This concludes the proof of Claim 2.
\\
\\
{\bf Claim 3.} Let $M\in \Gc$.
\begin{enumerate}
\item If $P(M)$ is maximal in $\Pc$, then $P(M)=[c,c,\ldots,c,r]$ where $r$ is an integer with $1\leq r\leq c$, and $\set(M)=\set(M_{(r)})$ (since $M_r$ is a minimal generator of the star configuration $I_{c-r+1}$, then $\set(M_{(r)})$ is well-defined).

\item Assume $P(M)$ is not maximal in $\Pc$.
\begin{itemize}
\item[(a)]  If $d_t=d_{i_0}$, then $\set(M)=\supp(M_{(d_{i_0}) })^C.$
\item[(b)] If $d_t<d_{i_0}$, then $$\set(M)=\supp(M_{(d_{i_0}) })^C \cup \set_{ \supp\left(M_{(d_{i_0}) }\right)}(M_{(d_t)}),$$
 where $ \set_{ \supp(M_{(d_{i_0})})}(M_{(d_t)})$ is $\set(M_{(d_t)})$ when we consider $M_{(d_t)}$ as a minimal generator of the star configuration
   $I_{d_{i_0}-d_t+1,  \supp(M_{(d_{i_0}) })}$. \\ 
  \end{itemize}
\end{enumerate}

(1) Write $m=qc+r$ with $1\leq r\leq c$, then the alex order implies that $P(M)=[c,c,\ldots,c,r]$, thus every $M'$ with $M'>M$ has $P(M')=P(M)$ and thus $M'=(F_1\cdots F_s)^q M_{(r)}'$ for some $M_{(r)}'$ with $|\supp(M_{(r)}')|=s-c+r$. Then, $M':M = M_{(r)}' : M_{(r)}$ by cancellation. This proves that
{\small$$
(M'\in \Gc\,\mid\,M'>M): M = (N'\in \G_{r,(1)}\,\mid\, N'>_{\revlex}M_{(r)}) : M_{(r)}= (\set(M_{(r)})).
$$}

(2) (a) The inclusion ``$\supseteq$" follows from Claim 1. For the other inclusion, it suffices to show that
$(M'\in \Gc\,\mid\, M'>_{\tau} M ) : M \subseteq \a_0(M)$. By Proposition \ref{6.mon} it suffices to prove $M':M \in \a_0(M)$ for any $M'>_{\tau}M$. If  $P(M')>_{\alex}P(M)$, the inclusion follows by Claim 1. If $P(M')=P(M)$, we look at the number $h_0=h_0(M,M')$ defined in Claim 2. If $h_0\leq i_0(M)$, the inclusion follows by Claim 2(a). If $h_0>i_0(M)$, then $M':M =M_{(d_t)}':M_{(d_t)} \subseteq \a_{\supp(M_{(d_t)}')}$ by Lemma \ref{6.colon}. The statement follows because $\supp(M_{(d_t)}')\subseteq \supp(M_{(d_{i_0})}')=\supp(M_{(d_{i_0})})$, where the equality holds because $h_0 \geq i_0$.

(b) We need to show $$(M'\in \Gc\,\mid\,M'>M):M = \a_0(M) + (F\in \set_{ \supp(M_{(d_{i_0})} )}(M_{(d_t)})).$$
``$\supseteq$" By Claim 1 one has $\a_0(M)\subseteq (M'\in \Gc\,\mid\,M'>M):M$. Let $F\in \set_{ \supp(M_{(d_{i_0})} )}(M_{(d_t)})$, by the revlex order, there exists $G<F$ with $G \in\supp(M_{(d_t)})$, so we let $N':=(MG)/F$. It follows that $N'$ is a minimal monomial generator of $ I_{d_{i_0}-d_t+1,\supp(M_{(d_{i_0}) } ) }$ and $F\in N':M_{(d_t)}$. By definition of $N'$ and Lemma \ref{6.part2}(2) one has $\supp(N')\subseteq \supp(M_{(d_{j})})$ for all $i_0\leq j \leq t-1$. Therefore, $M':=\left(\prod_{j=1}^{t-1}M_{(d_j)}\right)N'$ is its own normal form, and clearly one has $P(M')=P(M)$ and $M'>_{\tau}M$. Finally, by construction, $F\in (M':M)=N':M_{(d_t)}$.

``$\subseteq$"  By Proposition \ref{6.m=1} $|\set_{ \supp(M_{(d_{i_0}) })}(M_{(d_t)})| \leq m_{ \supp(M_{(d_{i_0}) })}(M_{(d_{i_0})}) - (s-c+d_{i_0}) + (d_{i_0}-d_t + 1) - 1\leq (s-c+d_{i_0})  - s + c -d_t = d_{i_0}-d_t$. Then $| \supp(M_{(d_{i_0}) })^C \cup \set_{ \supp(M_{(d_{i_0}) })}(M_{(d_t)})| \leq c-d_{i_0} + d_{i_0}-d_t=c-d_t<c$. Then, by Proposition \ref{6.mon}, we only need to show that $M':M \subseteq \a_0(M) + (F\in \set_{ \supp(M_{(d_{i_0})})}(M_{(d_t)}))$.

If $P(M')>_{\alex}P(M)$, or if $P(M')=P(M)$ and $h_0(M,M')\leq i_0(M)$, then the inclusion holds by Claims 1 and 2(a), respectively. Assume then $P(M')=P(M)$ and $h_0(M,M') >i_0$. By Lemma \ref{6.part2}(2) we have $h_0=t$, thus $ M' = LM_{(d_t)}'$ and $ M = LM_{(d_t)}$ where $L=\prod_{j=1}^{t-1}M_{(d_j)}$, and by the property of normal forms,  $\supp(M_{(d_t)}')\subseteq \supp(M_{(d_{i_0})}')=\supp(M_{(d_{i_0}) })$.

Observe that  $M_{(d_t)}'>_{\rm revlex} M_{(d_t)}$ (since $M'\gg_{\rm revlex} M$) and, by cancellation, $M':M=M_{(d_t)}' : M_{(d_t)}$, thus $M':M=M_{(d_t)}' : M_{(d_t)}\subseteq(F\in \supp(M_{(d_{i_0}) })\,\mid\,FM_{(d_t)}\in (M_{(d_t)}')) \subseteq (\set_{ \supp(M_{(d_{i_0}) })}(M_{(d_t)}))$. This proves Claim 3.\\

Now, the statement of the theorem follows by Claim 3 and Proposition \ref{6.m=1} if we prove that $\set(M)$ is a regular sequence for every $M\in \Gc$. By assumption on $\F$, it suffices to show that $|\set(M)|< c$ for any $M\in \Gc$. Indeed, if $P(M)=[c,\ldots,c,r]$, then $1\leq r \leq c$ and $m(M_{(r)})\leq s$, thus $|\set(M)|\leq c-r<c$. If $P(M)=[d_1,\ldots,d_t]$ and $d_t=d_{i_0}$, then $|\set(M)|=c-d_{i_0}<c$. 
If $P(M)=[d_1,\ldots,d_t]$ is not maximal and $d_t<d_{i_0}$, then 
$m_{\supp(M_{(d_{i_0})})} (M_{(d_t)}) \leq |\supp(M_{(d_{i_0})})|=s-c+d_{i_0}$, thus $|\set(M)|\leq c-d_t<c$.

\end{proof}

Recall that a star configuration is {\em linear} if $\deg(F_j)=1$ for every $j=1,\ldots,s$.

\begin{Corollary}\label{7.linear}
If $I_{c,\F}$ is a linear star configuration, then for every $m\geq 1$ the ideal $I_{c,\F}^{(m)}$ has linear quotients. In particular, monomial star configurations have linear quotients.
\end{Corollary}

\begin{Example}\label{6.exa2}
(1) Let $\F=\{F_1,\ldots,F_5\}$ be five quadrics such that any $4$ of them form a regular sequence, and set $I'=I_{3,\F}^{(4)}$. Then $I'$ has $2$-c.i. quotients.\\

(2) Let $\mathcal G=\{G_1,\ldots,G_6\}$ be six cubics such that any $5$ of them form a regular sequence. Set $J'=I_{4,\mathcal G}^{(2)}$. Then $J'$ has $3$-c.i. quotients.\\
\end{Example}

In the next section, we employ Theorem \ref{6.delta} to determine a formula for {\em every} graded Betti number of $R/I_{c,\F}^{(m)}$, see Theorem \ref{7.formula}.

\section{The Betti table of $R/I_{c,\F}^{(m)}$}

In this section we determine the graded Betti numbers of $R/I_{c,\F}^{(m)}$. By Remark \ref{6.set}, we need to know $|\set(M)|$ for every $M\in \Gc$, and since Theorem \ref{6.delta} gives an explicit formula for $|\set(M)|$, our first main result is a complete formula for the Betti tables of any symbolic power of any star configuration, see Theorem \ref{7.formula}. 

To simplify the computations, in this section we also provide a number of {\em closed} formulas {\em solely in terms of $s,c,m$}. The second main result of this section is Theorem \ref{7.strands} where we provide such a closed formula for $\left \lceil\frac{m}{2}\right\rceil+1$ of the Koszul strands in the Betti table of $R/I_{c,\F}^{(m)}$. 
Since in total there are $m- \left\lceil\frac{m}{c}\right\rceil$ Koszul strands, then Theorem \ref{7.strands} provides a closed formula for more than half of the Betti table of $R/I_{c,\F}^{(m)}$. The remaining ones can often be computed directly (at least for relatively small $c$ and $m$), see Corollary \ref{7.small} and 
Examples \ref{7.7} and \ref{7.last}. 

On the other hand, already a closed formula -- depending only on $s,c,m$ -- for the very top strand of the Betti table of $R/I_{c,\F}^{(m)}$ appears to be extremely complicated, see Example \ref{7.extop} -- see also Proposition \ref{7.top} and Corollary \ref{7.c=4} for partial results.
\bigskip

First, from Theorem \ref{6.delta} we obtain the shape of the Betti table of $R/I_{c,\F}^{(m)}$.
\begin{Corollary}\label{7.stranded}
Let $R,c,s,\F$ be as in Setting \ref{4.set} and assume $\F$ consists of forms of the same degree $\delta$.  Then $R/I_{c,\F}^{(m)}$ has a Koszul stranded Betti table for any $m\geq 1$.

If additionally $I_{c,\F}$ is a linear star configuration, then $R/I_{c,\F}^{(m)}$ has a linearly stranded Betti table.
\end{Corollary}

\begin{Example}\label{7.exa2}
Let $I'$ and $J'$ be the ideals of Example \ref{6.exa2}. Then the Betti table of $R/I'$ has the form $\beta_I$ of Example \ref{6.exa}, and the Betti table of $R/J'$ has the form $\beta_J$ showed in Example \ref{6.exa}. 
\end{Example}

One can then determine immediately the Castelnuovo--Mumford regularity of these ideals.

\begin{Corollary}\label{6.reg}
Let $R,c,s,\F$ be as in Setting \ref{4.set} and assume $\F$ consists of forms of the same degree $\delta$.  
Then $${\rm reg}(R/I_{c,\F}^{(m)})=\delta m (s-c + 1) + (c-1)(\delta -1) -1.$$
\end{Corollary}

Recall that for a partition $[\ul{d}] \in \Pc$, the number $i_0$ denotes the index of overlap. We can now prove the explicit formula for the graded Betti numbers of $R/I_{c,\F}^{(m)}$.
\begin{Theorem}\label{7.formula}[The graded Betti numbers of the symbolic powers of star configurations]
Let $R,c,s,\F$ be as in Setting \ref{4.set} and assume the forms in $\F$ all have degree $\delta$. Fix $m>1$.

For any fixed integer $1\leq i \leq c$, for any $\left\lceil\frac{m}{c}\right\rceil< t \leq m$ let $j_i:=\delta(t(s-c)+m+i-1)$, then
$$
\begin{array}{lcl}
\beta_{i,j_i}(R/I_{c,\F}^{(m)}) &= &\sum\limits_{[\ul{d}]\in [\Pc]_t,\,d_t=d_{i_0}} \binom{c-d_{i_0}}{i-1}\binom{s}{s-c+d_1}\binom{s-c+d_1}{s-c+d_2}\cdots \binom{s-c+d_{(i_0-1)}}{s-c+d_{i_0}}\\
&&\\
 &+ &\sum\limits_{[\ul{d}]\in [\Pc]_t\setminus\{[c,\ldots,c,r]\},\,d_t<d_{i_0}} \binom{s}{s-c+d_1} \cdots \binom{s-c+d_{(i_0-1)}}{s-c+d_{i_0}} A_{i-1}([\ul{d}])\\
 &&\\
 & + & \delta_{t,\left\lceil\frac{m}{c}\right\rceil} \binom{s}{c-r+1-i}\binom{s-c+r+i-2}{i-1}.
 \end{array}
$$
where $A_{i-1}([\ul{d}]):=\sum\limits_{j=i-1}^{d_{i_0}-d_t}\binom{c-d_{i_0}+j}{i-1}\binom{s-c+d_t+j-1}{j}$, and $\delta_{u,v}$ is Kronecker's delta (i.e. the last summand $ \binom{s}{c-r+1-i}\binom{s-c+r+i-2}{i-1}$ only appears when $t=\left\lceil\frac{m}{c}\right\rceil$).
\end{Theorem}

\begin{proof} By Proposition \ref{6.rks}, for every fixed $i$ to determine all the graded Betti numbers $\beta_{i,j_i}(R/I_{,c\F}^{(m)})$ it suffices to add all $|\set(M)_{i-1}|$ where $M$ ranges in $\Gc$. Notice that if $|\set(M)|=j$, then $|\set(M)_{i-1}|=\binom{j}{i-1}$. 

 {\bf The first summand.} We prove that the first summand is the number of copies of $R$ in  $\beta_{i, j_i}(R/I_{c,\F}^{(m)})$ obtained from all elements $M\in \Gc$ whose partition $P(M)=[d_1,\ldots,d_t]$ has $d_t=d_{i_0}$ (via Proposition \ref{6.rks}).

First, by Lemma \ref{6.part2}, for every $[d_1,\ldots,d_t]\in \Pc$, either $d_{i_0}=d_t$ or $d_{i_0}<d_t$. For any $[\ul{d}]\in [\Pc]_t$ with $d_{i_0}=d_t$ and every $M\in \Gc$ with $P(M)=[\ul{d}]$ we have $|\set(M)|=c-d_{i_0}$ by Theorem \ref{6.delta}(2). For every such partition $[\ul{d}]$ we have 
{\small $$\binom{s}{s-c+d_1}\binom{s-c+d_1}{s-c+d_2} \cdots \binom{s-c+d_{(i_0-1)}}{s-c+d_{i_0}} \binom{s-c+d_{i_0}}{s-c+d_{(i_0+1)}} \cdots \binom{s-c+d_{(t-1)}}{s-c+d_{t}}$$} monomials $M\in \Gc$ with $P(M)=[\ul{d}]$. Since however $d_j=d_{i_0}$ for every $i_0\leq j \leq t$ by assumption, then this number equals $\binom{s}{s-c+d_1}\binom{s-c+d_1}{s-c+d_2} \cdots \binom{s-c+d_{(i_0-1)}}{s-c+d_{i_0}}$.

This number of monomials multiplied by $|\set(M)_{i-1}|=\binom{c-d_{i_0}}{i-1}$ is the first summand, which then is the contribution of these generators, in terms of copies of $R$, to $\beta_{i,j_i}(R/I_{c,\F}^{(m)})$ (via Proposition \ref{6.rks}).

{\bf The second summand.} We show that the second summand  is the number of copies of $R$ in $\beta_{i,j_i}(R/I_{c,\F}^{(m)})$ obtained from all elements $M\in \Gc$ whose partition $P(M)=[d_1,\ldots,d_t]$ has $d_t<d_{i_0}$ (via Proposition \ref{6.rks}), except, possibly, the maximum partition $[c,\ldots,c,r]$, which is counted separately in the last summand.

For each monomial $M\in \Gc$ with $P(M)=[d_1,\ldots,d_t]$ not maximal in $\Pc$, and with $d_t<d_{i_0}$ we have $|\set(M)|=c - d_{i_0} + j$ for some $0 \leq j \leq d_{i_0}-d_t$, by Theorem \ref{6.delta}(3), and

$$|\set(M)|=c-d_{i_0} + j \text{ if and only if }m_{\supp(M_{(d_{i_0})})}(M_{(d_t)})=s-c+d_t + j,$$
i.e. if and only if the smallest variable in the squarefree monomial $\supp(M_{(d_t)})$ of degree $s-c+d_t$ is the $(s-c+d_t + j)$-th smallest variable in $\supp(M_{(d_{i_0})})$. Thus, if we fix $M_{(d_1)},\ldots,M_{(d_{i_0})}$, then for every $0\leq j \leq d_{i_0}-d_t$ there are precisely $\binom{s-c+d_t+j-1}{j}$ possibilities for $M_{(d_t)}$, one for each squarefree monomial of degree $s-c+d_t$ whose smallest variable is the $(s-c+d_t + j)$-th smallest variable in $\supp(M_{(d_{i_0})})$.
Adding them together we obtain 
$$
\sum_{j=i-1}^{d_{i_0}-d_t} \binom{c-d_{i_0}+j}{i-1}\binom{s-c+d_t+j-1}{j}
$$
which is the contribution to $\beta_{i,j_i}(R/I_{c,\F}^{(m)})$ given by these monomials.

We then only need to count the possibilities we have for $M_{(d_1)},\ldots,M_{(d_{i_0})}$, which is easily seen to be  
$$\binom{s}{s-c+d_1}\binom{s-c+d_1}{s-c+d_2} \cdots \binom{s-c+d_{(i_0-1)}}{s-c+d_{i_0}}.$$ 
Then the contribution of any non-maximal partition with $d_t<d_{i_0}$ to 
$\beta_{i,j_i}(R/I_{c,\F}^{(m)})$ (via Proposition \ref{6.rks}) is precisely the second summand
$$\binom{s}{s-c+d_1}\binom{s-c+d_1}{s-c+d_2} \cdots \binom{s-c+d_{(i_0-1)}}{s-c+d_{i_0}} A_{i-1}([\ul{d}]).$$

{\bf The third summand.} if $t>\left\lceil\frac{m}{c}\right\rceil$, then by Proposition \ref{6.rks}, the sum of the first two summands in the statement gives the formula for the graded Betti numbers $\beta_{i,j_i}(R/I_{c,\F}^{(m)})$. If, however, $t=\left\lceil\frac{m}{c}\right\rceil$, we still have to count the contribution given by the generators $M$ where $P(M)=[c,\ldots,c,r]$ is the maximal partition. 
By Theorem \ref{6.delta}(1), for each of them we have $|\set(M)|=|\set(M_{(r)})|$. Since $M_{(r)}$ are the possible minimal generators of $I_{c-r+1,\F}$, then the graded Betti numbers obtained from the generators $M\in \Gc$ associated to the maximal partition are precisely the graded Betti numbers of $R/I_{c-r+1,\F}$, i.e. $\binom{s}{c-r+i-1}\binom{s-c+r-1+i-1}{i-1}$. This explains the last summand, and concludes the proof.

\end{proof}

\begin{Example}\label{7.extop}
Take $m=19$ and $c=6$. Then the top Koszul strand of $R/I_6^{(19)}$ is determined by the monomials of smallest degree, namely the ones associated to the partitions $[d_1,\ldots,d_t]\in \Po_{\leq 6}(19)$ of length $t=\left\lceil\frac{m}{c}\right\rceil=4$. These partitions are 
$$
[6,6,6,1],\;\;[6,6,5,2],\;\;[6,6,4,3],\;\;[6,5,5,3],\;\;[6,5,4,4],\;\;[5,5,5,4].
$$
Thus, there are $\binom{s}{5} + s\binom{s-1}{s-4} + \binom{s}{2}(s-2) + s\binom{s-1}{s-3} + s(s-1) + s(s-1)$ 
minimal generators of the smallest possible degree. Let $j_i:=\delta(t(s-c)+m+i-1)$ with $t$ as above.
Then the top Koszul strand is
$$
\begin{array}{ll}
\beta_{i,j_i}(R/I_6^{(19)}) = & \binom{s}{6-i}\binom{s-6+i-1}{i-1} +  s\sum\limits_{j=i-1}^3\binom{j+1}{i-1}\binom{s-5+j}{j} +  \binom{s}{2}\sum\limits_{j=i-1}^1\binom{j+2}{i-1}\binom{s-4+j}{j} \\
& \qquad +\,  s\sum\limits_{j=i-1}^2\binom{j+1}{i-1}\binom{s-4+j}{j} +  2s\sum\limits_{j=i-1}^2\binom{j+1}{i-1}\binom{s-3+j}{j}.
\end{array}
$$
\end{Example}

\begin{proof} We compute the contribution of each partition via Theorem \ref{7.formula}.

The contribution of the top partition to $\beta_i(R/I_6^{(19)})$ is the $i$-th Betti number of $R/I_6$, i.e. $\binom{s}{6-i}\binom{s-6+i-1}{i-1}$. All other partitions have $d_{i_0}<d_t$. For the second partition we have $i_0=3$, $d_{i_0}=5$ and $d_t=2$, thus its contribution to $\beta_i(R/I_6^{(19)})$ is 
$$
\binom{s}{s-1} A_{i-1}([6,6,5,2]) = s\sum_{j=i-1}^3\binom{j+1}{i-1}\binom{s-5+j}{j}.
$$
Analogously, for the third partition we have $i_0=3$, $d_{i_0}=4$ and $d_t=3$. Then, its contribution is $\binom{s}{s-2} A_{i-1}([6,6,4,3])$. For the fourth partition we have $i_0=2$, $d_{i_0}=5$ and $d_t=3$, thus  its contribution is $\binom{s}{s-1} A_{i-1}([6,5,5,3])$.
In the fifth partition we have $i_0=2$, $d_{i_0}=5$ and $d_t=4$. Its contribution is $\binom{s}{s-1} A_{i-1}([6,5,4,4])$. For the last partition we have $i_0=1$, $d_{i_0}=5$ and $d_t=4$, then its contribution is $\binom{s}{s-1} A_{i-1}([5,5,5,4])$. 

The formula is now obtained by adding these numbers and observing that $$A_{i-1}([5,5,5,4]) = A_{i-1}([6,5,4,4]) = \sum\limits_{j=i-1}^2\binom{j+1}{i-1}\binom{s-3+j}{j}.$$

\end{proof}

\subsection{Closed formulas in terms of $c,m,s$.} Theorem \ref{7.formula} gives the formula for all graded Betti numbers of $R/I_{c,\F}^{(m)}$. It is however easier to work with a closed formula solely written in terms of $c,s,m$, at least for some of the Koszul strands, or for small values of $c$ and $m$. This is the goal of this subsection. 

In Theorem \ref{7.strands} we provide such a closed formula for more than half of the Betti table of $R/I^{(m)}$, including the last strand with ``irregular" behavior. In Proposition \ref{7.top} we provide a closed formula for the top strand in several situation, and illustrate why a closed formula in general may be extremely complicated to obtain (assuming it exists). We employ these results to provide explicit examples, see Examples \ref{7.7} and \ref{7.last}. 

We can for instance quickly show that many generators $M$ have $|\set(M)|=c-1$.
\begin{Lemma}\label{7.suff}
Let $M\in \Gc$. If $P(M)=[d_1,\ldots,d_t]$ has $d_{t-1}=1$ then 
$|\set(M)|=c-1$.
\end{Lemma}

\begin{proof}
Let $h$ be the largest index such that $d_h>1$, notice that $h\leq t-2$ and $d_{h+1}=1$. Let $[\ul{b}]$ be the partition of length $t-1$, defined by setting $b_j=d_j$ for $1\leq j \leq h$, $b_{h+1}=2$ and $b_j=1$ for $h+2\leq j \leq t-1$. It is easily seen that $[\ul{b}] \in \Pc$ and $[\ul{b}]>_{\alex}[\ul{d}]$. In particular, this shows that $[\ul{d}]$ is not maximal in $\Pc$. 

Let $i_0:=i_0([\ul{d}])$ be the index of overlap (see Definition \ref{6.i0}). The existence of $[\ul{b}]$ yields that  $i_0\geq h+1$, thus $d_{i_0}=1=d_t$. The statement now follows by Theorem \ref{6.delta}(2).

\end{proof}

In Theorem \ref{4.Symb2} we proved that $I_{c,\F}^{(m)}$ is generated in degrees $\delta(t(s-c)+m)$ for any $\left\lceil\frac{m}{c} \right\rceil \leq t \leq m$. We also provided  a combinatorial way to determine the number of generators in each degree, i.e. to count the elements in $U_t = \{M\in \Gc\,\mid\, \deg(M)=\delta(t(s-c)+m)\}$ for any $\left\lceil\frac{m}{c} \right\rceil \leq t \leq m$. We now provide a closed formula for $|U_t|$ solely in terms of $c,s,m$, for more than half of the $U_t$'s.

\begin{Theorem}\label{7.binom}
Let $R,c,s,\F$ be as in Setting \ref{4.set}, and assume $\F$ consists of forms of the same degree $\delta$. Fix $m\geq 1$. 
\begin{enumerate}
\item For any $t> m/2$, the ideal $I_{c,\F}^{(m)}$ has  
$$ \binom{s}{c-1} \binom{m-t+c-2}{m-t}$$
minimal generators of degree $\delta(t(s-c)+m)$.

\item If $t=m/2$ and $m$ is even, then $I_{c,\F}^{(m)}$ has  
$$ \binom{s}{c-2} + \binom{s}{c-1}\left(\binom{c-2+t}{t} - (c-1)\right)$$
minimal generators of degree $\delta(t(s-c)+m)$.
\end{enumerate}
\end{Theorem}

 For part (2), the assumption that $m$ is even is only needed to ensure that $t=m/2$ is an integer. Also, interestingly, when $m$ is odd, the Koszul strand associated to $t=(m+1)/2$ starts with the ``expected" number of generators (i.e. the number of generators follows the same formula ruling the strands below it), but all other graded Betti numbers follow a different rule (Theorem \ref{7.strands}(2)). See for instance the second strands in both ideals of Example \ref{7.7}.

\begin{proof}
For any monomial $M$ on $\F$, we define the $\F$-degree of $M$ as $\deg_{\F}(M)=d$ if $M$ is the product of $d$ (not necessarily distinct) elements of $\F$. It is then easily seen that 
$$U_t = \{M\in \Gc\,\mid\, \deg_{\F}(M)=t(s-c)+m\}.$$
 Also, since $\F$ allows a unique monomial support, for any $\F'\subseteq \F$ and any $d\geq 1$, the set of all monomials of degree $d$ in $\F'$ has precisely $\binom{d+|\F'|-1}{|\F'|-1}$ elements. 
Set $I=I_{c,\F}$, and let $G(I)$ denote the set of all minimal monomial generators of $I$.
 
(1) For any $M\in \Gc$, let $P(M)=[d_1,\ldots,d_t]$ we have $d_t=1$, since otherwise $m=|P(M)|=\sum_{j=1}^t d_j \geq 2t >m$. Thus for any $M\in U_t$ the normal form of $M$ is $M=M_{(d_1)}\cdots M_{(d_{t-1})}M_{(1)}$. Let 
$$
U_{t,N}:=\{M\in U_t\,\mid\, N\text{ is the last term in the normal form of }M\}.
$$
It follows from the above that $U_t = \bigcup_{N\in G(I)} U_{t,N}$; additionally, the sets $U_{t,N}$ are all disjoint because if $N\neq N'\in G(I)$, are the respective last terms of monomials $M,M' \in U_t$, then $M$ and $M'$ have different normal forms, thus in particular $M\neq M'$.

It follows that $|U_t| = \sum_{N\in G(I)} | U_{t,N}|$. Since there are $\binom{s}{s-c+1}=\binom{s}{c-1}$ distinct elements in $G(I)$, it suffices to prove that for any $N\in G(I)$ we have $| U_{t,N}| = \binom{m-t+c-2}{m-t}$. 

Fix $N\in G(I)$; for any $M\in U_{t,N}$ let $M=M_{(d_1)}\cdots M_{(d_{t-1})}N$ be its normal form. Let $\F':=\supp(N)^C$ and observe that $|\F'|=c-1$ (because $|\supp(N)|=s-c+1$). 

By all the above, the statement follows if we show that the map
$$\Psi: U_{t,N} \lra \{\text{monomials in }\F' \text{ of }\F'\text{-degree }=m-t\}$$
defined as $\Psi(M)=M/N^t$ is a bijection. We first prove that $\Psi$ is well-defined. By definition of normal form $N$ divides $M_{(d_j)}$ for every $j$, thus $M/N^t$ is a monomial in $\F'=\supp(N)^C$. Since $\supp(M_{(d_j)})=s-c+d_j$ for every $1\leq j \leq t-1$, then $M/N^t$ has degree 
$$\deg_{\F}(M) - \deg_{\F}(N^t)  = t(s-c)+m -t(s-c+1)=m-t.$$

Injectivity is easily seen: if $\Psi(M)=\Psi(M')$, then $\frac{M}{N^t} = \frac{M'}{N^t}$, thus $M=M'$. Finally, let 
$Q$ be a monomial in $\F'$ with $\deg_{\F'}(Q)=m-t$, we need to show that $M:=QN^t \in U_{t,N}$. Clearly, $M$ has $\deg_{\F}(M)=(m-t)+t=m$. 
We write its normal form. Since $Q$ may not be in some $\G_{c',(m')}$, we cannot use Notation \ref{5.1}, so we must use the notation of Definition \ref{3.monsupp}; so let  $Q=Q^{(1)}\cdots Q^{(r)}U$ be its normal form, and since $\deg_{\F'}=m-t$, then $r\leq m-t<t$ (the latter inequality holds because $2t>m$). Therefore $t-r>0$, so it is easily seen that $M$ has normal form
$$
M = (Q^{(1)}N)(Q^{(2)}N)\cdots (Q^{(r)}N),\;\; \text{ with }t-r \text{ copies of }N,
$$
i.e. $M^{(j)}=(Q^{(j)}N)$ for $j=1,\ldots,r$ and $M^{(j)}=N$, for $r+1\leq j \leq t$.

Since $\supp(Q^{(j)}) \cap \supp(N)= \emptyset$ for every $j$, then $|\supp(M^{(j)})|=s-c+1+|\supp(Q^{(j)})$, thus 
$$
\Sdeg(M)=\left(\sum_{j=1}^r (|\supp(Q^{(j)})|+1)\right) + (t-r)= (m-t + r)+(t-r)=m.
$$
It follows that $M\in \Gc$.  Finally, by the above $\lambda(M)=t$ and the last term in the normal form of $M$ is $N$, therefore 
$M\in U_{t,N}$. This proves that $\Psi$ is a bijection, and then  $|U_t| = \binom{s}{c-1}\binom{(m-t)+(c-2)}{m-t}$, as desired.

(2) is proved similarly. The difference is that now there are $\binom{s}{c-2}$ monomials $M\in \Gc$ with $P(M)=[2,\ldots,2]$. Therefore, we only need to determine the number of monomials $M\in \Gc$ with $\lambda(M)=t=m/2$ and $d_t=1$ (because if $d_t\geq 2$, then $\Sdeg(M)=d_1+\ldots+d_t>m$, which is a contradiction). Similarly to the above, for any of the $\binom{s}{c-1}$ elements $N\in \G_{c,(1)}$ there is a bijection 
$$
\Psi': U_{t,N} \lra \{\text{monomials in }\F' \text{ of }\F'\text{-degree }=m/2 \text{ that are not pure powers}\}
$$
where, again, $\Psi'(M):=M/N^t$, and by a pure power we mean a monomial of the form $F^t$ for some $F\in \F'$ (we need to remove these monomials, because the only way to obtain them would be from the partition $[2,\ldots,2]$ which has already been counted separately).
Then
$$
|U_{t,N}| = \binom{(c-1)+ m/2 -1}{m/2}- (c-1) =\binom{c-2+t}{t} - (c-1).
$$
Thus we obtain $|U_t| = \binom{s}{c-1} |U_{t,N}| + \binom{s}{c-2}$, which yields the desired formula.

\end{proof}

If $c=1$ then $I_{c,\F}$ is principal and then also $I_{c,\F}^{(m)}$ is principal and generated by an element of degree $\delta ms$, so we may assume $c\geq 2$. 

By Corollary \ref{7.stranded} the ideal $I_{c,\F}^{(m)}$ has a Koszul stranded Betti table with precisely $m - \left \lceil \frac{m}{c}\right \rceil + 1$ Koszul strands. Indeed
$$
\beta_{i,j}(R/I_{c,\F}^{(m)}) \neq 0 \text{ only if }j \in \left\{\delta(t(s-c)+m + i - 1)\,\mid \left\lceil\frac{m}{c} \right\rceil \leq t \leq m\right\}.
$$
To determine the Betti table it suffices to understand the graded Betti numbers of these strands. 

In Theorem \ref{7.strands}(1) we prove that  a bit more than half of the Koszul strands (the bottom half, approximately) have very regular graded Betti numbers. 
Additionally, in Theorem \ref{7.strands}(2), we prove that this regularity result is sharp, in the sense that the last ``irregular strand" is the first one left out from the formula in (1) (the one corresponding to $t=\left\lceil\frac{m}{2}\right\rceil$), 
and we compute all its graded Betti numbers. 
 
\begin{Theorem}\label{7.strands}
Let $R,c,s,\F$ be as in Setting \ref{4.set} and assume the forms in $\F$ all have degree $\delta$. Fix $m>1$ and assume $c\geq 2$. 
\begin{enumerate}
\item For any fixed integer $1\leq i \leq c$, one has 
$$
\beta_{i,j}(R/I_{c,\F}^{(m)}) =\binom{c-1}{i-1}\binom{c-2 + m-t}{c-2}  \binom{s}{c-1},$$
if $j=\delta(t(s-c)+m+i-1)$ for some $\left\lceil \frac{m}{2} \right\rceil < t \leq m.$

\item Let $t:=\left\lceil\frac{m}{2} \right\rceil$; for any $1\leq i \leq c$, let $j_i:=\delta\left( t(s-c)+m+i-1)\right)$.
\begin{enumerate}
\item if $m$ is odd, then $
\beta_{i,j_i}(R/I_{c,\F}^{(m)}) = \binom{c-1}{i-1}\binom{c-2 + m-t}{c-2}  \binom{s}{c-1} - \binom{c-2}{i-2}\binom{s}{c-2}
$;
\item If $m=2$, then $\beta_{i,j_i}(R/I_{c,\F}^{(2)})=\binom{s}{c-1-i}\binom{s-c+i}{i-1}$;
\item If $m=4$, then $$
\begin{array}{ll}
\beta_{i,j_i}(R/I_{c,\F}^{(4)})  &= \binom{c-2}{i-1}\binom{s}{c-2} \\
& \qquad + \binom{s}{c-3}\left[ \binom{c-3}{i-1} + \binom{c-2}{i-1}(s-c+1) + \binom{c-1}{i-1}\binom{s-c+2}{2}\right];
\end{array}$$

\item If $m$ is even and $m>4$, then {\small$$
\begin{array}{ll}
\beta_{i,j_i}(R/I_{c,\F}^{(m)}) & = \binom{c-2}{i-1}\binom{s}{c-2}  + \binom{c-1}{i-1}\binom{s}{c-1}\left[\binom{c-2 + t}{t} - (c-1)\right]\\
& \qquad -  \binom{c-2}{i-2}\binom{s}{c-3}(s-c+3).
\end{array}
$$}
\end{enumerate}
\end{enumerate}

\end{Theorem}

%
%

\begin{proof}
(1) By the above, for every $i$ we only need to show the formula for the graded Betti numbers when $j_i=\delta(t(s-c)+m+i-1)$ for some $\left\lceil \frac{m}{2} \right\rceil < t \leq m$. We then fix such a $j:=j_i$.

 By Proposition \ref{6.rks}, the only elements $M\in \Gc$ contributing to 
the graded Betti number $\beta_{i,j}(R/I_{c,s}^{(m)})$ are the ones in $U_t=\{M\in \Gc\,\mid\, \deg(M)=\delta(t(s-c)+m)\}$. By Theorem \ref{7.binom}, there are $\binom{s}{c-1} \binom{m-t+c-2}{m-t}$ elements in $U_t$. For any element $M\in U_t$, let $P(M)=[d_1,\ldots,d_t]$. First, we observe that the last two entries of $[\ul{d}]$ are $1$. Indeed, if not, then $d_{t-1}\geq 2$, so 
$$m = \sum_{h=1}^t d_h = \sum_{h=1}^{t-1} d_h  + d_t \geq 2(t-1) + 1\geq m +1 >m,$$ yielding a contradiction. 
Now, by Lemma \ref{7.suff} we have $|\set(M)|=c-1$, then
$$
\beta_{i,j}(R/I_{c,s}^{(m)}) = \binom{c-1}{i-1} |U_t| = \binom{c-1}{i-1}\binom{c-2 + m-t}{c-2}  \binom{s}{c-1}.
$$

(2)(a) Recall that $|U_{t}|= \binom{s}{c-1}\binom{c-2+m-t}{c-2}$ (this follows by Theorem \ref{7.binom}). 
For any $M\in U_{t}$ we compute $|\set(M)|$. Let $P(M)=[d_1,\ldots,d_t]$. First, we observe that if $d_1>2$, then $d_{t-1}=1$. Indeed, assume by contradiction that $d_{t-1}\geq 2$, then
$$
m=\sum_{h=1}^t d_h \geq 3 + \sum_{h=2}^t d_h \geq 3 + \sum_{h=2}^{t-1} d_h + 1 \geq 3 + 2(t-2) + 1= 2t=m+1
$$
which is impossible. Therefore, $d_{t-1}=1$, and then $|\set(M)|=c-1$ by Lemma \ref{7.suff}.

On the other hand, if $m$ is odd and $d_1=2$, since $t=(m+1)/2$ it follows that $P(M)=[2,2,\ldots,2,1].$
By Lemma \ref{6.part2}(1), the index of overlap is $i_0(P(M))=1$. Since $d_{i_0}=d_1=2$ while $d_t=1$, then by Theorem \ref{6.delta}(3) one has 
$$
|\set(M)|=c-2 + m_{\supp(M_{(2)})}(M_{(1)}) - (s-c+1).
$$
Fix any of the $\binom{s}{s-c+2}=\binom{s}{c-2}$ minimal generators $M_{(2)}$ of $I_{c-1,\F}$, write $M_{(2)}=F_{j_1}\cdots F_{j_{s-c+2}}$ with $1\leq j_1<\cdots < j_{s-c+2}\leq s$. Since $M_{(1)}$ divides $M_{(2)}$ and $|\supp(M_{(1)})|=s-c+1$, then 
$$
m_{\supp(M_{(2)})}(M_{(1)}) = \left\{ \begin{array}{ll}
s-c+1, & \text{ if } M_{(1)} = F_{j_1}\cdots F_{j_{s-c+1}}\\
s-c+2, & \text{ otherwise }
\end{array}
\right.
$$
Therefore, there are precisely $\binom{s}{c-2}$ monomials $M$ with $P(M)=[2,\ldots,2,1]$ and $|\set(M)|=c-2$, while the other $\left(\binom{s}{c-1}\binom{c-2+m-t}{c-2}-\binom{s}{c-2}\right)$ monomials $M$ with $P(M)=[2,2,\ldots,2,1]$ have $|\set(M)|=c-1$. Therefore, by Proposition \ref{6.rks} we have 
$$
\beta_{i,j}(R/I_c^{(m)}) = \left(\binom{s}{c-1}\binom{c-2+m-t}{c-2}-\binom{s}{c-2}\right)\binom{c-1}{i-1} + \binom{s}{c-2}\binom{c-2}{i-1}
$$
Since $\binom{c-1}{i-1}-\binom{c-2}{i-1}=\binom{c-2}{i-2}$, the formula in the statement follows.

(2)(b) Since $m=2$ then the Koszul strand in the statement is the top strand in the Betti table of $R/I_{c,\F}^{(2)}$ (which has only 2 Koszul strands). Therefore, the statement follows by Proposition \ref{7.top}.

(2)(c)--(d) are fairly similar to (2)(a). We prove (d). Let $M\in \Gc$ and $P(M)=[d_1,\ldots,d_t]$. Similarly to the above, if $d_1\geq 4$ or $d_1=d_2=3$, then $d_{t-1}=1$, thus, by Lemma \ref{7.suff}, for all these monomials one has $|\set(M)|=c-1$. Assume next $d_1=3$ and $d_2<3$. Since $t=m/2$, then $P(M)=[3,2,\ldots,2,1]$. By Lemma \ref{6.part2} it follows that the index of overlap is $i_0(P(M))=2$. Since the second entry of $P(M)$ is strictly larger than the last entry, by Theorem \ref{6.delta}(3) we obtain as in (2)(a) that $|\set(M)|=c-2 + m_{\supp(M_{(2)})}(M_{(1)}) - (s-c+1).$ As above then, for any fixed choice of $M_{(3)}$ and $M_{(2)}$, there is precisely one $M_{(1)}$ giving $|\set(M)|=c-2$, while all the other ones have $|\set(M)|=c-1$. Therefore, there are $\binom{s}{s-c+3}\binom{s-c+3}{s-c+2}=\binom{s}{c-3}(s-c+3)$ monomials $M\in U_t$ with $P(M)=[3,2,\ldots,2,1]$ and $|\set(M)|=c-2$, while all other monomials with $P(M)=[3,2,\ldots,2,1]$ have $|\set(M)|=c-1$. 

Finally, we need to investigate the monomials $M\in \Gc$ whose associated partition is $P(M)=[2,\ldots,2]$. There are $\binom{s}{s-c+2}=\binom{s}{c-2}$ of them; for each of them, since all entries are equal, one has $d_{i_0}=d_t=2$, therefore $|\set(M)|=c-2$.

Summing up, we found that all $M\in U_t$ have $|\set(M)|=c-1$ except $\binom{s}{c-3}(s-c+3) + \binom{s}{c-2}$ of them, which have $|\set(M)|=c-2$. Combining these facts with the formula of Theorem \ref{7.binom}(2) for $|U_t|$ gives the formula in the statement. 

The case (c) is proved analogously, except that instead of having the partition $[3,2,\ldots,2,1]$ one has the partition $[3,1]$. 

\end{proof}

In general, a closed formula, only depending on $s,c,m$ of every Koszul strand of the Betti table of $R/I_{c,\F}^{(m)}$ appears unlikely to be obtained. In fact, even finding a general closed formula (solely in terms of $s,c,m$) for the very first strand appears very challenging, even in the monomial case (i.e. when $\F=\{y_1,\ldots,y_s\}$ are variables). See for instance Example \ref{7.extop}. 

 We have been able to determine a closed formula for the top strand when $1\leq m\leq c$, or the remainder of the division of $m$ by $c$ is at least $c-3$, see Proposition \ref{7.top} and Corollary \ref{7.c=4}. 

Our formulas show that small remainders tend to have the most complicated formulas, because of the many possible partitions associated to generators of smallest possible degree, and the many different possibilities for $|\set(M)|$, as one can see in Example \ref{7.extop}.

\begin{Proposition}\label{7.top}
Let $R,c,s,\F$ be as in Setting \ref{4.set} and assume the forms in $\F$ all have degree $\delta$. Fix $m>1$, let $c\geq 2$. 
The graded Betti numbers 
of the first non-zero Koszul strand in the Betti table of $R/I_{c,\F}^{(m)}$ are the following:

\begin{enumerate}
\item If $m\leq c$, then the top Koszul strand has length $m$, and for every $1\leq i \leq m$ its graded Betti numbers are:
$$
\beta_{i, \delta((s-c) + m + i-1)}(R/I_{c,\F}^{(m)}) 
= \binom{s}{c-m-i+1}\binom{s-c+m +i-2}{i-1} 
$$

\item Write $m=qc+r$ for an integer $0\leq r \leq c-1$.
\begin{enumerate}
\item If $r=0$, then 
$$\beta_{i, \delta(qs+ i-1)}(R/I_{c,\F}^{(m)}) = \left\{\begin{array}{rl}
1,&  \text{ if }\;i=1.\\
0,&  \text{ if }\;i>1.
\end{array}\right.
$$

\item If $r=c-1$, then 
$$\beta_{i, \delta((q+1)s + i-2)}(R/I_{c,\F}^{(m)}) = \left\{\begin{array}{rl}
 s,&  \text{ if }\;i=1.\\
s-1,&  \text{ if }\;i=2.\\
0,&  \text{ if }\;i>2.
\end{array}\right.
$$
\item  If $r=c-2$, then 
$$\beta_{i, \delta((q+1)s+i-3)}(R/I_{c,\F}^{(m)}) = \left\{\begin{array}{rl}
 \binom{s}{2} + s,&  \text{ if }\;i=1.\\
 s(s-1),&  \text{ if }\;i=2.\\
\binom{s-1}{2},&  \text{ if }\;i=3.\\
0,&  \text{ if }\;i>3.
\end{array}\right.
$$
\item  If $r=c-3$, then 
{\small$$\beta_{i, \delta(q(s+1)+ i-4)}(R/I_{c,\F}^{(m)}) = \left\{\begin{array}{rl}
 \binom{s}{3} + s(s-1) + s \cdot \max\{0,q-1\},&  \text{ if }\;i=1.\\
 \binom{s}{2}(s-3) + (2s^2-3s) + s \cdot \max\{0,q-1\} ,&  \text{ if }\;i=2.\\
s\binom{s-2}{2} + (s^2-2s),&  \text{ if }\;i=3.\\
\binom{s-1}{3},&  \text{ if }\;i=4.\\
0,&  \text{ if }\;i>4.
\end{array}\right.
$$}

\end{enumerate}
\end{enumerate}
\end{Proposition}

\begin{proof}
(1) Since $m\leq c$, then the first strands starts at $t=\left\lceil\frac{m}{c}\right\rceil=1$, thus we are looking at partitions of length 1 -- clearly the only possible one is $[m]$. Therefore, the monomials $M$ associated to it are squarefree with $|\supp(M)|=s-c+m$, i.e. they are precisely the minimal generators of the star configuration $I_{c-m+1,\F}$. By Proposition \ref{6.m=1}, the ideal $I_{c-m+1,\F}$ has $\delta$-c.i. quotients; moreover, $\beta_i(R/I_{c-m+1}) = \binom{s}{(c-m+1)-i}\binom{s-(c-m+1)+i-1}{i-1}$ (e.g. by \cite[Cor.~3.5]{PS} and the techniques of \cite{GHMN}). The formula now follows.

(2) (a) By assumption $m=qc$, thus the partition of smallest length of $m$ is $[c,c,\ldots,c]$, and clearly every other partition has larger length. Since there is only one monomial associated to this partition, i.e. $M=(F_1F_2\cdots F_s)^q$, the statement follows immediately.

For (b)--(d), by Proposition \ref{6.rks} we only need to compute $|\set(M)|$ for any minimal generator of smallest degree of $I_{c,\F}^{(m)}$. We need to find all partitions in $\Pc$ having the smallest possible length, namely $t=\left\lceil\frac{m}{c}\right\rceil=q+1$. Note that for this value of $t$, the top strand of the Betti table has graded Betti numbers $\beta_{i,\delta((q+1)s -(c-r) + i-1)}(R/I_{c,\F}^{(m)})$ -- we evaluate these numbers.

(b) Since $r=c-1$, then we need to compute $\beta_{i,\delta((q+1)s + i-2)}(R/I_{c,\F}^{(m)})$. To this end we need to look at the partitions in $\Pc$ having length $t=q+1$. The only partition is $[c,c,\ldots,c,c-1]$. The associated monomials have normal form $M=M_{(c)}^qM_{(c-1)}$, by Theorem \ref{6.delta}(1) we have $|\set(M)|=m(M_{(c-1)}) - s + 1$. This number is 1 for the $(s-1)$ monomials for which $F_s$ divides $M_{(c-1)}$ and $0$ for the remaining monomial $M_{(c-1)} = F_1\cdots F_{s-1}$. The formula now follows from Proposition \ref{6.rks}.

(c) Since $r=c-2$, we need to evaluate $\beta_{i,\delta((q+1)s + i-3)}(R/I_{c,\F}^{(m)})$. The only two partitions in $\Pc$ having length $t=q+1$ are $[c,\ldots,c,c-2]$ and $[c,\ldots,c,c-1,c-1]$ (notice that if $q=1$, then this is simply $[c-1,c-1]$). 
Since the first partition is maximal, then for any monomial $M$ associated to it we have $|\set(M)|=|\set(M_{(c-2)}|=m(M_{(c-2)}) - s + 2$. Since $M_{(c-2)}$ runs among all minimal generators of $I_{3,\F}$,  then the contributions of the monomials $M$ to the graded Betti numbers are precisely the graded Betti numbers of $I_{3}$, so the contribution is

 \begin{center}
$\binom{s}{2}$ copies of $R$ for $\beta_1$,\quad $s(s-2)$ copies of $R$ for $\beta_2$ \quad and $\binom{s-1}{2}$ copies of $R$ for $\beta_3$. 
\end{center}

Next, there are precisely $m$ generators associated to the partition $[c,c,\ldots,c,c-1,c-1]$; the index of overlap is $i_0=t-1$, thus $d_{i_0}=d_t=c-1$, so by Theorem \ref{6.delta}(2) we have $|\set(M)|=c-(c-1)=1$.  Therefore, these generators contribute with $s$ copies of $R$ for $\beta_1$ and $s$ copies of $R$ for $\beta_2$. Adding these numbers to the above ones provides the formula.

(d) Since $r=c-3$, then we need to compute $\beta_{i,\delta((q+1)s + i-4)}(R/I_{c,\F}^{(m)})$. We need to examine the partitions in $\Pc$ having length $t=q+1$. If $q=1$, then the shortest partitions in $\Pc$ are $[c,c-3]$, and $[c-1,c-2]$. If $q\geq 2$, then in addition to the partitions $[c,\ldots,c,c-3]$, and $[c,\ldots,c,c-1,c-2]$ we also also have the partition $[c,\ldots,c-1,c-1,c-1]$. This explains the presence of the term ``$\max\{0,q-1\}$" in the formula.

Now, analogously to part (c), the contribution of the maximal partition are the graded Betti numbers of $I_4$, thus 

 \begin{center}
$\binom{s}{3}$  copies of $R$ for $\beta_1$, \;\; $\binom{s}{2}(s-3)$ for $\beta_2$, \;\; $s\binom{s-2}{2}$ for $\beta_3$\;\; and $\binom{s-1}{3}$ for $\beta_4$. 
\end{center}

Next, the partition $[c,\ldots,c,c-1,c-2]$ has index of overlap $i_0=q=t-1$, and $d_{i_0}=c-1>c-2=d_t$. Therefore, by Theorem \ref{6.delta}(3) $|\set(M)|=1 + m_{\supp(M_{(c-1)})}(M_{(c-2)}) -s + 2$. There are $s$ choices of $M_{(c-1)}$. For each of them, $|\set(M)|=1$ for the only possible monomial $M_{(c-2)}$ with $m_{\supp(M_{(c-1)})}(M_{(c-2)})=1$, and $|\set(M)|=2$ for the other $s-1$ choices of $M_{(c-2)}$. Then, $|\set(M)|=1$ for $s$ of these monomials and $|\set(M)|=2$ for the other $\binom{s}{s-1}\binom{s-1}{s-2}-s=s(s-2)$ of them.
Therefore, the contribution of the second partition to the Betti table is 

 \begin{center}
$s(s-1)$ copies of $R$ for $\beta_1$,\quad  $2(s(s-1))-s= 2s^2-3s$ for $\beta_2$, \quad and $s(s-2)= s^2-2s$ for $\beta_3$.
\end{center}

This gives the formula if $q=1$. If $q\geq 2$, we also have the partition $[c,\ldots,c,c-1,c-1,c-1]$, which contributes with $\binom{s}{s-1}=s$ minimal generators $M$, for which $d_{i_0}=d_t=c-1$. Thus $|\set(M)|=c-1$ for each of them. Therefore, they contribute with $s$ copies of $R$ for both $\beta_1$ and $\beta_2$.

Adding the above numbers give the stated formula.

\end{proof}

\begin{Corollary}\label{7.c=4}
Let $R,c,s,\F$ be as in Setting \ref{4.set} and assume the forms in $\F$ all have degree $\delta$. If either $m\leq c$ or $c\leq 4$, then Proposition \ref{7.top} provides a closed formula (in terms of $c,s$ and $m$) for the top Koszul strand in the Betti table of $R/I_{c,\F}^{(m)}$. 
\end{Corollary}

Corollary \ref{7.c=4} and Theorem \ref{7.strands} provide closed formulas for the Betti table of $R/I_{c,\F}^{(m)}$ when $c$ and $m$ are relatively low, namely when $c\leq 4$ and $m\leq 7$. However, 
when $m=8,10,11$, there is only one more strand to compute, namely the next-to-the-top strand (corresponding to $U_t$ where $t=\left \lceil\frac{m}{c}\right\rceil +1$), which in these cases is pretty simple to compute. See for instance Example \ref{7.last} for the computations when $m=10$ and $c=4$. Therefore, we have the following:

\begin{Corollary}\label{7.small}
Let $R,c,s,\F$ be as in Setting \ref{4.set} and assume the forms in $\F$ all have degree $\delta$. 
Then our results provide closed formulas (in terms of $c,s$ and $m$) for the graded Betti numbers of $I_{2,\F}^{(m)}$ for any $m\geq 1$, and of $R/I_{c,\F}^{(m)}$ when $c\leq 4$ and $m\leq 11$.
\end{Corollary}

\begin{Example}\label{7.7}
Let $\F=\{F_1,\ldots,F_7\}$ be forms of the same degree $\delta$ such that any $5$ of them form a regular sequence. We define $I:=I_{3,\F}^{(7)}$ and $I':=I_{4,\F}^{(5)}$.

Then for any $\delta\geq 1$,  we have a closed formula for the Betti tables of $I$ and $I'$. 
In fact, by Theorem \ref{4.Symb2}, the Betti table of $I$ has precisely 5 Koszul strands. The last 4 Koszul strands of the Betti table are completely determined by the formulas of Theorem \ref{7.strands}; the top one is determined by the formula in Proposition \ref{7.top}(3).

Analogously, $I'$  has precisely 4 Koszul strands. The last 3 Koszul strands are determined by Theorem \ref{7.strands}; the top one is determined by Proposition \ref{7.top}(4).

Below are the Betti table of $I$ when $\delta =1$, and the Betti table of $I'$ when $\delta'=2$. 
{\small $$
\beta_I:=\begin{array}{r | c c c c}
\beta(S/I_{3,\F}^{(7)})    & 0 & 1 & 2 & 3 \\\hline
 0 & 1 &  &  &  \\
 \vdots &  & &  \\
 18 && 28 & 42 & 15 \\
 19 && 	&  & \\
 20 &&  	& 	&  \\
  21 &&  	& 	& \\
  22 && 84 & 161 & 77 \\
 23 && 	& & \\
 24 &&  	& 	& \\
 25 &&  	& 	& \\
 26 && 63 & 126 & 63 \\
 27 && 	&  & \\
 28 &&  	& 	&  \\
 29 &&  	& 	& \\
 30 && 42 & 84 & 42 \\
 31 && 	&  & \\
 32 &&  	& 	& \\
 33 && 	&  & \\
 34 &&  21	& 42	&21 \\

\end{array}
\qquad\beta_{I'}:=\begin{array}{r | c c c c c}
\beta(S/I_{4,\F'}^{(5)})    & 0 & 1 & 2 & 3 & 4\\\hline
 0 & 1 &  &  & &   \\
 \vdots &  & & & &   \\
21 &&77&&&\\
22 &&&161&&\\
23 &&&&105&\\
24 &&&&&20\\
25 &&&&&\\
26 &&&&&\\
27 &&210&&&\\
28 &&&609&&\\
29 &&&&588&\\
30 &&&&&189\\
31 &&&&&\\
32 &&&&&\\
33 &&105&&&\\
34 &&&315&&\\
35 &&&&315&\\
36 &&&&&105\\
37 &&&&&\\
38 &&&&&\\
39 &&35&&&\\
40 &&&105&&\\
41 &&&&105&\\
42 &&&&&35\\

\end{array}
$$}

\end{Example}

\begin{Example}\label{7.last}
Let $\F=\{F_1,\ldots,F_s\}$ be $s\geq 5$ forms of the same degree $\delta$ and assume any 5 of them form a regular sequence. Let $c=4$ and $m=10$.

By Theorem \ref{4.Symb2} and Corollary \ref{7.stranded} the Betti table of $R/I_{4,\F}^{(10)}$ consists of 8 Koszul strands, one for each $U_t$ with $3\leq t \leq 10$ (recall that $U_t$ is defined in Corollary \ref{6.surj}(2)). 

Our previous results provide closed formulas for 7 of these 8 strands:
\begin{itemize}
\item The top strand, $t=3$, has entries $\binom{s}{2} + s$, $s(s-1)$, $\binom{s-1}{2}$, $0,\ldots,0$, by Proposition \ref{7.top}(2)(c);
\item The strand where $t=5$ has entries $a_1,a_2,a_3,a_4$ where, after simplifying, $a_1=\binom{s}{2} + 18\binom{s}{3}$, $a_2=54\binom{s}{3}$, $a_3=54\binom{s}{3} - 3\binom{s}{2}$ and $a_4=18\binom{s}{3} - s(s-1)$, by Theorem \ref{7.strands}(2);
\item The $t$-th strand with $6\leq t \leq 10$ has entries $a,3a,3a,a$, where $a:=\binom{10-t+2}{2}\binom{s}{3}$, by Theorem \ref{7.strands}(1).
\end{itemize}
Finally, the strand where $t=4$ has entries $a_1',a_2'a,_3',a_4'$, where 
$$\begin{array}{ll}
a_i'  & = \binom{2}{i-1} \left[ \binom{s}{2} + 2s(s-1) + s(s-3)\right] \\
 & \qquad +\, \binom{3}{i-1}\left[\binom{s}{3} + s(s-1)(s-3) + s\binom{s-2}{2}\right] + s\binom{1}{i-1}.
 \end{array}
$$
Thus, for instance, if $s=7$ the Koszul strands in the Betti table of $R/I_{4,\F}^{(10}$ are
$$
\begin{array}{r | c c c c}
   & 1 & 2 & 3  & 4 \\\hline
t=3 & 28 & 42 & 15 & 0 \\
t=4 & 413 & 1092 & 952 & 273 \\
t=5 & 651 & 1890 & 1827 & 588 \\
t=6 & 525 & 1575 & 1575 & 525 \\
 t=7 & 350 & 1050 & 1050 & 350 \\
 t=8 & 210 & 630 & 630 & 210 \\
 t=9 & 105 & 315 & 315 & 105 \\
 t=10 & 35 & 105 & 105 & 35 \\
\end{array}
$$
\end{Example}

\begin{proof}
By Theorem \ref{7.strands} and Proposition \ref{7.top}, we only need to justify the computations for the strand where $t=4$. The partitions of length 4 of $m=10$ are 
$$
[4,4,1,1],\;\;[4,3,2,1],\;\;[4,2,2,2],\;\;[3,3,3,1],\;\;[3,3,2,2].
$$
Since $[4,4,1,1]$ has only two jumps and $d_t=d_{t-1}$, then  $i_0([4,4,1,1])=3$ by Lemma \ref{6.part2}, thus $d_{i_0}=d_t=1$ for this partition. For each of the $\binom{s}{3}$ monomials associated to this partition we then have $|\set(M)|=c-1=3$ by Theorem \ref{6.delta}(2).  Analogously, for each of the $\binom{s}{2}$ monomials associated to $[4,2,2,2]$ and for each of the $\binom{s}{1}\binom{s-1}{s-2}=s(s-1)$ monomials associated to $[3,3,2,2]$ we have $|\set(M)|=c-2=2$, by Theorem \ref{6.delta}(2).

By Lemma \ref{6.part2} we have $i_0([4,3,2,1]) = 3$. By Theorem \ref{6.delta}(3), among all the monomials associated to this partition, there are $s(s-1)$ having $|\set(M)|=c-2=2$ and $s(s-1)(s-2) - s(s-1)=s(s-1)(s-3)$ having $|\set(M)|=c-1=3$.

Finally, $i_0([3,3,3,1])=1$, and by Theorem \ref{6.delta}(3), among all monomials associated to this partition, there are $s$ of them having $|\set(M)|=c-3=1$, there are $s(s-3)$ of them having $|\set(M)|=c-2=2$ and $s\binom{s-2}{2}$ having $|\set(M)|=c-1$.  The computation of the Koszul strand now follows by Proposition \ref{6.rks}.

\end{proof}

\end{document}